\documentclass[12pt]{amsart}
\usepackage[letterpaper,margin=1.1in]{geometry}
\usepackage{graphicx}
\usepackage{amssymb}
\usepackage[nobysame,alphabetic]{amsrefs}
\usepackage{enumerate}

\newtheorem{theorem}{Theorem}[section]
\newtheorem{lemma}[theorem]{Lemma}
\newtheorem{corollary}[theorem]{Corollary}

\theoremstyle{definition}
\newtheorem{definition}[theorem]{Definition}

\theoremstyle{remark}
\newtheorem{remark}[theorem]{Remark}
\newtheorem{example}[theorem]{Example}
\newtheorem{problem}[theorem]{Problem}

\newtheorem*{claim}{Claim}
\newtheorem*{ack}{Acknowledgements}

\newtheoremstyle{defnbreak}
  {}
  {}
  {}
  {}
  {\bfseries}
  {.}
  {\parindent}
  {}

\theoremstyle{defnbreak}
\newtheorem*{mdefinition}{Metadefinition}

\let\inf\relax \DeclareMathOperator*\inf{\vphantom{p}inf}

\newcommand{\cC}{\mathcal{C}}

\newcommand{\cF}{\mathcal{F}}
\newcommand{\cG}{\mathcal{G}}
\newcommand{\ccG}{{\overline{\mathcal{G}}}}
\newcommand{\cGp}{{\mathcal{G}^\perp}}
\newcommand{\cH}{\mathcal{H}}

\newcommand{\cL}{\mathcal{L}}
\newcommand{\cM}{\mathcal{M}}

\newcommand{\cP}{\mathcal{P}}
\newcommand{\cQ}{\mathcal{Q}}

\newcommand{\cS}{\mathcal{S}}
\newcommand{\cT}{\mathcal{T}}
\newcommand{\cTp}{{\mathcal{T}^\perp}}
\newcommand{\cQp}{{\mathcal{Q}^\perp}}
\newcommand{\cU}{\mathcal{U}}

\newcommand{\cY}{\mathcal{Y}}

\newcommand{\RR}{\mathbb{R}}
\newcommand{\ccS}{\overline{\mathcal{S}}}
\newcommand{\ccT}{{\overline{\mathcal{T}}}}
\newcommand{\cSccT}{{\cS\setminus\ccT}}
\newcommand{\cPcQ}{{\cP\setminus\cQ}}
\newcommand{\ball}{B}

\newcommand{\dist}{\mathop\mathrm{dist}\nolimits}
\newcommand{\spt}{\mathop\mathrm{spt}\nolimits}

\newcommand{\diam}{\mathop\mathrm{diam}\nolimits}
\newcommand{\Tan}{\mathop\mathrm{Tan}\nolimits}
\newcommand{\PsTan}{\mathop{\Psi\mbox{-}\mathrm{Tan}}\nolimits}
\newcommand{\bPsTan}{\mathop{\mathrm{b}\Psi\mbox{-}\mathrm{Tan}}\nolimits}
\newcommand{\uPsTan}{\mathop{\mathrm{u}\Psi\mbox{-}\mathrm{Tan}}\nolimits}
\newcommand{\D}[2]{\mathop{\mathrm{D}^{#1,#2}}\nolimits}
\newcommand{\ue}{\mathop{\mathrm{ex}^{}}\nolimits}

\newcommand{\mD}[2]{\mathop{\widetilde{\mathrm{D}}^{#1,#2}}\nolimits}
\newcommand{\mud}[2]{\mathop{\widetilde{\mathrm{d}}^{\,#1, #2}}\nolimits}
\newcommand{\covplain}{\mathop\mathrm{N}\nolimits}
\newcommand{\cov}[2]{\mathop{{\mathrm{N}}^{#1, #2}}\nolimits}
\newcommand{\udim}{\mathop{\overline{\mathrm{dim}}}\nolimits}

\newcommand{\cl}[1]{{\overline{#1}}}
\newcommand{\sing}[1]{\mathop\mathrm{sing}_{#1}\nolimits}

\newcommand{\CL}[1]{\mathfrak{C}(#1)}

\numberwithin{equation}{section}
\numberwithin{figure}{section}

\begin{document}

\title[Local Set Approximation]{Local set approximation: Mattila-Vuorinen type sets, Reifenberg type sets, and tangent sets}
\author{Matthew Badger}
\author{Stephen Lewis}
\thanks{M.~Badger was partially supported by an NSF postdoctoral fellowship, NSF DMS 1203497. S.~Lewis was partially supported by an RTG fellowship, NSF DMS 0838212.}
\date{September 27, 2014}
\address{Department of Mathematics\\ University of Connecticut\\ Storrs, CT 06269-3009}
\email{matthew.badger@uconn.edu}
\address{School of Mathematics\\ University of Minnesota\\ Minneapolis, MN 55455}

\begin{abstract} We investigate the interplay between the local and asymptotic geometry of a set $A\subseteq\RR^n$ and the geometry of model sets $\cS\subset\mathcal{P}(\RR^n)$, which approximate $A$ locally uniformly on small scales. The framework for local set approximation developed in this paper unifies and extends ideas of Jones, Mattila and Vuorinen, Reifenberg, and Preiss. We indicate several applications of this framework to variational problems that arise in geometric measure theory and partial differential equations. For instance, we show that the singular part of the support of an $(n-1)$-dimensional asymptotically optimally doubling measure in $\RR^n$ ($n\geq 4$) has upper Minkowski dimension at most $n-4$.
\end{abstract}

\maketitle

\setcounter{tocdepth}{1}
\tableofcontents

\section{Introduction}\label{s:intro}

In this article, we investigate the structure and size of sets $A\subseteq\RR^n$ that admit uniform local approximations by a class of model sets $\cS$. The sets $A$ that we consider have one of the following forms.
\begin{mdefinition} \ 
\begin{itemize}
\item $A$ is \emph{locally $\varepsilon$-approximable by $\cS$} if for every compact set $K\subseteq A$ there is some initial scale $r_K>0$ such that for all $x\in K$ and $0<r\leq r_K$ there is a model set $S_{x,r}\in \cS$ such that $A$ is $\varepsilon$-close to $S_{x,r}$ near $x$ at scale $r$.
\item $A$ is \emph{locally well approximated by $\cS$} if $A$ is  locally $\varepsilon$-approximable by $\cS$ for all $\varepsilon>0$.
\end{itemize}
\end{mdefinition}

For example, when $\cS=\cG(n,m)$ is the Grassmannian of $m$-dimensional subspaces of $\RR^n$, an embedded $C^1$ submanifold $M^m\subseteq \RR^n$ is locally well approximated by $\cS$.
However, since no stability conditions are imposed on the approximating sets $S_{x,r}$ in the metadefinition, sets that are locally well approximated by $\cG(n,m)$ in general may not admit classical tangent planes or may even have locally infinite $m$-dimensional Hausdorff measure.

Different meanings may be attached to the phrase ``$A$ is $\varepsilon$-close to $S_{x,r}$ near $x$ at scale $r$" appearing in the metadefinition, resulting in different models of local set approximation.  The principal distinction between models of local set approximation that have appeared in the literature is the directionality of approximation; that is, the symmetry or asymmetry of approximation measurements. On one hand, if  distance between an approximated set $A$ and an approximating set $S_{x,r}$ is measured by how close $A$ is to $S_{x,r}$ \emph{and} by how close $S_{x,r}$ is to $A$, then the approximation is \emph{bilateral}. On the other hand, if distance between $A$ and its approximant $S_{x,r}$ is measured only by how close $A$ is to $S_{x,r}$, then the approximation is \emph{unilateral}. The decision to use a bilateral or unilateral approximation model should depend on the application of the model.

\begin{example}[$\cS=G(n,m)$, bilateral approximation: Reifenberg flat sets] The prototypical example \label{Rex} of local set approximation is due to Reifenberg \cite{Reifenberg}, who considered sets that admit uniform local bilateral approximations by $m$-dimensional planes in $\RR^n$ ($1\leq m\leq n-1$) to study regularity of solutions of the Plateau problem in arbitrary codimension $n-m$. Following \cite{KT97}, these sets are now called Reifenberg flat sets: \begin{quotation} A set $A\subseteq\RR^n$ is \emph{$(\varepsilon,r_0)$ Reifenberg flat} if for all $x\in A$ and $0<r\leq r_0$ there exists a plane $P_{x,r}\in \cG(n,m)$ such that $\dist(a,x+P_{x,r})\leq \varepsilon r$ for all $a\in A\cap \ball(x,r)$ and $\dist(p,A)\leq \varepsilon r$ for all $p\in (x+P_{x,r})\cap \ball(x,r)$. \end{quotation} The main result on Reifenberg flat sets in \cite{Reifenberg} is \emph{Reifenberg's Topological Disk Theorem}: for all $1\leq m\leq n-1$ there exists $\delta=\delta(n,m)>0$ such that if $A\subseteq\RR^n$ is closed and $(\delta, r_0)$ Reifenberg flat for some $r_0>0$, then $A$ is locally homeomorphic to a ball in $\RR^m$. For a detailed formulation of Reifenberg's Topological Disk Theorem, see \cite{DT}.
\end{example}

\begin{example}[$\cS=G(n,m)$, unilateral approximation: linear approximation property]  Mattila and Vuorinen \cite{MV} introduced \label{MVex} a unilateral local approximation scheme in the context of obtaining dimension bounds on quasiconformal images of spheres: \begin{quotation} A set $A\subseteq\RR^n$ has the \emph{$m$-dimensional $(\varepsilon,r_0)$ linear approximation property} if for all $x\in A$ and $0<r\leq r_0$ there exists a plane $P_{x,r}\in \cG(n,m)$ such that $\dist(a,x+P_{x,r})\leq \varepsilon r$ for all $a\in A\cap \ball(x,r)$.\end{quotation}
Mattila and Vuorinen proved that for all $1\leq m\leq n-1$ there exists $C=C(n,m)>1$ such that if $A\subseteq\RR^n$ has the $m$-dimensional $(\varepsilon,r_0)$ linear approximation property for some $r_0>0$, then the upper Minkowski dimension of $A$ is at most $m + C\varepsilon^2$. Simple examples (variations on the Von Koch snowflake) show that the dependence on $\varepsilon$ in the dimension bound is sharp.\end{example}

Additional approximation schemes of Reifenberg type have appeared in several contexts. The common feature of these models is that local errors are measured bilaterally in terms of the mutual distance between the approximated set $A$ and the approximating sets $S_{x,r}$.

\begin{example}[$\cS=$ $2$-dimensional Almgren minimal cones in $\RR^3$, bilateral approximation] Let $\cM$ \label{Mex} denote the collection of 2-dimensional Almgren minimal cones in $\RR^3$, which by the classification of Taylor \cite{Taylor} have one of three fundamental types (see Figure \ref{figure:minimalCones1}).
The first type of minimal cones are planes. The second type of minimal cones are $Y$-type sets, which are unions of three half planes whose boundaries meet at $120^\circ$ along a line. The third and final type of minimal cones are $T$-type sets, which can be described as six-sheeted cones over the spine of a regular tetrahedron.
\begin{figure}[htb]
  \begin{center}\includegraphics[width=.7\textwidth]{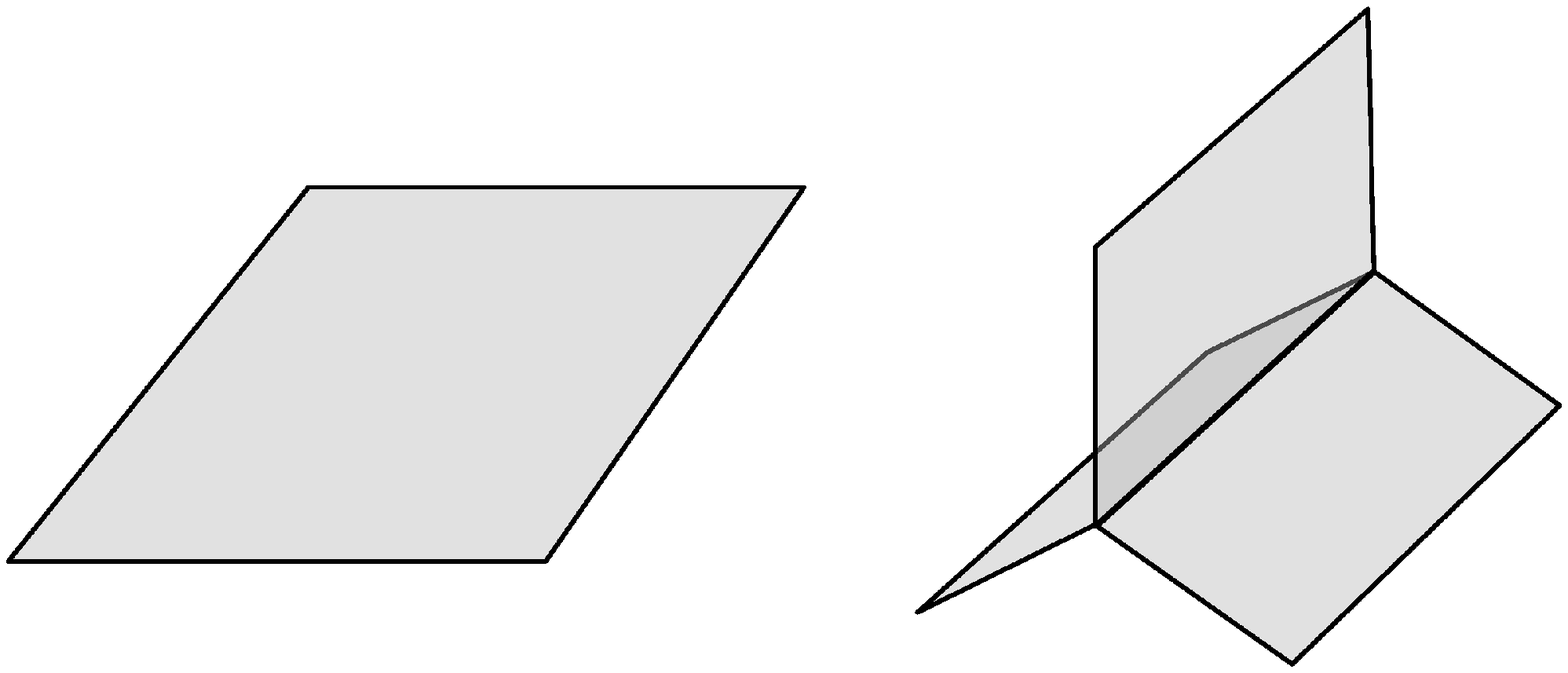}\end{center}
  \caption{The class $\cM$ consists of planes, $Y$-type sets, and $T$-type sets.}
\label{figure:minimalCones1}
\end{figure}

In \cite{DDT} David, De Pauw and Toro generalize Reifenberg's topological disk theorem to sets that are locally approximable by $\cM$. That is, they prove that there exists $\delta>0$ such that if $A\subseteq\RR^3$ is closed and locally $\delta$-approximable in the sense of Example \ref{Rex}, but with approximating planes $P_{x,r}\in\cG(n,m)$ replaced by approximating minimal cones $M_{x,r}\in\cM$, then $A$ is locally homeomorphic to an open subset of a plane, a $Y$-type set, or a $T$-type set. Moreover, they show that $A$ decomposes into three sets $A=A_1\cup A_2\cup A_3$, where $A_1$ is locally Reifenberg flat, $A$ looks like a $Y$-type set whose spine contains $x$ near each $x\in A_2$, and $A$ looks like a $T$-type set with vertex at $x$ near each $x\in A_3$.
\end{example}

\begin{example}[$\cS=$ graphs of Lipschitz functions $f:\RR^{n-1}\rightarrow \RR$, bilateral approximation] Let $\cL=\cL(n,N)$ denote the collection of all (rotations of) graphs of Lipschitz functions $f:\RR^{n-1}\rightarrow\RR$ such that $f(0)=0$ and $f$ has Lipschitz constant at most $N$. In \cite{LN}, J.~ L.~ Lewis and Nystr\"om investigated the boundary behavior of solutions to the $p$-Laplace equation (a canonical nonlinear degenerate elliptic equation) in certain rough domains whose boundaries admit local uniform approximations by $\cL$ in a Hausdorff distance sense. This class includes Lipschitz domains and domains whose boundaries are Reifenberg flat. \end{example}

\begin{example}[$\cS=$ zero sets of  harmonic polynomials in $\RR^n$, bilateral approximation] Let $\cH=\cH(n,d)$ denote \label{Bex} the class of all zero sets of nonconstant harmonic polynomials $p:\RR^n\rightarrow\RR$ of degree at most $d$ such that $p(0)=0$. In \cite{KT06}, Kenig and Toro showed that solutions for a certain two-phase free boundary problem for harmonic measures on two-sided domains in $\RR^n$ are locally well approximated by $\cH$ in a Hausdorff distance sense. Refined information about the structure and size of the free boundary was obtained by Badger \cites{Badger1, Badger3} by studying the geometry of abstract sets approximated by $\cH$ and measures on their support. For instance, in \cite{Badger3} Badger showed that if $A\subseteq\RR^n$ is closed, $A$ is locally well approximated by $\cH$, and $A$ looks pointwise on small scales like the zero set of a homogeneous harmonic polynomial, then $A=A_1\cup A_2$, where $A_1$ is locally well approximated by $\cG(n,n-1)$, while near each $x\in A_2$ the set $A$ looks locally like the zero set of a homogeneous harmonic polynomial of degree at least 2.
 \end{example}

\begin{example}[$\cS=$ supports of $m$-uniform measures in $\RR^n$, bilateral approximation] A Borel \label{KPex} measure $\mu$ on $\RR^n$ is \emph{$m$-uniform} if $\mu(\ball(x,r))= c\,r^m$ for all $x$ in the support of $\mu$ and for all $r>0$. A Borel measure $\mu$ on $\RR^n$ is \emph{$m$-asymptotically optimally doubling} if $$R(\mu,K,r):=\sup_{x\in K}\sup_{1/2\leq \tau\leq 1}\left|\frac{\mu(\ball(x,\tau r))}{\mu(\ball(x,r))} -\tau^m\right|\rightarrow 0\quad\text{as }r\downarrow 0$$ for every compact set $K$ in the support of $\mu$. Let $\cU=\cU(n,m)$ denote the collection of supports of $m$-uniform measures in $\RR^n$. The support of an $m$-asymptotically optimally doubling measure in $\RR^n$ is locally well approximated by $\cU$ (see \cite{Lewis}*{Theorem 3.8}).

In \cite{KP}, Kowalski and Preiss showed that the support of an $(n-1)$-uniform measure in $\RR^n$ is either an $(n-1)$-dimensional plane or the  light cone $C=\{x\in\RR^n: x_1^2+x_2^2+x_3^2=x_4^2\}$ (in some orthonormal coordinates). Uniform measures in codimension $n-m\geq 2$ have resisted a complete classification, but partial descriptions of them have been given by Preiss \cite{Preiss}, Kirchheim and Preiss \cite{KiP}, and Tolsa \cite{T-uniform}. The structure of $m$-asymptotically optimally doubling measures $\mu$ whose doubling characteristic decays locally uniformly at a H\"older rate (i.e.~ $R(\mu,K,r)\leq C_Kr^\alpha$ for all $0<r\leq r_K$) was studied by David, Kenig, and Toro \cite{DKT}, Preiss, Tolsa, and Toro \cite{PTT}, and S. Lewis \cite{Lewis}. In \cite{DKT} and \cite{PTT}, it was proved that the support of $\mu$ is an $m$-dimensional $C^{1,\beta}$ submanifold of $\RR^n$ away from a closed set of zero $m$-dimensional Hausdorff measure. Furthermore, Preiss, Tolsa, and Toro prove that when $m=n-1$ the support of $\mu$ is an $m$-dimensional $C^{1,\beta}$ submanifold away from a closed set of Hausdorff dimension at most $n-4$. Recently, in \cite{Lewis}, Lewis has proved that when $n=4$ and $m=3$ the support of $\mu$ admits local $C^{1,\beta}$ parameterizations at every point \emph{either} by open subsets of a 3-dimensional plane \emph{or} by open subsets of the light cone $C$.
\end{example}

Unilateral local approximation of sets by $\cG(n,m)$ was introduced independently by Jones \cite{JonesTSP} and is now an important tool in the theory of quantitative rectifiability (e.g. see \cite{DS93}).

\begin{example} In \cite{JonesTSP}, Jones introduced the idea of a unilateral approximation number, now called a \emph{Jones beta number} $\beta_E(Q)$, which measures how closely the set $E\subseteq\RR^n$ is to lying on a straight line inside a cube $Q$ in a scale-invariant fashion. Specifically, $$\beta_E(Q)=(\diam Q)^{-1}\inf_{\ell} \sup_{x\in E\cap Q} \dist(x,\ell),$$ where the infimum ranges over all straight lines (i.e.~ 1-dimensional affine subspaces) in $\RR^n$. Jones (when $n=2$) and Okikiolu \cite{O-TSP} (when $n\geq 3$) proved that a bounded set $E\subseteq\RR^n$ is a subset of a rectifiable curve in $\RR^n$  if and only if $\beta(E)^2:=\sum_{Q} \beta_E(3Q)^2 \diam Q<\infty$, where the sum ranges over all dyadic cubes $Q\subset\RR^n$; moreover, the length of the shortest curve containing $E$ is comparable (up to dimensional constants) to $\diam E+\beta(E)^2$.  Bilateral variants of Jones beta numbers were later introduced in \cite{DS93} and \cite{Toro}. \end{example}

The goal of this essay is to initiate the study of Reifenberg and Jones-Mattila-Vuorinen models of local set approximation in fuller generality than has been previously done. We develop a framework for describing bilateral and unilateral approximation of a set $A\subseteq\RR^n$ by a general class $\cS$ of closed sets in $\RR^n$. Then we direct our attention to two questions. How does the geometry of model sets transfer to the geometry of an approximated set? And, how does a set look like if it is approximated by several types of model sets? For instance, with respect to the first question, we generalize Mattila and Vuorinen's dimension bound from Example \ref{MVex} to sets that are unilaterally approximated by model sets $\cS$ that possess a uniform ``covering profile" (see \S8). And, with respect to the second question, we establish decomposition theorems for Reifenberg type sets that encompass the decompositions in Examples \ref{Mex} and \ref{Bex} as special cases (see \S6). Some new applications and examples illustrating these general results are given in \S9. For example, we show that the singular part of the support of an $(n-1)$-dimensional asymptotically optimally doubling measure in $\RR^n$ (see Example \ref{KPex}) has upper Minkowski dimension at most $n-4$ for all $n\geq 4$. At the end of \S9, we discuss a few open problems and directions for future research.

A central tool in our analysis that we wish to highlight is the notion of a tangent set of a closed set in $\RR^n$, which is modeled on Preiss' notion of a tangent measure of a Radon measure on $\RR^n$ from \cite{Preiss}. Roughly speaking, a tangent object (set or measure) is the limiting object obtained by zooming in on a point of the object along a sequence of shrinking scales going to zero. A precise definition of tangent sets appears in \S3. It turns out that many properties of tangent measures also hold for tangent sets. In particular, in \S5 we show that the collection $\Tan(A,x)$ of tangent sets of a closed set $A$ at $x\in A$ is connected in a certain strong sense: if $\Tan(A,x)\subseteq\cS$ and $\cT$ is ``separated at infinity" in $\cS$, then either $\Tan(A,x)\subseteq \ccT$ or $\Tan(A,x)\subseteq \cS\setminus\ccT$. See \S5 for a precise formulation.  This feature of tangent sets is used to analyze the structure of Reifenberg type sets in \S6. Tangent sets and pseudotangents sets (a related notion obtained by zooming in on a set along a convergent sequence of points in the set) are also used to provide characterizations of Reifenberg type sets in \S4 and characterizations of Mattila-Vuorinen type sets in \S7.

\begin{remark}\label{r:geo1} In this paper, although we only develop the theory of local set approximation in the Euclidean spaces $\RR^n$ ($n\geq 1$), we do not overly rely on the affine structure of $\RR^n$. Rather the limited geometric features of $\RR^n$ that we use are that $\RR^n$ is a \emph{proper} metric space (i.e.~ closed balls are compact), $\RR^n$ is \emph{translation invariant} (i.e.~ for all $z,w\in\RR^n$ there exists an isometry $\tau_{z,w}:\RR^n\rightarrow\RR^n$ such that $\tau_{z,w}(z)= w$), and $\RR^n$ is \emph{dilation invariant} (i.e.~ for all $r>0$ there exists a similarity $\delta_r:\RR^n\rightarrow\RR^n$ such that $\dist(\delta_r(x),\delta_r(y))=r\dist(x,y)$ for all $x,y\in\RR^n$). Because these geometric properties are shared by finite-dimensional Banach spaces and  (sub-Finsler) Carnot groups (e.g., see Le Donne \cite{LeDonne-MetricCarnot}), it is the authors' belief that the theory of local set approximation developed below should transfer to these spaces upon making the appropriate notational changes.\end{remark}

\begin{ack}  A portion of this research was completed while both authors visited the Institute for Pure and Applied Mathematics (IPAM), during the long program on Interactions between Analysis and Geometry in Spring 2013. Additional research was carried out during reciprocal visits to the math departments at Stony Brook University and the University of Washington. Some of the results in this article first appeared in the second author's PhD thesis at the University of Washington.\end{ack}

\section{Distances between sets and convergence of closed sets}
\label{s:dist}

For any $A\subseteq\RR^n$, let $\cl A$ denote the closure of $A$ in $\RR^n$. For all $x\in\RR^n$ and $r>0$, let $\ball(x,r)=\{y\in\RR^n:|x-y|\leq r\}$ denote the \emph{closed} ball in $\RR^n$ with center $x$ and radius $r$. Let $\CL{x}$ denote the collection of all closed subsets of $\RR^n$ that contain $x$. Also let $\CL{A}=\bigcup_{x\in A}\CL{x}$ denote the collection of all closed subsets of $\RR^n$ that contain some point of $A$.

The basic building block that we use below to construct distances between sets is the excess of one set over another.  Let $A,B\subseteq\RR^n$ be nonempty. The \emph{excess of $A$ over $B$} is the quantity defined by  \begin{equation}\ue(A,B) = \sup_{a\in A} \inf_{b\in B} |a-b|\in[0,\infty].\end{equation} By convention, we also assign $\ue(\emptyset, B)=0$, but leave the expression $\ue(A,\emptyset)$ undefined. When $A=\{x\}$, the excess of $\{x\}$ over $B$ is precisely the \emph{distance $\dist(x,B)$ of $x$ to $B$}. Geometrically the excess of $A$ over $B$ is less than (at most) some $\varepsilon>0$ precisely when $A$ is contained in the open (closed) $\varepsilon$-tubular neighborhood of $B$.   Some key properties of excess include \begin{itemize}
\item \emph{closure}: $\ue(A,B)=\ue(A,\cl{B})=\ue(\cl{A},\cl{B})=\ue(\cl{A},B)$; if $A\neq\emptyset$ and $A$ is bounded, then there exist $\bar a\in\cl{A}$ and $\bar b\in \cl{B}$ such that $\ue(A,B)=\left|\bar a-\bar b\right|$;
\item \emph{containment}: $\ue(A,B)=0$ if and only if $A\subseteq \cl{B}$;
\item \emph{monotonicity}: if $A\subseteq A'$ and $B\supseteq B'$, then $\ue(A,B)\leq \ue(A',B')$; and
\item \emph{triangle inequality}: $\ue(A,C)\leq \ue(A,B)+\ue(B,C)$.
\end{itemize} We leave the proof of these properties as an exercise for the reader. We emphasize that excess is an asymmetric quantity and is allowed to be infinite.

For all $x\in\RR^n$ and $r>0$ and for all $A,B\subseteq\RR^n$ with $B$ nonempty, we define the \emph{relative excess of $A$ in $\ball(x,r)$ over $B$} as \begin{equation}\mud x r (A,B):= \frac 1r \ue(A\cap \ball(x,r), B)\in[0,\infty).\end{equation}  We include the factor $1/r$ in the definition of relative excess so that $\mud x r$ is scale invariant in the sense that \begin{equation}\mud{x}{r} (A,B)=\mud{\lambda x}{\lambda r}(\lambda A,\lambda B)\quad \text{for all }\lambda>0.\end{equation} Relative excess is also translation invariant in the sense that \begin{equation}\mud{x}{r}(A,B)=\mud{x+z}{r}(z+A,z+B)\quad\text{for all }z\in\RR^n.\end{equation} In contrast to the excess, the relative excess of one set over another is always finite. Moreover, if $\cl{B}$ contains $x$, then $\mud{x}{r}(A,B)\leq 1$; and if $\cl{B}\cap \ball(x,r)\neq\emptyset$, then $\mud{x}{r}(A,B)\leq 2$. Relative excess inherits the following additional properties from  excess.
\begin{lemma} \label{l:relativeExcess} Let $A,B,C\subseteq\RR^n$ with $B$ and $C$ nonempty, let $x,y\in\RR^n$ and let $r,s>0$.
\begin{itemize}
\item \emph{closure}: $\mud{x}{r}(A,B)=\mud{x}{r}(A,\cl{B})\leq \mud{x}{r}(\cl{A},\cl{B})=\mud{x}{r}(\cl{A},B)$, and for all $\delta>0$, $\mud{x}{r}(\cl{A},B)\leq (1+\delta) \mud{x}{r(1+\delta)}(A,B)$.
\item \emph{containment}: $\mud{x}{r}(A,B)=0$ if and only if $A\cap \ball(x,r)\subseteq \cl{B}$.
\item \emph{monotonicity}: If $\ball(x,r)\subseteq \ball(y,s)$, $A\subseteq A'$ and $B\supseteq B'$, then  \begin{equation}\label{osm}\mud{x}{r}(A,B) \leq \frac{s}{r} \mud {y}{s}(A',B').\end{equation}
\item \emph{strong quasitriangle inequality}: If $\mud{x}{r} (A,B) \leq \varepsilon$, then
    \begin{equation}\label{osqti}\mud{x}{r}(A,C)\leq\mud {x}{r} (A,B) +  (1+\varepsilon) \mud {x}{r(1+\varepsilon)} (\cl{B},C).\end{equation}
\item \emph{weak quasitriangle inequalities}: If $x\in \cl{B}$, then \begin{equation}\label{oswqti}\mud{x}{r}(A,C)\leq \mud{x}{r}(A,B)+2\mud{x}{2r}(\cl{B},C).\end{equation}
If $\cl{B}\cap \ball(x,r)\neq\emptyset$, then \begin{equation}\label{oswqti3}\mud{x}{r}(A,C)\leq \mud{x}{r}(A,B)+3\mud{x}{3r}(\cl{B},C).\end{equation}
\end{itemize} \end{lemma}

\begin{proof} To start, note that \begin{equation*} \ue(A\cap \ball(x,r), B) = \ue (A\cap \ball(x,r),\cl{B}) \leq \ue(\cl{A}\cap \ball(x,r),\cl{B})=\ue(\cl{A}\cap \ball(x,r),B) \end{equation*} by the closure and monotonicity prosperities of excess. Similarly, for each $\delta>0$, \begin{equation*} \ue(\cl{A}\cap \ball(x,r),B) \leq \ue(\cl{A \cap \ball(x,r(1+\delta))},B)=(1+\delta)\frac{1}{1+\delta}\ue(A\cap \ball(x,r(1+\delta)),B).\end{equation*} Dividing through each of the two displayed lines by $r$ establishes the first and second parts of the closure property of relative excess, respectively. Likewise the containment property of relative excess follows immediately from the containment property of excess.

If $\ball(x,r)\subseteq \ball(y,s)$, $A\subseteq A'$ and $B\supseteq B'$, then \begin{equation*} \mud{x}{r}(A,B) = \frac{1}{r}\ue(A\cap \ball(x,r),B) \leq \frac{1}{r}\ue(A'\cap \ball(y,s),B')=\frac{s}{r}\mud{y}{s}(A',B')\end{equation*} by monotonicity of excess.  This establishes the monotonicity property of relative excess.

The strong quasitriangle inequality holds trivially if $A\cap \ball(x,r)=\emptyset$. Thus suppose that $A\cap \ball(x,r)$ is nonempty and $\mud{x}{r}(A,B)\leq \varepsilon$ for some $\varepsilon\geq 0$.  Fix $a\in A\cap \ball(x,r)$ and fix $\delta>0$. By the closure and monotonicity properties of excess, there exists $\bar b\in \cl{B}$ such that \begin{equation*} \left|a-\bar b\right| = \ue(\{a\},B)\leq \ue(A\cap \ball(x,r),B)= r\mud{x}{r}(A,B)\leq r\varepsilon. \end{equation*} Similarly, there exists $c\in C$ such that \begin{equation*} \left|\bar b-c\right|\leq \ue(\{\overline{b}\},C)+\delta \leq \ue(\cl{B}\cap \ball(x,r(1+\varepsilon)),C)+\delta = r(1+\varepsilon)\mud{x}{r(1+\varepsilon)}(\cl{B},C)+\delta.\end{equation*} Hence, combining the two displayed equations, \begin{equation*} \ue(\{a\},C)\leq \left|a-\bar b\right|+\left|\bar b-c\right| \leq r \mud{x}{r}(A,B)+r(1+\varepsilon)\mud{x}{r(1+\varepsilon)}(\cl{B},C)+\delta.\end{equation*} Letting $\delta\rightarrow 0$ and then taking the supremum over all $a$ in  $A\cap \ball(x,r)$, we obtain \begin{equation*} \ue(A\cap \ball(x,r),C) \leq r\mud{x}{r}(A,B)+r(1+\varepsilon)\mud{x}{r(1+\varepsilon)}(\cl{B},C).\end{equation*} The strong quasitriangle inequality follows by dividing the last line through by $r$.

If $x\in \cl{B}$, then $\mud{x}{r}(A,B)\leq 1$. Similarly, if $\ball(x,r)\cap \cl{B}\neq\emptyset$, then $\mud{x}{r}(A,B)\leq 2$ Thus, the weak quasitriangle inequalities follows from the strong quasitriangle inequality. \end{proof}

For all $x\in\RR^n$ and $r>0$, define \begin{equation}\mD x r [A,B] = \max\left\{\mud x r (A,B), \mud x r (B,A)\right\}\in[0,\infty)\end{equation} for all nonempty sets $A,B\subseteq\RR^n$. To the authors' knowledge, the quantity $\mD{x}{r}[\cdot,\cdot]$ (without the normalization factor $1/r$) first appeared in a paper by Walkup and Wets \cite{WW}. (For further context, see the bibliographic notes in \cite{RWbook}*{Chapter 4}.) Thus, we call $\mD{x}{r}[\cdot,\cdot]$ the \emph{relative Walkup-Wets distance in $\ball(x,r)$}.

If $A$ and $B$ both contain $x$, then $\mD{x}{r}[A,B]\leq 1$; if $A$ and $B$ both intersect $\ball(x,r)$, then $\mD{x}{r}[A,B]\leq 2$. The relative Walkup-Wets distance inherits the following properties from the relative excess.

\begin{lemma} \label{l:relativeWW} Let $A,B,C\subseteq\RR^n$ be nonempty sets, let $x,y\in\RR^n$ and let $r,s>0$.
\begin{itemize}
\item \emph{closure}: $\mD{x}{r}[A,B] \leq \mD{x}{r}[\cl{A},\cl{B}]\leq (1+\delta)\mD{x}{r(1+\delta)}[A,B]$ for all $\delta>0$.
\item \emph{containment}: $\mD{x}{r}[A,B]=0$ if and only if $A\cap \ball(x,r)\subseteq\cl{B}$ and $B\cap \ball(x,r)\subseteq \cl{A}$. In particular, $\mD{x}{r}[\cl{A},\cl{B}]=0$ if and only if $\cl{A}\cap \ball(x,r)=\cl{B}\cap \ball(x,r)$.
\item \emph{monotonicity}: If $\ball(x,r)\subseteq \ball(y,s)$, then \begin{equation}\label{Dmono} \mD{x}{r}[A,B] \leq \frac{s}{r} \mD{y}{s}[A,B].\end{equation}
\item \emph{strong quasitriangle inequality}: If $\mud{x}{r} (A,B) \leq \varepsilon_1$ and  $\mud{x}{r}(C,B) \leq \varepsilon_2$, then
\begin{equation}\label{sqti}\mD {x}{r}[A, C]\leq (1+\varepsilon_2)\mD {x}{r(1+\varepsilon_2)} [A, \cl{B}] + (1+\varepsilon_1)\mD{x}{r(1+\varepsilon_1)} [\cl{B}, C].\end{equation}
\item \emph{weak quasitriangle inequalities}: If $x\in \cl{B}$, then
\begin{equation}\label{qti}\mD x r [A, C]\leq 2 \mD x {2r} [A, \cl{B}] + 2\mD x {2r} [\cl{B}, C].\end{equation} If $\cl{B}\cap \ball(x,r)\neq \emptyset$, then \begin{equation}\label{qti3}\mD x r [A, C]\leq 3 \mD x {3r} [A, \cl{B}] + 3\mD x {3r} [\cl{B}, C].\end{equation}
\item \emph{scale invariance}: \begin{equation}\label{si}\mD{x}{r}[A,B]=\mD{\lambda x}{\lambda r}[\lambda A,\lambda B]\quad\text{for all }\lambda>0.\end{equation}
\item \emph{translation invariance}: \begin{equation}\label{ti}\mD{x}{r}[A,B]=\mD{x+z}{r}[z+ A,z+ B]\quad\text{for all }z\in\RR^n.\end{equation}
\end{itemize} \end{lemma}

\begin{proof} We will derive the quasitriangle inequalities for  relative Walkup-Wets distance, but leave verification of the other properties of $\mD{x}{r}[\cdot,\cdot]$ to the reader.

 Suppose that $\mud{x}{r}(A,B) \leq \varepsilon_1$ and $\mud{x}{r}(C,B)\leq \varepsilon_2$. Since $\mud{x}{r}(A,B)\leq \varepsilon_1$,
\begin{align*} \mud x r (A, C) &\leq \mud x r (A,\cl{B}) + (1+\varepsilon_1) \mud x {r(1+\varepsilon_1)} (\cl{B}, C) \\ &\leq (1+\varepsilon_2)\mud x {r(1+\varepsilon_2)}(A, \cl{B})+ (1+\varepsilon_1) \mud x {r(1+\varepsilon_1)} (\cl{B}, C)\end{align*} by the strong quasitriangle inequality and monotonicity of relative excess. Similarly, since $\mud{x}{r}(C,B)\leq \varepsilon_2$,
\begin{equation*}\mud {x}{r} (C, A) \leq (1+\varepsilon_1) \mud x {r(1+\varepsilon_1)} (C,\cl{B}) + (1+\varepsilon_2)\mud x {r(1+\varepsilon_2)}(\cl{B}, A).\end{equation*} Thus, noting that $\max\{a+b, c+d\}\leq \max\{a, d\} + \max\{b,c\}$, we obtain
\begin{align*} \mD x r [A, C] &= \max\{\mud x r(A, C),  \mud x r (C, A)\} \\
&\leq \max\{  (1+\varepsilon_2)\mud x {r(1+\varepsilon_2)}(A, \cl{B})+ (1+\varepsilon_1) \mud x {r(1+\varepsilon_1)} (\cl{B}, C), \\
 &\qquad\qquad (1+\varepsilon_1) \mud x {r(1+\varepsilon_1)} (C,\cl{B}) + (1+\varepsilon_2)\mud x {r(1+\varepsilon_2)}(\cl{B}, A)\}\\
&\leq \max\{  (1+\varepsilon_2)\mud x {r(1+\varepsilon_2)}(A, \cl{B}), (1+\varepsilon_2)\mud x {r(1+\varepsilon_2)}(\cl{B}, A)\}\\
&\qquad+\max\{(1+\varepsilon_1) \mud x {r(1+\varepsilon_1)} (\cl{B}, C), (1+\varepsilon_1) \mud x {r(1+\varepsilon_1)} (C,\cl{B}) \}\\
&= (1+\varepsilon_2)\mD x {r(1+\varepsilon_2)} [A, \cl{B}] + (1+\varepsilon_1)\mD x {r(1+\varepsilon_1)} [\cl{B}, C].\end{align*}
Therefore, the relative Walkup-Wets distance satisfies the strong quasitriangle inequality. The weak quasitriangle inequalities then follow, since $\mud{x}{r}(A,B)\leq 1$ and $\mud{x}{r}(C,B)\leq 1$ if $x\in\cl{B}$, while $\mud{x}{r}(A,B)\leq 2$ and $\mud{x}{r}(C,B)\leq 2$ if $\cl{B}\cap \ball(x,r)\neq\emptyset$.
\end{proof}

\begin{remark} It is an unfortunate, but unavoidable fact that the relative Walkup-Wets distance $\mD{x}{r}[\cdot,\cdot]$ does not satisfy the triangle inequality. To rectify this, one might be tempted to instead work with the \emph{relative Hausdorff distance $\D{x}{r}[\cdot,\cdot]$} defined by \begin{equation}\D{x}{r}[A,B]=\frac{1}{r}\max\left\{\ue(A\cap \ball(x,r), B\cap \ball(x,r)), \ue(B\cap \ball(x,r),A\cap \ball(x,r))\right\}\end{equation} for all sets $A$ and $B$ that intersect $\ball(x,r)$. Although the relative Hausdorff distance does satisfy the triangle inequality $\D{x}{r}[A,C]\leq \D{x}{r}[A,B]+ \D{x}{r}[B,C]$, the quantity is deficient in other respects. One of the main defects in terms of our applications below is that the relative Hausdorff distance does not satisfy the monotonicity property (\ref{Dmono}). To see this, take $A=\{0,1\}\subset\RR^1$ and $B_i=\{0,1+1/i\}\subset\RR^1$. Then $\D{0}{1}[A,B_i]=1$ and  $\D{0}{1+1/i}[A,B_i]=1/(i+1)$, so that  $$\frac{\D{0}{1}[A,B_i]}{\D{0}{1+1/i}[A,B_i]}=i+1\rightarrow\infty\quad\text{ as }i\rightarrow\infty.$$ For additional reasons to work with the relative Walkup-Wets distance instead of with the relative Hausdorff distance, see Remarks \ref{r:naive} and \ref{theta_justification}.
\end{remark}

\begin{remark} We could also have defined the relative excess and Walkup-Wets distance using intersections with open balls $U(x,r)$ instead of intersections with closed balls $\ball(x,r)$; the decision of which convention to use is largely a matter of taste. If the definition with open balls is selected, then Lemmas \ref{l:relativeExcess} and \ref{l:relativeWW} hold as is, except that inequality in the closure properties becomes equality: $\mud{x}{r}(A,B)=\mud{x}{r}(\cl{A},\cl{B})$ and $\mD{0}{r}[A,B]=\mD{0}{r}[\cl{A},\cl{B}]$ (intersections with open balls). \end{remark}

We now recall a convenient topology on the space of all nonempty closed sets in $\RR^n$.

\begin{theorem}\label{t:aw} There exists a metrizable topology on the collection $\CL{\RR^n}$ of all nonempty closed sets in $\RR^n$ in which a sequence $(A_i)_{i=1}^\infty$ of nonempty closed sets converges to a set $A\in \CL{\RR^n}$ if and only if \begin{equation} \label{e:aw} \lim_{i\rightarrow\infty} \ue(A_i\cap \ball(0,r),A)=0\quad\text{and}\quad \lim_{i\rightarrow\infty} \ue(A\cap \ball(0,r),A_i)=0\quad\text{for all }r>0.\end{equation} Moreover, for every compact set $K\subset\RR^n$, the subcollection $\CL{K}$ of $\CL{\RR^n}$, consisting of all closed sets that intersect $K$, is sequentially compact; that is, for any sequence $(A_i)_{i=1}^\infty$ of sets in $\CL{K}$ there exist a subsequence $(A_{i_j})_{j=1}^\infty$ and a set $A\in\CL{K}$ such that $(A_{i_j})_{j=1}^\infty$ converges to $A$ in the sense of (\ref{e:aw}).\end{theorem}

In variational analysis, the topology on $\CL{\RR^n}$ described in Theorem \ref{t:aw} is called the \emph{Attouch-Wets topology}; for further information, including a proof of Theorem \ref{t:aw}, see \cite{Beer}*{Chapter 3}, \cite{RWbook}*{Chapter 4}, or \cite{DS}*{Chapter 8}. Below we always endow $\CL{\RR^n}$ with the Attouch-Wets topology and write $A_i\rightarrow A$ or $A=\lim_{i\rightarrow\infty} A_i$ (in $\CL{\RR^n}$) to denote that a sequence $(A_i)_{i=1}^\infty$ of sets in $\CL{\RR^n}$ converges to a set $A\in\CL{\RR^n}$ in the sense of (\ref{e:aw}). If each set $A_i$ belongs to $\CL{K}$ for some compact set $K\subset\RR^n$, then we may also write \emph{$A_i\rightarrow A$ in $\CL{K}$} to emphasize that the limit $A$ belongs to $\CL{K}$, as well.

\begin{lemma}\label{l:modes} Let $A,A_1,A_2,\dots\in\CL{\RR^n}$. The following statements are equivalent: \begin{enumerate}
 \item $A_i\rightarrow A$ in $\CL{\RR^n}$;
 \item $\lim_{i\rightarrow \infty}\mD{x}{r}[A_i,A]=0$ for all $x\in\RR^n$ and for all $r>0$;
 \item $\lim_{i\rightarrow \infty}\mD{x_0}{r_j}[A_i,A]=0$ for some $x_0\in\RR^n$ and for some sequence $r_j\rightarrow\infty$.
 \end{enumerate}
 \end{lemma} \begin{proof} By definition $A_i\rightarrow A$ in $\CL{\RR^n}$ if and only if (\ref{e:aw}) holds. Since for each $r>0$, $$\mD{0}{r}[A_i,A] = \frac{1}{r}\max\{ \ue(A_i\cap \ball(0,r),A), \ue(A\cap \ball(0,r),A_i)\},$$ we immediately obtain $A_i\rightarrow A$ in $\CL{\RR^n}$ if and only if $\lim_{i\rightarrow\infty}\mD{0}{r}[A_i,A]=0$ for all $r>0$. In particular, $(ii)\Rightarrow (i)\Rightarrow (iii)$.

Suppose that $(iii)$ holds for some $x_0\in\RR^n$ and for some sequence $r_j\rightarrow \infty$. Fix $x\in\RR^n$ and $r>0$. Since $r_j\rightarrow\infty$, we can pick $k$ such that $\ball(x,r)\subseteq \ball(x_0,r_k)$. Hence $$\limsup_{i\rightarrow\infty} \mD{x}{r}[A_i,A] \leq \frac{r_k}{r} \limsup_{i\rightarrow\infty} \mD{x_0}{r_k}[A_i,A] =0$$ by  monotonicity  of the relative Walkup-Wets distance (see (\ref{Dmono})). Therefore, since $x\in\RR^n$ and $r>0$ were fixed arbitrarily, $(iii)\Rightarrow(ii)$.
 \end{proof}

\begin{lemma}\label{QtriangleVsLimits} Let $A,A_1,A_2,\dots \in\CL{\RR^n}$. If $A_i\to A$ in $\CL{\RR^n}$,  then for all nonempty $B\subseteq\RR^n$, for all $x\in\RR^n$ and for all $r>0$,
\begin{equation}\begin{split}\label{QtriangleVsLimits.1}
\frac{1}{1+\varepsilon}\limsup_{i\to\infty}&\mD x {r/(1+\varepsilon)}[A_i, B] \\ \leq &\mD x r [A, B] \leq (1+\varepsilon)\liminf_{i\to\infty}\mD x {(1+\varepsilon)r} [A_i, B]\quad \text{for all }\varepsilon>0\end{split}\end{equation}  and \begin{equation} \label{QtriangleVsLimits.2} \limsup_{\varepsilon\downarrow 0} \limsup_{i\rightarrow\infty} \mD{x}{r/(1+\varepsilon)}[A_i,B] \leq \mD{x}{r}[A,B]
\leq \liminf_{\varepsilon\downarrow 0}\liminf_{i\rightarrow\infty} \mD{x}{r(1+\varepsilon)}[A_i,B].\end{equation}\end{lemma}

\begin{proof} Suppose that $A_i\rightarrow A$ in $\CL{\RR^n}$. Fix a nonempty set $B\subseteq\RR^n$, $x\in\RR^n$ and $r>0$.

Let $\varepsilon>0$ be arbitrary. On one hand, $\mud{x}{r}(A,A_i)\leq \varepsilon$ for all $i$ sufficiently large, since $A_i\rightarrow A$ in $\CL{\RR^n}$. On the other hand, writing $L=\mud{x}{r}(B,A)$, $$\mud{x}{r}(B,A_i)\leq \mud{x}{r}(B,A)+(1+L) \mud{x}{r(1+L)}(A,A_i)$$ by the strong quasitriangle inequality for the relative excess. In particular, $$M=\sup_i \mud{x}{r}(B,A_i) <\infty,$$  since $A_i\rightarrow A$ in $\CL{\RR^n}$.  Hence, by the strong quasitriangle inequality for the Walkup-Wets distance, $$\mD{x}{r}[A,B] \leq (1+M)\mD{x}{r(1+M)}[A,A_i] + (1+\varepsilon)\mD{x}{r(1+\varepsilon)}[A_i,B]$$ for all $i$ sufficiently large. Thus, since $A_i\rightarrow A$ in $\CL{\RR^n}$, $$\mD{x}{r}[A,B] \leq (1+\varepsilon) \liminf_{i\rightarrow\infty} \mD{x}{r(1+\varepsilon)}[A_i,B].$$ This establishes the upper bound for $\mD{x}{r}[A,B]$ in (\ref{QtriangleVsLimits.1}). We can reach the lower bound by a similar argument:

Note that $\mud{x}{r/(1+\varepsilon)}(A_i,A)\leq \varepsilon$ for all $i$ sufficiently large, since $A_i\rightarrow A$ in $\CL{\RR^n}$. Let $N=\mud{x}{r/(1+\varepsilon)}(B,A)<\infty$. Then, by the strong quasitriangle inequality for the Walkup-Wets distance, $$\mD{x}{r/(1+\varepsilon)}[A_i,B] \leq (1+N) \mD{x}{r(1+N)/(1+\varepsilon)}[A_i,A]+(1+\varepsilon)\mD{x}{r}[A,B]$$ for all $i$ sufficiently large. Rearranging terms yields $$\mD{x}{r}[A,B] \geq \frac{1}{1+\varepsilon} \mD{x}{r/(1+\varepsilon)}[A_i,B] - \frac{1+N}{1+\varepsilon}\mD{x}{r(1+N)/(1+\varepsilon)}[A_i,A]$$ for all $i$ sufficiently large. Thus, since $A_i\rightarrow A$ in $\CL{\RR^n}$, $$\mD{x}{r}[A,B]\geq \frac{1}{1+\varepsilon}\limsup_{i\rightarrow \infty} \mD{x}{r/(1+\varepsilon)}[A_i,B],$$ as desired.

Finally, suppose that $0<\varepsilon\leq \varepsilon'$. Then, by (\ref{QtriangleVsLimits.1}),  $$\mD{x}{r}[A,B]\leq (1+\varepsilon)\liminf_{i\rightarrow\infty} \mD{x}{r(1+\varepsilon)}[A_i,B]\leq (1+\varepsilon')\liminf_{i\rightarrow\infty}\mD{x}{r(1+\varepsilon)}[A_i,B].$$ Hence $$\mD{x}{r}[A,B]\leq (1+\varepsilon') \liminf_{\varepsilon\downarrow 0} \liminf_{i\rightarrow\infty} \mD{x}{r(1+\varepsilon)}[A_i,B].$$ Thus, letting $\varepsilon'\downarrow 0$, we obtain $$\mD{x}{r}[A,B] \leq \liminf_{\varepsilon\downarrow 0}\liminf_{i\rightarrow\infty}\mD{x}{r(1+\varepsilon)}[A_i,B].$$ This establishes the upper for $\mD{x}{r}[A,B]$ in (\ref{QtriangleVsLimits.2}). The lower bound in (\ref{QtriangleVsLimits.2}) follows from the lower bound in (\ref{QtriangleVsLimits.1}) by a parallel argument.
\end{proof}

\begin{remark}\label{r:naive} Suppose that we declare $A_i\rightarrow A$ in \emph{$\CL{0}$ relative to the Hausdorff distance} if $\lim_{i\rightarrow\infty} \D{0}{r}[A_i,A]=0$ for all $r>0$. Then, with respect to convergence relative to the Hausdorff distance, $\CL{0}$ is not sequentially compact. Indeed, let $A_i=\{0,1+1/i\}\subset\RR^1$ and suppose that $(A_i)_{i=1}^\infty$ has a subsequence $(A_{ij})_{j=1}^\infty$ which converges relative to the Hausdorff distance to some $A\in\CL{0}$. On one hand, since $$\lim_{i\rightarrow\infty} \D{0}{r}[A_i,A]=0\quad\text{and}\quad\lim_{i\rightarrow\infty} \D{0}{r}[A_i,\{0,1\}]=0\quad\text{for all }r>1,$$ the triangle inequality for the Hausdorff distance yields $\D{0}{r}[A,\{0,1\}]=0$ for all $r>1$. This implies that $A=\{0,1\}$. On the other hand, $(A_{ij})_{j=1}^\infty$ does not converge relative to the Hausdorff distance to $\{0,1\}$, because $\D{0}{1}[A_i,\{0,1\}]= 1$ for all $i\geq 1$. Nevertheless, $A_i\rightarrow\{0,1\}$ in the Attouch-Wets topology, because $$\lim_{i\rightarrow\infty} \mD{0}{r}[A_i,\{0,1\}]=0\quad\text{for all }r>0.$$ This example indicates another advantage of working with the relative Walkup-Wets distance instead of the relative Hausdorff distance. \end{remark}

\section{Tangent sets and pseudotangent sets}
\label{s:tangents}

In this section, we define and establish some essential properties of tangent sets and pseudotangent sets, which are modeled on tangent measures (introduced by Preiss \cite{Preiss}) and pseudotangent measures (introduced by Kenig and Toro \cite{KT99}). The main novelty in our presentation appears in Lemma \ref{l:bdtan}.

\begin{definition}[tangent sets and pseudotangent sets]\label{blowups} Let $A, T, D\subseteq\RR^n$ be nonempty sets with $A$ and $T$ closed and let $x\in A$.
We say that $T$ is a \emph{pseudotangent set of $A$ at $x$ directed along $D$} if there exist sequences $x_i\in A$ and $r_i>0$ such that $x_i\rightarrow x$, $r_i\rightarrow 0$, \begin{equation}\frac{A-x_i}{r_i}\rightarrow T \quad\text{in }\CL{0}\end{equation} and \begin{equation}\frac{x_i-x}{r_i}\in D\quad\text{for all }i.\end{equation} If $D=\RR^n$, i.e.~ if no restrictions are imposed on $x_i-x$, then we call $T$ a \emph{pseduotangent set of $A$ at $x$}. If $D=\{0\}$, i.e.~ if $x_i=x$ for all $i$, then we call $T$ a \emph{tangent set of $A$ at $x$}. Let $\Tan_D(A,x)$, $\PsTan(A,x)$, and $\Tan(A,x)$ denote the collections of all pseudotangent sets of $A$ at $x$ directed along $D$, all pseudotangent sets of $A$ at $x$, and all tangent sets of $A$ at $x$, respectively.
\end{definition}

\begin{remark} \label{r:GH} A closed set $A\in\CL{x}$ containing $x\in\RR^n$ can be identified with a pointed metric space $(A,d|_A,x)$, where $d|_A$ denotes the restriction of the Euclidean metric to $A$. In geometry, a \emph{metric tangent} of $A$ at $x$ (see e.g.~ \cite{CCol1}, \cite{LeDonne}) is a Gromov-Hausdorff limit of a sequence of pointed metric spaces $(A,r_i^{-1} d|_{A},x)$ for some $r_i\rightarrow 0$. Although related, the concepts of metric tangents and tangent sets are distinct, because metric tangents are identified up to isometry, whereas tangent sets are not identified by isometries.
\end{remark}

\begin{remark}
Since $\CL{0}$ is sequentially compact, $\PsTan(A,x)\subseteq\CL{0}$ and the collection $\Tan_D(A,x)$ of pseudotangent sets of $A$ at $x$ directed along $D$ is nonempty as long as \begin{equation*} \frac{A-x}{r_i}\cap D\cap B(0,s)\neq\emptyset\quad\text{for some sequence }r_i\rightarrow 0\text{ and some }s>0.\end{equation*} In particular, $\PsTan(A,x)$ and $\Tan(A,x)$ are nonempty for all $A\in\CL{x}$.
\end{remark}

\begin{lemma}\label{l:TanDclosed} $\Tan_D(A,x)$ is closed in $\CL{0}$ for all $A\in\CL{x}$ and nonempty $D\subseteq\RR^n$.\end{lemma}

\begin{proof} Let $D\subseteq\RR^n$ be a nonempty set and let $A\subset\RR^n$ be a closed set containing $x$. Suppose that $T_i\in \Tan_D(A,x)$ for all $i\geq 1$ and $\lim_{i\rightarrow\infty}T_i=T$ for some $T\in \CL{0}$. By definition of pseudotangents directed along $D$, for each $i\geq 1$ we can find sequences $x^{i}_j\in A$ and $r^{i}_j>0$ such that $\lim_{j\rightarrow\infty} x^{i}_j=x$, $\lim_{j\rightarrow\infty} r^{i}_j=0$, $\lim_{j\rightarrow\infty}(A-x^{i}_j)/r^{i}_j= T_i$ and $(x^{i}_j-x)/r^{i}_j\in D$ for all $j\geq 1$. For each $i\geq 1$, pick $k(i)\geq 1$ large enough so that $x_i:=x^{i}_{k(i)}$ and $r_i:=r^{i}_{k(i)}$ satisfy $|x_i-x|\leq 1/i$, $r_i\leq 1/i$ and \begin{equation}\label{e:TanD1}\mD{0}{i}\left[\frac{A-x_i}{r_i}, T_i\right] \leq \frac{1}{i^2}.\end{equation} Then $x_i\in A$ and $r_i>0$ for all $i\geq 1$, $\lim_{i\rightarrow\infty} x_i=x$, $\lim_{i\rightarrow\infty} r_i=0$, and $(x_i-x)/r_i\in D$ for all $i\geq 1$. Moreover, for all $r>0$ and for all $i\geq r$,
\begin{equation*}\begin{split} \mD{0}{r/2}\left[\frac{A-x_i}{r_i},T\right] &\leq 2\mD{0}{r}\left[\frac{A-x_i}{r_i},T_i\right]+2\mD{0}{r}[T_i,T]\\
&\leq \frac{2i}{r}\mD{0}{i}\left[\frac{A-x_i}{r_i},T_i\right]+2\mD{0}{r}[T_i,T]\leq \frac{2}{ri}+2\mD{0}{r}[T_i,T],\end{split}\end{equation*} where the first line holds by the weak quasitriangle inequality since $0\in T_i$ and the second line holds by monotonicity and (\ref{e:TanD1}). Hence, for all $r>0$, $$ \limsup_{i\rightarrow\infty} \mD{0}{r/2}\left[\frac{A-x_i}{r_i},T\right] \leq \frac{2}{r}\limsup_{i\rightarrow\infty} \frac{1}{i}+2\limsup_{i\rightarrow\infty} \mD{0}{r}[T_i,T]=0.$$ Thus, $T=\lim_{i\rightarrow\infty}(A-x_i)/r_i$ is a pseudotangent set of $A$ at $x$ directed along $D$. Therefore, $\Tan_D(A,x)$ is closed in $\CL{0}$.
\end{proof}

\begin{lemma}\label{inv_of_tan} Let $A\in\CL{x}$ and let $D\subseteq\RR^n$ be nonempty.
\begin{itemize}
\item If $B\in \Tan_D(A,x)$ and $\lambda > 0$, then $\lambda B\in \Tan_{\lambda D}(A,x)$.
\item If $B\in \PsTan(A,x)$ and $y\in B$, then ${B-y}\in\PsTan(A,x)$.
\end{itemize}
\end{lemma}

\begin{proof} Let $A\in\CL{x}$ and let $D\subseteq\RR^n$ be nonempty. Suppose $B\in \Tan_D(A,x)$ and $\lambda >0$. Since $B$ is a pseudotangent of $A$ at $x$ directed along $D$, there exist sequences $x_i\in A$ and $r_i>0$ with $(x_i-x)/r_i\in D$ such that $x_i\rightarrow x$, $r_i\rightarrow 0$ and $(A-x_i)/r_i\rightarrow B$. For all $r>0$, $$\lim_{i\rightarrow\infty} \mD{0}{r}\left[\frac{A-x_i}{r_i/\lambda},\lambda B\right] =\lim_{i\rightarrow \infty}\mD{0}{r/\lambda}\left[\frac{A-x_i}{r_i},B\right]=0.$$ Hence $\lambda B=\lim_{i\rightarrow \infty} (A-x_i)/(r_i/\lambda)$ and $(x_i-x)/(r_i/\lambda)\in \lambda D$ for all $i\geq 1$. Therefore, $\lambda B\in\Tan_{\lambda D}(A,x)$.

Now suppose $B\in \PsTan(A,x)$ and let $y\in B$. Since $B$ is a pseudotangent of $A$ at $x$, there exist sequences $x_i\in A$ and $r_i>0$ such that $x_i\to x$, $r_i\rightarrow 0$, and $(A-x_i)/r_i\to B$. Since $y\in B=\lim_{i\rightarrow\infty} (A-x_i)/r_i$, there is also a sequence $z_i\in A$ such that $(z_i-x_i)/r_i\to y$. On one hand, $z_i\to x$, because $|z_i-x|\leq |z_i-x_i|+|x_i-x|\leq r_i(|y|+1) + |x_i-x|$ for $i\gg 1$, $x_i\rightarrow x$ and $r_i\rightarrow 0$. On the other hand, $(A-z_i)/r_i\rightarrow B-y$, because for all $r>0$,
\begin{align*}
\mD{0}{r}\left[\frac{A-z_i}{r_i},B-y\right]&\leq 2 \mD{0}{2r}\left[\frac{A-z_i}{r_i},\frac{A-x_i}{r_i}-y\right]+2\mD{0}{2r}\left[\frac{A-x_i}{r_i}-y, B-y\right]\\
&\leq \frac{1}{r}\left|\frac{z_i-x_i}{r_i}-y\right| + 2\mD{y}{2r}\left[\frac{A-x_i}{r_i}, B\right],
\end{align*} $(z_i-x_i)/r_i\rightarrow y$ and $(A-x_i)/r_i\rightarrow B$. Therefore, $B-y$ is a pseudotangent of $A$ at $x$.
\end{proof}

\begin{lemma}\label{l:iterate} Let $A\in\CL{x}$. \begin{itemize}
\item  If $B\in\Tan(A,x)$ and $C\in \Tan(B,0)$, then $C\in \Tan(A,x)$.
\item If $B\in \PsTan(A,x)$ and $C\in\PsTan(B,y)$ for some $y\in B$, then $C\in\PsTan(A,x)$.
\end{itemize}
\end{lemma}

\begin{proof} Suppose $B\in \Tan(A,x)$ and $C\in\Tan(B,0)$. Since $C$ is a tangent set of $B$ at $0$, there exists a sequence $r_i>0$ such that $r_i\rightarrow 0$ and $B/r_i\rightarrow C$. By Lemma \ref{inv_of_tan}, $B/r_i\in \Tan(A,x)$ for all $i\geq 1$. Therefore, $C=\lim_{i\rightarrow\infty} B/r_i\in \overline{\Tan(A,x)}=\Tan(A,x)$ by Lemma \ref{l:TanDclosed}.

Similarly, suppose $B\in\PsTan(A,x)$ and $C\in\PsTan(B,y)$ for some $y\in B$. Since $C$ is a pseudotangent of $B$ at $y$, there exist sequences $y_i\in B$ and $s_i>0$ such that $y_i\rightarrow y$, $s_i\rightarrow 0$ and $(B-y_i)/s_i\rightarrow C$. By Lemma \ref{inv_of_tan}, $(B-y_i)/s_i\in \PsTan(A,x)$ for all $i\geq 1$. Therefore, $C=\lim_{i\rightarrow\infty} (B-y_i)/s_i\in \overline{\PsTan(A,x)}=\PsTan(A,x)$ by Lemma \ref{l:TanDclosed}. \end{proof}

\begin{lemma}\label{l:bdtan} Let $A\in\CL{x}$ and let $D\subseteq\RR^n$ be nonempty. If $D$ is bounded, then for all $B\in \Tan_D(A, x)$ there exists $C\in \Tan(A,x)$ and $y\in C\cap \overline{D}$ such that $B=C-y$.\end{lemma}

\begin{proof} Suppose $D\subseteq\RR^n$ is nonempty and bounded and let $B\in \Tan_D(A,x)$. Since $B$ is a pseudotangent of $A$ at $x$ directed along $D$, there exist sequences $x_i\in A$ and $r_i>0$ with $x_i\rightarrow x$ and $r_i\rightarrow 0$ such that $(A-x_i)/r_i\rightarrow B$ and $(x_i-x)/r_i\in D$ for all $i\geq 1$. Because $D$ is bounded, we may assume---by passing to a subsequence of $(x_i,r_i)_{i=1}^\infty$---that $(x_i-x)/r_i\to y$ for some $y\in \overline{D}$. To finish the proof, it suffices to show that $(A-x)/r_i\rightarrow B+y$, for then  $C:=B+y\in \Tan(A,x)$, $y\in C\cap \overline{D}$, and $B=C-y$.

To that end, note that for all $r>0$ and $i\geq 1$,
\begin{align*}
\mD{0}{r}\left[\frac{A-x}{r_i}-y,B\right] &\leq 2\mD{0}{2r}\left[\frac{A-x}{r_i}-y,\frac{A-x_i}{r_i}\right]+2\mD{0}{2r}\left[\frac{A-x_i}{r_i},B\right]\\
&\leq \frac{1}{r}\left|\frac{x_i-x}{r_i}-y\right|+2\mD{0}{2r}\left[\frac{A-x_i}{r_i},B\right],
\end{align*} where the weak quasitriangle inequality in the first line is valid since $0\in (A-x_i)/r_i$. Thus, for all $r>0$, \begin{align*} \limsup_{i\rightarrow\infty}\mD{y}{r}\left[\frac{A-x}{r_i},B+y\right] &=\limsup_{i\rightarrow\infty} \mD{0}{r}\left[\frac{A-x}{r_i}-y,B\right]\\
&\leq \frac{1}{r}\limsup_{i\rightarrow\infty} \left|\frac{x_i-x}{r_i}-y\right|+2\limsup_{i\rightarrow\infty}\mD{0}{2r}\left[\frac{A-x_i}{r_i},B\right]=0,\end{align*}
because $(x_i-x)/r_i\to y$ and $(A-x_i)/r_i\to B$. Therefore, $(A-x)/r_i\to B+y$.\end{proof}

By a slight abuse of terminology\footnote{A bounded pseudotangent set can be an unbounded set---and vice versa.}, we call $T\in\PsTan(A,x)$ a \emph{bounded} pseudotangent set of $A$ at $x$ or an \emph{unbounded} pseudotangent set of $A$ at $x$ if
$T=\lim_{i\rightarrow\infty} (A-x_i)/r_i$ for some sequences $x_i\in A$ and $r_i>0$ (with $x_i\rightarrow x$ and $r_i\rightarrow 0$) whose  direction set $$\left\{\frac{x_i-x}{r_i}:i\geq 1\right\}$$ is bounded or unbounded, respectively. We let $\bPsTan(A,x)$ and $\uPsTan(A,x)$ denote the collections of all bounded pseudotangent sets of $A$ at $x$ and all unbounded pseudotangent sets of $A$ at $x$, respectively, so that $$\PsTan(A,x)=\bPsTan(A,x)\cup \uPsTan(A,x).$$ Using this new language, Lemma \ref{l:bdtan} says that \emph{every bounded psuedotangent set of $A$ at $x$ is the translate of a tangent set of $A$ at $x$}.

\begin{lemma}\label{l:bPsTan} Let $A\in\CL{x}$.
\begin{itemize}
\item If $B\in \bPsTan(A,x)$ and $\lambda > 0$, then $\lambda B\in \bPsTan(A,x)$.
\item If $B\in \uPsTan(A,x)$ and $\lambda>0$, then $\lambda B\in\uPsTan(A,x)$.
\item If $B\in \bPsTan(A,x)$ and $y\in B$, then ${B-y}\in\bPsTan(A,x)$.
\item If $B\in \uPsTan(A,x)$ and $y\in B$, then ${B-y}\in\uPsTan(A,x)$.
\end{itemize}\end{lemma}

\begin{proof} By Lemma \ref{inv_of_tan}, if $T$ is a pseudotangent of $A$ at $x$ directed along $\{(x_i-x)/r_i\}$ and $\lambda>0$, then $\lambda T$ is a pseudotangent of $A$ directed along $\{\lambda (x_i-x)/r_i\}$. Since $\{(x_i-x)/r_i\}$ and $\{\lambda(x_i-x)/r_i\}$ are simultaneously bounded or unbounded, it follows that $\bPsTan(A,x)$ and $\uPsTan(A,x)$ are invariant under dilations.

To show that $\bPsTan(A,x)$ and $\uPsTan(A,x)$ are invariant under translations, we consult the proof of Lemma \ref{inv_of_tan}. If $T$ is a pseudotangent set of $A$ at $x$ directed along $\{(x_i-x)/r_i\}$ and $y\in T$, then there exists a sequence $z_i\in A$ with $(z_i-x_i)/r_i\rightarrow y$ such that $T-y$ is a pseudotangent set of $A$ at $x$ directed along $$\frac{z_i-x}{r_i}=\frac{x-x_i}{r_i} + \frac{z_i-x_i}{r_i}.$$ Since $\{(z_i-x_i)/r_i\}$ is bounded, $\{(x_i-x)/r_i\}$ and $\{(z_i-x)/r_i\}$ are simultaneously bounded or unbounded. Therefore, $\bPsTan(A,x)$ and $\uPsTan(A,x)$ are translation invariant.\end{proof}

\begin{remark} Let $A\in\CL{x}$. Because $\bPsTan(A,x)=\bigcup_{j=1}^\infty \Tan_{B(0,j)}(A,x),$ a sequence $T_i$ of bounded pseudotangent sets of $A$ at $x$ may converge to an unbounded pseudotangent set $T$ of $A$ at $x$ that is not a bounded pseudotangent set of $A$ at $x$. Thus, in general, $\bPsTan(A,x)$ is not closed in $\CL{0}$. (For a specific example, consider $\bPsTan(A,0)$ where $A$ is the set given in Remark \ref{r:ubad} below.)\end{remark}

\begin{remark}\label{r:tan-closed} To avoid technicalities related to issues of convergence, we have only defined tangent and pseudotangent sets of nonempty \emph{closed} sets $A\subseteq\RR^n$. However, suppose that $A\subseteq\RR^n$ is an arbitrary nonempty set, not necessarily closed, and let $x\in A$. Furthermore, suppose that for some sequence $r_i\rightarrow 0$ we have $(\cl{A}-x)/r_i\rightarrow T$ for some $T\in\Tan(\cl{A},x)$. Then, for all $r>0$, \begin{equation}\label{e:r-c} \limsup_{i\rightarrow\infty}\mD{0}{r} \left[ \frac{A-x}{r_i},T\right]
\leq \limsup_{i\rightarrow\infty}\mD{0}{r}\left[\cl{\left.\frac{A-x}{r_i}\right.},\cl{T}\right]
=\limsup_{i\rightarrow\infty} \mD{0}{r}\left[\frac{\cl{A}-x}{r_i},T\right]=0,\end{equation} where the inequality holds by the closure property of relative Walkup-Wets distance and the final equality holds since $(\cl{A}-x)/r_i\rightarrow T$ in the Attouch-Wets topology. We interpret equation (\ref{e:r-c}) to mean that the tangent sets $T$ of $\cl{A}$ at $x\in A$ are reasonable candidates for ``geometric blow-ups" of $A$ at $x$ even though the set $A$ is not necessarily closed. \end{remark}

\section{Bilateral approximation and Reifenberg type sets}
\label{lsa}

In this section, we develop basic methods for studying the local geometry of a set $A\subseteq\RR^n$ via bilateral approximations.

\begin{definition}[Reifenberg type sets] \label{lapprox} Let $A\subseteq\RR^n$ be nonempty. \begin{enumerate}
 \item A \emph{local approximation class} $\cS$ is a nonempty collection of sets in $\CL{0}$ such that $\cS$ is a \emph{cone}; that is, for all $S\in \cS$ and $\lambda>0$, $\lambda S\in\cS$.
 \item For every $x\in\RR^n$ and $r>0$, define the \emph{(bilateral) approximability} $\Theta^{\cS}_A(x,r)$ of $A$ 
 by $\cS$ at location $x$ and scale $r$ by \begin{equation*}\Theta^{\cS}_A(x,r)=\inf_{S\in \cS}\mD{x}{r}[A,x+S]\in[0,\infty).\end{equation*}
 \item We say that $x\in A$ is an \emph{$\cS$ point of $A$} if $\lim_{r\downarrow 0} \Theta^\cS_A(x,r)=0$.
 \item We say $A$ is \emph{(bilaterally) $(\varepsilon,r_0)$-approximable by $\cS$} if $\Theta^\cS_A(x,r)\leq \varepsilon$ for all $x\in A$ and $0<r\leq r_0$.
\item We say $A$ is \emph{locally (bilaterally) $\varepsilon$-approximable by $\cS$} if for every compact set $K\subseteq A$ there exists $r_K$ such that $\Theta^\cS_A(x,r) \leq \varepsilon$ for all $x\in K$ and $0< r \leq r_K$.
\item We say $A$ is \emph{locally (bilaterally) well approximated by $\cS$} if $A$ is locally (bilaterally) $\varepsilon$-approximable by $\cS$ for all $\varepsilon>0$.
\end{enumerate}
\end{definition}

\begin{remark} \label{r:asymp_vs_local} Let $\varepsilon\geq 0$. A nonempty set $A\subseteq\RR^n$ is locally $\varepsilon'$-approximable by $\cS$ for all $\varepsilon'>\varepsilon$ if and only if $\limsup_{r\downarrow 0} \sup_{x\in K} \Theta^\cS_A(x,r) \leq \varepsilon$ for every compact set $K\subseteq A$. \end{remark}

We now collect some basic properties of approximability.

\begin{lemma}[size] \label{l:tsize} Let $\cS$ be a local approximation class. For all nonempty sets $A\subseteq\RR^n$, locations $x\in\RR^n$ and scales $r>0$, $$\frac{\dist(x,A)}{r}\leq \Theta^\cS_A(x,r) \leq 1+\frac{\dist(x,A)}{r}.$$ In particular,  $0\leq \Theta^\cS_A(x,r)\leq 1$ for all $x\in A$.\end{lemma}
\begin{proof} Let $A\subseteq\RR^n$ a nonempty set, $x\in\RR^n$ and $r>0$ be given. For the lower bound, simply note that $\Theta^{\cS}_A(x,r) \geq \inf_{S\in\cS} \mud{x}{r}(x+S,A)\geq \dist(x,A)/r$, since $0\in S$ for all $S\in\cS$. To verify the upper bound, pick any set $S\in\cS$. On one hand, for all $y\in (x+S)\cap \ball(x,r)$, $$\dist(y,A) \leq |y-x|+\dist(x,A) \leq r+\dist(x,A).$$ Hence $\mud{x}{r}(x+S,A)\leq 1+ \dist(x,A)/r$. On the other hand, for all $z\in A\cap \ball(x,r)$, $$\dist(z,x+S)\leq |z-x|\leq r.$$  Hence $\mud{x}{r}(A,x+S)\leq 1$. Therefore, $\Theta^{\cS}_A(x,r)\leq \mD{x}{r}[A,x+S]\leq 1+\dist(x,A)/r$.   \end{proof}

\begin{lemma}[scale and translation invariance]\label{l:theta_inv} Let $\cS$ be a local approximation class, let $A\subseteq\RR^n$ be nonempty, let $x\in\RR^n$ and let $r>0$. Then $$\Theta^\cS_{A}(x,r)=\Theta^\cS_{\lambda A}(\lambda x,\lambda r)\quad\text{for all }\lambda>0$$ and $$\Theta^\cS_{A}(x,r)=\Theta^\cS_{A+z}(x+z,r)\quad\text{for all }z\in\RR^n.$$\end{lemma}

\begin{proof} To check invariance under dilation, pick $\lambda>0$. Then \begin{align*} \Theta_A^{\cS}(x,r) = \inf_{S\in\cS} \mD{x}{r}[A,x+S]
 &=\inf_{S\in\cS}\mD{\lambda x}{\lambda r}[\lambda A,\lambda x+\lambda S]\\
 &=\inf_{S\in\cS}\mD{\lambda x}{\lambda r}[\lambda A,\lambda x+S]=\Theta_{\lambda A}^{\cS}(\lambda x,\lambda r),\end{align*} where the second equality holds by dilation invariance of the relative Walkup-Wets distance and the penultimate equality holds since $\cS$ is a cone. To verify invariance under translation, let $z\in\RR^n$. Then $$\Theta_A^{\cS}(x,r)=\inf_{S\in\cS}\mD{x}{r}[A,x+S]=\inf_{S\in\cS} \mD{x+z}{r}[A+z,x+z+S]=\Theta_{A+z}^{\cS}(x+z,r)$$ by translation invariance of the relative Walkup-Wets distance. \end{proof}

\begin{lemma}[closure] \label{l:close} Let $\cS$ be a local approximation class and let $A\subseteq\RR^n$ be nonempty. For all $x\in\RR^n$ and $r>0$, $\Theta_A^\cS(x,r) \leq \Theta_{\cl{A}}^\cS(x,r) \leq (1+\delta) \Theta_A^\cS(x,(1+\delta)r)$ for all $\delta>0$.\end{lemma}

\begin{proof} Recall that every set in a local approximation class is closed. Let $S,T\in\cS$ and let $\delta>0$. On one hand, the closure property of the relative Walkup-Wets distance gives $$\Theta_A^\cS(x,r) \leq \mD{x}{r}[A,x+S] \leq \mD{x}{r}[\cl{A},\cl{x+S}]=\mD{x}{r}[\cl{A},x+S].$$ On the other hand, the closure property of the relative Walkup-Wets distance gives $$\Theta_{\cl{A}}^\cS(x,r) \leq \mD{x}{r}[\cl{A},x+T]=\mD{x}{r}[\cl{A},\cl{x+T}] \leq (1+\delta) \mD{x}{(1+\delta)r}[A,x+T].$$ Taking the infimum over all $S,T\in\cS$, we conclude that $$\Theta_A^\cS(x,r)\leq \Theta_{\cl{A}}^\cS(x,r) \leq (1+\delta)\Theta_{A}^\cS(x,(1+\delta)r)$$ for all $\delta>0$.\end{proof}

\begin{lemma}[monotonicity] \label{l:tmono} Let $\cS$ be a local approximation class and let $A\subseteq\RR^n$ be a nonempty set. If $\ball(x,r)\subseteq \ball(y,s)$ and $|x-y|\leq ts$, then  \begin{equation}\label{e:m1}\Theta_A^{\cS}(x,r)\leq \frac{s}{r}\left(t+(1+t)\Theta_A^{\cS}(y,(1+t)s)\right).\end{equation} In particular, if $0<r\leq s$,  then \begin{equation}\label{e:m2}\Theta_A^\cS(x,r)\leq \frac{s}{r}\Theta_A^\cS(x,s).\end{equation} \end{lemma}
\begin{proof} Suppose that $\ball(x,r)\subseteq \ball(y,s)$ and $|x-y|\leq ts$. Let $S\in\cS$ be fixed and write $\rho=\mud{y}{s}(A,y+S)$. Since $\mud{y}{s}(x+S,y+S)\leq t$, the strong quasitriangle inequality implies \begin{equation*}\begin{split} \mD{y}{s}[A,x+S] &\leq (1+\rho)\mD{y}{(1+\rho)s}[x+S,y+S]+(1+t)\mD{y}{(1+t)s}[A,y+S]\\
&\leq t + (1+t)\mD{y}{(1+t)s}[A,y+S].\end{split}\end{equation*} Thus, by monotonicity, \begin{equation*}\Theta_A^\cS(x,r)\leq \mD{x}{r}[A,x+S] \leq \frac{s}{r}\mD{y}{s}[A,x+S] \leq \frac{s}{r}\left(t+(1+t)\mD{y}{(1+t)s}[A,y+S]\right).\end{equation*} Taking the infimum over $S\in \cS$ yields (\ref{e:m1}).\end{proof}

\begin{lemma}[limits] \label{ApproxVsLimits} Let $\cS$ be a local approximation class and let $A,A_1,A_2,\dots\in\CL{\RR^n}$. If $A_i\rightarrow A$ in $\CL{\RR^n}$, then for all $x\in \RR^n$ and $r>0$, \begin{equation}\begin{split}\label{ApproxVsLimits.1}\frac{1}{1+\varepsilon} \limsup_{i\to\infty}\,&\Theta^\cS_{A_i}\left(x,\frac{r}{1+\varepsilon}\right)\\ \leq &\,\Theta^\cS_{A}(x,r) \leq
(1+\varepsilon) \liminf_{i\to\infty}\Theta^\cS_{A_i}(x,r(1+\varepsilon))\quad\text{for all }\varepsilon>0\end{split}\end{equation}  and
 \begin{equation}\label{ApproxVsLimits.2} \limsup_{\varepsilon\downarrow 0} \limsup_{i\to \infty}\Theta^\cS_{A_i}\left(x,\frac{r}{1+\varepsilon}\right) \leq \Theta^\cS_A(x,r) \leq \liminf_{\varepsilon\downarrow 0}\liminf_{i\to\infty}\Theta_{A_i}^{\cS}(x,r(1+\varepsilon)).\end{equation}\end{lemma}
\begin{proof} Let $\varepsilon>0$ be given. For the lower bound in (\ref{ApproxVsLimits.1}), fix $T\in\cS$. Then, by Lemma \ref{QtriangleVsLimits}, $$\frac{1}{1+\varepsilon}\limsup_{i\rightarrow\infty} \Theta_{A_i}^\cS\left(x,\frac{r}{1+\varepsilon}\right) \leq \frac{1}{1+\varepsilon}\limsup_{i\rightarrow\infty}\mD{x}{r/(1+\varepsilon)}[A_i,x+T]\leq \mD{x}{r}[A,x+T].$$ Taking the infimum over all $T\in\cS$, we immediately obtain $$\frac{1}{1+\varepsilon}\limsup_{i\rightarrow\infty} \Theta_{A_i}^\cS\left(x,\frac{r}{1+\varepsilon}\right) \leq \Theta_A^\cS(x,r).$$ For the upper bound in (\ref{ApproxVsLimits.1}), choose a subsequence $(A_{ij})_{j=1}^\infty$ of $(A_i)_{i=1}^\infty$ such that $$\liminf_{i\rightarrow\infty} \Theta^\cS_{A_i}(x,r(1+\varepsilon)) = \lim_{j\rightarrow\infty}\Theta^\cS_{A_{ij}}(x,r(1+\varepsilon)).$$ Also choose a sequence $(S_j)_{j=1}^\infty$ of sets in $\cS$ such that \begin{equation}\label{e:avl1}\mD{x}{r(1+\varepsilon)}[A_{ij},x+S_j] \leq \Theta_{A_{ij}}^\cS(x,r(1+\varepsilon))+\frac{1}{j}\quad\text{for all }j\geq 1.\end{equation} Since $\CL{0}$ is sequentially compact, there exist a subsequence $(S_{jk})_{k=1}^\infty$ of $(S_j)_{j=1}^\infty$ and a set $S\in\CL{0}$ such that $S_{jk}\rightarrow S$ in $\CL{\RR^n}$. Let $L=\mud{x}{2r}(x+S,A)$. Then, for all $k\geq 1$, \begin{align*}\mud{x}{r}&(x+S_{jk},A_{ijk})
\leq \mud{x}{r}(x+S_{jk},x+S)+2\mud{x}{2r}(x+S,A_{ijk})\\ &\leq \mud{x}{r}(x+S_{jk},x+S)+2\mud{x}{2r}(x+S,A)+2(1+L)\mud{x}{2(1+L)r}(A,A_{ijk})\end{align*} by the weak quasitriangle inequality and strong quasitriangle inequality for relative excess, respectively. In particular, since $S_{jk}\rightarrow S$ in $\CL{\RR^n}$ and $A_{ijk}\rightarrow A$ in $\CL{\RR^n}$, $$M=\sup_k \mud{x}{r}(x+S_{jk},A_{ijk})<\infty.$$
Thus, for all $k$ sufficiently large such that $\mud{x}{r}(A,A_{ijk})\leq \varepsilon$, \begin{align*}\Theta_A^{\cS}(x,r) &\leq \mD{x}{r}[A,x+S_{jk}] \leq (1+M)\mD{x}{r(1+M)}[A,A_{ijk}]+(1+\varepsilon)\mD{x}{r(1+\varepsilon)}[A_{ijk},x+S_{jk}]\\
&\leq (1+M)\mD{x}{r(1+M)}[A,A_{ijk}]+(1+\varepsilon)\left(\Theta_{A_{ijk}}(x,r(1+\varepsilon))+\frac{1}{k}\right) \end{align*} by the strong quasitriangle inequality and (\ref{e:avl1}). It follows that $$\Theta_A^\cS(x,r) \leq (1+\varepsilon)\lim_{k\rightarrow\infty}\Theta_{A_{ijk}}(x,r(1+\varepsilon)) = (1+\varepsilon)\liminf_{i\rightarrow\infty} \Theta_{A_i}(x,r(1+\varepsilon)).$$ This establishes (\ref{ApproxVsLimits.1}). Finally, (\ref{ApproxVsLimits.2}) can be derived from (\ref{ApproxVsLimits.1}) using the same argument used to derive (\ref{QtriangleVsLimits.2}) from (\ref{QtriangleVsLimits.1}). \end{proof}

Given a local approximation class $\cS$, let $\ccS$ denote the closure of $\cS$ in $\CL{0}$. We now introduce two closely related families of local approximation classes.

\begin{definition}[bilateral $\varepsilon$-enlargements] Let $\cS$ be a local approximation class. For all $\varepsilon\geq 0$, define  $$(\cS;\varepsilon)^\Theta_{0,\infty}=\{\hat S\in\CL{0}: \Theta_{\hat S}^{\cS}(0,r)\leq\varepsilon\text{ for all }r>0\}$$ and $$(\cS;\varepsilon)^\Theta_{\RR^n,\infty}=\{\hat S\in\CL{0}: \Theta_{\hat S}^{\cS}(x,r)\leq\varepsilon\text{ for all }x\in \hat S\text{ and all }r>0\}.$$
\end{definition}

\begin{definition} A local approximation class $\cS$ is \emph{translation invariant} if for all $S\in\cS$ and $x\in S$, the translate $S-x\in \cS$. \end{definition}

\begin{lemma} \label{l:enlarge} Let $\cS$ be a local approximation class and let $\varepsilon\geq 0$. Then $(\cS;\varepsilon)^\Theta_{\RR^n,\infty}$ and $(\cS;\varepsilon)^\Theta_{0,\infty}$ are local approximation classes. Moreover, \begin{itemize}
\item $(\cS;\varepsilon)^\Theta_{\RR^n,\infty}$ and $(\cS;\varepsilon)^\Theta_{0,\infty}$ are closed in $\CL{0}$ and $(\cS;0)^\Theta_{0,\infty}=\ccS$; and,
\item $(\cS;\varepsilon)^\Theta_{\RR^n,\infty}$ is the maximal translation invariant local approximation class that is contained in $(\cS;\varepsilon)^\Theta_{0,\infty}$.
\end{itemize}\end{lemma}

\begin{proof} The enlargements $(\cS;\varepsilon)^\Theta_{\RR^n,\infty}$ and $(\cS;\varepsilon)^\Theta_{0,\infty}$ are local approximation classes, because approximability $\Theta^\cS_A(x,r)$ is scale invariant, i.e.~ $\Theta^{\cS}_{\lambda A}(\lambda x,\lambda r)=\Theta^\cS_A(x,r)$ for all $\lambda>0$.

To show that $(\cS;\varepsilon)^\Theta_{0,\infty}$ is closed in $\CL{0}$, suppose that $A_i\rightarrow A$ in $\CL{0}$ for some sequence $(A_i)_{i=1}^\infty$ in $(\cS;\varepsilon)^\Theta_{0,\infty}$. By Lemma \ref{ApproxVsLimits}, for all $r>0$, $$\Theta^\cS_A(0,r) \leq \liminf_{\delta\downarrow 0}\liminf_{i\rightarrow\infty} \Theta^\cS_{A_i}(0,(1+\delta)r)\leq \liminf_{\delta\downarrow 0}\liminf_{i\rightarrow\infty}\varepsilon=\varepsilon,$$ since each set $A_i\in(\cS;\varepsilon)^\Theta_{0,\infty}$. Hence $A\in (\cS;\varepsilon)^\Theta_{0,\infty}$, and thus,  $(\cS;\varepsilon)^\Theta_{0,\infty}$ is closed in $\CL{0}$.

In the special case $\varepsilon=0$, we have $\ccS\subseteq (\cS;0)^\Theta_{0,\infty}$, because $\cS\subseteq (\cS;0)^\Theta_{0,\infty}$ and $(\cS;0)^\Theta_{0,\infty}$ is closed. On the other hand, suppose that $A\in(\cS;0)^\Theta_{0,\infty}$. Then, for all $k\geq 1$, we can find $S_k\in\cS$ such that $\mD{0}{k}[A,S_k]\leq 1/k^2$, because $\Theta^\cS_A(0,k)=0$. By monotonicity, $$\mD{0}{r}[A,S_k] \leq (k/r)\mD{0}{k}[A,S_k]\leq 1/rk$$ for all $r>0$ and for all $k\geq r$. It follows that $S_k\rightarrow A$ in $\CL{0}$. Hence $A\in \ccS$ and $(\cS;0)^\Theta_{0,\infty}\subseteq\cl{S}$. Therefore, $(\cS;0)^\Theta_{0,\infty}=\ccS$.

Clearly $(\cS;\varepsilon)^\Theta_{\RR^n,\infty}\subseteq (\cS;\varepsilon)^\Theta_{0,\infty}$ and $(\cS;\varepsilon)^\Theta_{\RR^n,\infty}$ is translation invariant. Suppose that $\cT\subseteq (\cS;\varepsilon)^\Theta_{0,\infty}$ and $\cT$ is translation invariant. Let $T\in\cT$ and let $x\in \cT$. Then $$\Theta^\cS_T(x,r)=\Theta^{\cS}_{T-x}(0,r)\leq \varepsilon\quad\text{for all }r>0,$$ because $T-x\in \cT\subseteq (S;\varepsilon)^\Theta_{0,\infty}$. Hence $T\in(\cS;\varepsilon)^\Theta_{\RR^n,\infty}$ for all $T\in\cT$. Thus $\cT\subseteq(\cS;\varepsilon)^\Theta_{\RR^n,\infty}$ for all translation invariant $\cT\subseteq(\cS;\varepsilon)^\Theta_{0,\infty}$. Therefore, $(\cS;\varepsilon)^\Theta_{\RR^n,\infty}$ is the maximal translation invariant local approximation class contained in $(\cS;\varepsilon)^\Theta_{0,\infty}$.

To show that $(\cS;\varepsilon)^\Theta_{\RR^n,\infty}$ is closed in $\CL{0}$, it suffices by the previous paragraph to check that $\overline{(\cS;\varepsilon)^\Theta_{\RR^n,\infty}}$ is translation invariant. To that end, let $A\in \overline{(\cS;\varepsilon)^\Theta_{\RR^n,\infty}}$ and fix $x\in A$. By the sequential definition of closure, there exists a sequence $(A_i)_{i=1}^\infty$ of sets in $(\cS;\varepsilon)^\Theta_{\RR^n,\infty}$ such that $A_i\rightarrow A$. Choose points $x_i\in A_i$ for all $i\geq 1$ so that $x_i\rightarrow x$. Then each $A_i-x_i\in (\cS;\varepsilon)^\Theta_{\RR^n,\infty}$ by translation invariance of $(\cS;\varepsilon)^\Theta_{\RR^n,\infty}$. We claim that $A_i-x_i\rightarrow A-x$. Indeed, for any $r>0$, \begin{align*} \mD{-x}{r}[A-x,A_i-x_i] &\leq 2\mD{-x}{2r}[A-x,A_i-x]+2\mD{-x}{2r}[A_i-x,A_i-x_i]\\ &= 2\mD{0}{2r}[A,A_i]+2\mD{-x}{2r}[A_i-x,A_i-x_i]\\ &\leq 2\mD{0}{2r}[A,A_i]+\frac{|x-x_i|}{r},\end{align*} where the first inequality holds because $-x\in A_i-x$. Hence, since $A_i\rightarrow A$ and $x_i\rightarrow x$, $$\limsup_{i\rightarrow\infty} \mD{-x}{r}[A-x,A_i,x_i] =0\quad\text{for all }r>0.$$ Thus, $A_i-x_i\rightarrow A-x$ and $A-x\in \overline{(\cS;\varepsilon)^\Theta_{\RR^n,\infty}}$. Therefore, $(\cS;\varepsilon)^\Theta_{\RR^n,\infty}$ is closed. \end{proof}

The following theorem connects the $\cS$ points of a set $A$ with the tangent sets of $\cl{A}$.

\begin{theorem}\label{l:tan-perturb} Let $A\subseteq\RR^n$, let $x\in A$, and let $\varepsilon\geq 0$. Then $\limsup_{r\downarrow 0} \Theta^{\cS}_A(x,r)\leq \varepsilon$ if and only if $\Tan(\cl{A},x)\subseteq (\cS;\varepsilon)^\Theta_{0,\infty}$.\end{theorem}

\begin{corollary} Let $A\subseteq\RR^n$ and let $x\in A$. \label{TanVsTheta} Then $x$ is an $\cS$ point of $A$ if and only if  $\Tan(\cl{A},x)\subseteq{\ccS}$.
\end{corollary}

\begin{proof}[Proof of Theorem \ref{l:tan-perturb}] Fix a nonempty set $A\subseteq\RR^n$, $x\in A$, and $\varepsilon\geq 0$. Suppose $\limsup_{r\downarrow 0} \Theta^\cS_A(x,r)\leq \varepsilon$ and let $T\in \Tan(\cl{A},x)$, say $T=\lim_{i\rightarrow\infty} (\cl{A}-x)/r_i$ for some sequence $r_i\rightarrow 0$. Choose any scale $r>0$  and a small error $\delta>0$. By Lemma \ref{ApproxVsLimits}, Lemma \ref{l:theta_inv}, and Lemma \ref{l:close}, \begin{align*}\Theta_T^{\cS}(0,r) &\leq (1+\delta)\liminf_{i\rightarrow\infty} \Theta_{(\cl{A}-x)/r_i}^\cS(0,(1+\delta)r)= (1+\delta)\liminf_{i\rightarrow\infty} \Theta_{\cl{A}}^\cS(x,(1+\delta)rr_i)\\
&\leq (1+\delta)^2\liminf_{i\rightarrow\infty} \Theta_A^\cS(x,(1+\delta)^2rr_i) \leq (1+\delta)^2\varepsilon,\end{align*} where the last inequality holds because $\limsup_{r\downarrow 0} \Theta^\cS_A(x,r)\leq \varepsilon$. Letting $\delta\rightarrow 0$, we see that $\Theta_T^\cS(0,r)\leq \varepsilon$ for all $r>0$. Therefore, $T\in (\cS;\varepsilon)^\Theta_{0,\infty}$ for all $T\in \Tan(\cl{A},x)$.

Conversely, suppose that $\Tan(\cl{A},x)\subseteq (S;\varepsilon)^\Theta_{0,\infty}$. Choose a sequence $r_i\rightarrow 0$ such that $$\lim_{i\rightarrow\infty}\Theta^\cS_A(x,r_i)= \limsup_{r\rightarrow 0} \Theta^\cS_A(x, r).$$ Since $\CL{0}$ is sequentially compact, we may assume (by passing to a subsequence) that $(\cl{A}-x)/r_i\to T$ for some $T\in \Tan(\cl{A},x)$. By Lemma \ref{l:close}, Lemma \ref{l:theta_inv}, and  Lemma \ref{ApproxVsLimits}, $$\lim_{i\rightarrow\infty} \Theta_A^\cS(x,r_i) \leq \limsup_{i\rightarrow\infty} \Theta_{\cl{A}}^\cS(x,r_i)= \limsup_{i\rightarrow\infty} \Theta_{(\cl{A}-x)/r_i}^\cS (0,1) \leq (1+\delta) \Theta_{T}^\cS(0,1+\delta)\leq (1+\delta)\varepsilon$$ for all $\delta>0$, where the last inequality holds since $T\in(\cS;\varepsilon)^\Theta_{0,\infty}$. Letting $\delta\rightarrow 0$, we obtain $\limsup_{r\rightarrow 0} \Theta_A^\cS(x,r) = \lim_{i\rightarrow\infty} \Theta_A^\cS(x,r_i)\leq \varepsilon$.
\end{proof}

\begin{corollary}\label{TanVsTheta0} Let $A\subseteq\RR^n$, let $x\in A$, and let $\varepsilon\geq 0$. \begin{itemize}
\item If $\Tan(\cl{A},x)\cap (\cS;\varepsilon)^\Theta_{0,\infty}\neq\emptyset$, then $\liminf_{r\rightarrow 0} \Theta^\cS_A(x,r)\leq \varepsilon$.
\item If $\liminf_{r\rightarrow 0} \Theta^{\cS}_A(x,r)=0$, then $\Tan(\cl{A},x)\cap \ccS \neq\emptyset$.
\end{itemize}
\end{corollary}

\begin{proof} Suppose that $T\in \Tan(\cl{A},x)\cap (\cS;\varepsilon)^\Theta_{0,\infty}$, say $T=\lim_{i\rightarrow\infty} (\cl{A}-x)/r_i$ for some $r_i\rightarrow 0$. Then, as in the proof of Theorem \ref{l:tan-perturb}, $$\limsup_{i\rightarrow\infty}\Theta_A^\cS(x,r_i) \leq \limsup_{i\rightarrow\infty} \Theta_{(\cl{A}-x)/r_i}^\cS(0,1)\leq
(1+\delta)\Theta_T^\cS(0,1+\delta)\leq (1+\delta)\varepsilon,$$ where the last inequality holds since $T\in (\cS;\varepsilon)^\Theta_{0,\infty}$. Letting $\delta\rightarrow 0$, we conclude that $\liminf_{r\rightarrow 0} \Theta_A^\cS(x,r)\leq \limsup_{i\rightarrow\infty} \Theta_A^\cS(x,r_i) \leq \varepsilon$.

For the second statement, suppose that $x\in A\subseteq\RR^n$ and $\liminf_{r\rightarrow 0} \Theta^\cS_A(x,r)=0$. Then $\lim_{i\rightarrow \infty}\Theta^\cS_A(x,r_i)=0$ for some sequence $r_i\rightarrow 0$. Passing to a subsequence, we may assume that $(\cl{A}-x)/r_i\rightarrow B$ for some $B\in \Tan(\cl{A},x)$. For all $0<r\leq 1/4$,  \begin{align*}\Theta_B^{\cS}(0,r) &\leq 2\liminf_{i\rightarrow\infty} \Theta_{(\cl{A}-x)/r_i}^\cS(0,2r)\\
&= 2\liminf_{i\rightarrow\infty} \Theta_{\cl{A}}^\cS(x,2rr_i) \leq 4\liminf_{i\rightarrow\infty} \Theta_{A}^\cS(x,4rr_i)\leq \frac{1}{r}\liminf_{i\rightarrow\infty} \Theta^\cS_A(x,r_i)=0.\end{align*} by Lemma \ref{ApproxVsLimits}, Lemma \ref{l:theta_inv}, Lemma \ref{l:close}, and Lemma \ref{l:tmono}. Hence, by Corollary \ref{TanVsTheta}, $\Tan(B,0)\subseteq \ccS$. Pick any $C\in\Tan(B,0)$. Then $C\in \Tan(\cl{A},x)$ by Lemma \ref{l:iterate}. Therefore, $C\in \Tan(\cl{A},x)\cap \ccS$. In particular, $\Tan(\cl{A},x)\cap \ccS\neq\emptyset$.\end{proof}

The following theorem characterizes locally $\varepsilon$-approximable sets in terms of the pseudotangent sets of their closure.

\begin{theorem}\label{PsTanVsPerturb} Let $A\subseteq\RR^n$ be nonempty and let $\varepsilon\geq 0$. The following are equivalent: \begin{enumerate}
 \item $A$ is locally $\varepsilon'$-approximable by $\cS$ for all $\varepsilon' > \varepsilon$;
 \item $\PsTan(\cl{A},x)\subseteq(\cS;\varepsilon)^\Theta_{0,\infty}$ for all $x\in A$;
 \item $\PsTan(\cl{A},x)\subseteq(\cS;\varepsilon)^\Theta_{\RR^n,\infty}$ for all $x\in A$.
 \end{enumerate}
\end{theorem}

\begin{corollary}\label{PsTanVsApprox} Let $A\subseteq\RR^n$ be nonempty. Then $A$ is locally well approximated by $\cS$ if and only if $\PsTan(\cl{A},x)\subseteq\ccS$ for all $x\in A$.\end{corollary}

\begin{proof}[Proof of Theorem \ref{PsTanVsPerturb}] Fix $\varepsilon\geq 0$. On one hand, $(iii)\Rightarrow(ii)$ is immediate, because $(\cS;\varepsilon)^\Theta_{\RR^n,\infty}\subseteq(\cS;\varepsilon)^\Theta_{0,\infty}$. On the other hand, $(ii)\Rightarrow (iii)$, because $(\cS;\varepsilon)^\Theta_{\RR^n,\infty}$ is the maximal translation invariant local approximation class contained in $(\cS;\varepsilon)^\Theta_{0,\infty}$ by Lemma \ref{l:enlarge} and $\PsTan(\cl{A},x)$ is a translation invariant local approximation class by Lemma \ref{inv_of_tan}. To complete the proof, we will show that $(i)$ is equivalent to $(ii)$.

Suppose that $A$ is locally $\varepsilon'$-approximable by $\cS$ for all $\varepsilon'>\varepsilon$. Then, by Remark \ref{r:asymp_vs_local}, \begin{equation}\label{e:equiv1} \limsup_{r\downarrow 0} \sup_{x\in K} \Theta^\cS_A(x,r)\leq \varepsilon\end{equation} for every compact subset $K\subseteq A$. Let $T\in \PsTan(\cl{A},x)$ for some $x\in A$, say that $T=\lim_{i\rightarrow\infty} (\cl{A}-y_i)/r_i$ for some sequences $y_i\in \cl{A}$ and $r_i>0$ such that $y_i\to x$ and $r_i\rightarrow 0$. For each $i\geq 1$, choose $x_i\in A$ such that $|x_i-y_i|\leq r_i/i$. Then, since for all $s>0$, $$\mD{0}{s}\left[\frac{\cl{A}-x_i}{r_i},T\right] \leq \frac{1}{s}\left|\frac{x_i-y_i}{r_i}\right|+2\mD{0}{2s}\left[\frac{\cl{A}-y_i}{r_i},T\right],$$ we have that $(\cl{A}-x_i)/r_i\rightarrow T$ as well. Fix a scale $r>0$ and an error $\delta>0$. By Lemmas \ref{ApproxVsLimits}, \ref{l:theta_inv}, and \ref{l:close}, and by (\ref{e:equiv1}) applied with the compact set $K=\{x\}\cup \{x_i:i\geq 1\}\subseteq A$, \begin{align*}\Theta_T^{\cS}(0,r) &\leq (1+\delta)\liminf_{i\rightarrow\infty} \Theta_{(\cl{A}-x_i)/r_i}^\cS(0,(1+\delta)r)= (1+\delta)\liminf_{i\rightarrow\infty} \Theta_\cl{A}^\cS(x_i,(1+\delta)rr_i) \\ &\leq (1+\delta)^2\liminf_{i\rightarrow\infty} \Theta_A^\cS(x_i,(1+\delta)^2rr_i) \leq (1+\delta)^2\varepsilon.\end{align*} Letting $\delta\rightarrow 0$, we obtain $\Theta^{\cS}_T(0,r)\leq \varepsilon$ for all $r>0$. This shows that $T\in (S;\varepsilon)^\Theta_{0,\infty}$ and $\PsTan(\cl{A},x)\subseteq (\cS;\varepsilon)^\Theta_{0,\infty}$ for all $x\in A$. Therefore,    $(i)\Rightarrow(ii)$.

Finally, suppose that $\PsTan(\cl{A},x)\subseteq(S;\varepsilon)^\Theta_{0,\infty}$ for all $x\in A$.  Fix a compact set $K \subseteq A$. Let $x_i\in K$ and $r_i\rightarrow 0$ be sequences such that $$\lim_{i\rightarrow\infty} \Theta^\cS_A(x_i,r_i)= \limsup_{r\downarrow 0} \sup_{x\in K}\Theta^\cS_A(x, r).$$ Passing to a subsequence of $(x_i,r_i)_{i=1}^\infty$, we may assume (since $K$ and $\CL{0}$ are sequentially compact) that $x_i\rightarrow x$ for some $x\in K$ and $(\cl{A}-x_i)/r_i\rightarrow T$ for some $T\in\PsTan(\cl{A},x)$.  By Lemma \ref{l:close}, Lemma \ref{l:theta_inv} and  Lemma \ref{ApproxVsLimits}, \begin{align*}\lim_{i\rightarrow\infty} \Theta_A^\cS(x_i,r_i) &\leq \limsup_{i\rightarrow\infty} \Theta_{\cl{A}}^\cS(x_i,r_i) \\
&=\limsup_{i\rightarrow\infty} \Theta_{(\cl{A}-x_i)/r_i}^\cS (0,1) \leq (1+\delta) \Theta_{T}^\cS(0,1+\delta)\leq (1+\delta)\varepsilon\end{align*} for all $\delta>0$, where the last inequality holds since $T\in(\cS;\varepsilon)^\Theta_{0,\infty}$. Letting $\delta\rightarrow 0$, we obtain $\limsup_{r\rightarrow 0}\sup_{x\in K} \Theta_A^\cS(x,r) = \lim_{i\rightarrow\infty} \Theta_A^\cS(x_i,r_i)\leq \varepsilon$. Thus, by Remark \ref{r:asymp_vs_local}, $A$ is locally $\varepsilon'$-approximable by $\cS$ for all $\varepsilon'>\varepsilon$. Therefore, $(ii)\Rightarrow (i)$.
\end{proof}

\begin{remark} In Definition \ref{lapprox}, we defined sets locally $\varepsilon$-approximable by $\cS$ \emph{intrinsically} with uniform bounds on approximability in each compact subset $K\subseteq A$. Alternatively, one could define sets locally $\varepsilon$-approximable by $\cS$ \emph{extrinsically} with uniform bounds on approximability  in $K\cap A$ for each compact set $K\subset\RR^n$. The two definitions agree when $A$ is closed. The analogue of Theorem \ref{PsTanVsPerturb} in the extrinsic case reads as follows:

\emph{For all $\varepsilon'>\varepsilon$ and for all compact sets $K\subset\RR^n$ there exists $r_0>0$ such that $\Theta_A^\cS(x,r)\leq \varepsilon'$ for all $x\in A\cap K$ and for all $0<r\leq r_0$ if and only if $\PsTan(\cl{A},x)\subseteq (\cS;\varepsilon)^\Theta_{0,\infty}$ for all $x\in \cl{A}$.}
\end{remark}

\begin{remark}\label{theta_justification} We view Theorems \ref{l:tan-perturb} and \ref{PsTanVsPerturb} as validations that defining the bilateral approximability of a set using relative Walkup-Wets distance is the ``correct" approach. Analogous statements using the relative Hausdorff distance instead of the relative Walkup-Wets distance fail for general local approximation classes. For example, consider $$S =\{0\}\cup \bigcup_{i\in \mathbb{Z}} \partial \ball(0, 2^i)\subset \RR^n\quad\text{and}\quad \mathcal {S} = \{\lambda S: \lambda > 0\}.$$ Then $\cS$ is a local approximation class, and moreover, $\cS$ is closed in $\CL{0}$. Let $r_i = 2^{-i} - 3^{-i}$ for all $i\geq 1$ and let $e$ be a unit vector. Construct $A\in\CL{0}$ by adding extra points to $S$: $$A = S\cup \{ r_i e \mid i \geq 1\}.$$ Intuitively, it is clear that tangent sets of $A$ at $0$ belong to $\cS$, because the extra point $r_ie$ added to $S$ at scale $2^{-i}$ to form $A$ become relatively closer and closer to $S$ as $i\rightarrow\infty$. Indeed, since $\Theta^S_A(0,r)\to 0$ as $r\to 0$, we know $\Tan(A,0) \subseteq \cS$ by Corollary \ref{TanVsTheta}. However, for the Hausdorff distance analogue of bilateral approximability, one can show that $$\inf_{S'\in \cS} \D{0}{r_i} [S', A] \geq 1/4\quad\text{for all }i\gg 1.$$ Therefore, there is no analogue of Theorem \ref{l:tan-perturb} for the relative Hausdorff distance.
\end{remark}

\section{Connectedness of the cone of tangent sets at a point}\label{s:connect}

Our goal in this section is to prove that the cone of tangent sets at a point is connected in a certain sense. This result (Theorem \ref{bcthm}) is motivated by an analogous statement for tangent measures established by Preiss \cite{Preiss}*{Theorem 2.6}. Also, see Kenig, Preiss and Toro \cite{KPT}*{Theorem 2.1 and Corollary 2.1}.

\begin{definition}[singular class] Let $\cT$ be a local approximation class. We define the \emph{(bilateral) singular class of $\cT$} to be the local approximation class $\cTp$ given by  $$\cTp=\{C\in\CL{0}: \liminf_{r\downarrow 0}\Theta^\cT_C(0,r)>0\}=\{C\in\CL{0}:\Tan(C,0)\cap\ccT=\emptyset\}.$$\end{definition}

\begin{definition}[separation at infinity] Let $\cT$ and $\cS$ be local approximation classes. We say that \emph{$\cT$ is (bilaterally) separated at infinity in $\cS$} if there exists $\phi>0$ such that $$\limsup_{r\uparrow\infty} \Theta^{\cT}_S(0,r)\geq \phi\quad\text{for all }S\in\cSccT.$$ To emphasize a choice of some $\phi>0$, we may say that \emph{$\cT$ is $\phi$ separated at infinity in $\cS$}.
\end{definition}

\begin{theorem}[connectedness of the cone of tangent sets at a point]  Suppose that $\cT$ is separated at  \label{bcthm} infinity in $\cS$. If $\Tan(A,x)\subseteq\cS$ for some $A\in\CL{x}$, then either $\Tan(A,x)\subseteq\ccT$ or $\Tan(A,x)\subseteq \cTp$. \end{theorem}

\begin{proof} Suppose that $\cT$ is $\phi$ separated at infinity in $\cS$ for some $0<\phi\leq 1$, so that \begin{equation}\label{bcthm.1}\inf_{S\in \cSccT} \limsup_{r\to\infty} \Theta^\cT_S(0,r)\geq \phi.\end{equation} Let $A\in\CL{x}$ for some $x\in\RR^n$. The conclusion of the theorem follows immediately from two claims: \begin{equation}\label{e:t2t} \Tan(A,x)\subseteq\cS\text{ and }\Tan(A,x)\cap\ccT\neq\emptyset\Longrightarrow \Tan(A,x)\subseteq\ccT\end{equation} and
 \begin{equation}\label{e:tperp} \Tan(A,x)\subseteq\cSccT\Longrightarrow \Tan(A,x)\subseteq\cTp.\end{equation}

To start, let's prove (\ref{e:tperp}). Suppose for contradiction that $\Tan(A,x)\subseteq\cSccT$, but there exists $B\in \Tan(A,x)\setminus \cTp$. First off, $\liminf_{r\rightarrow 0}\Theta^\cT_B(0,r)=0$, since $B\not\in\cTp$. Thus, by Corollary \ref{TanVsTheta0}, there exists $C\in \Tan(B,0)\cap \ccT$. However, by Lemma \ref{l:iterate}, $C\in \Tan(A,x)\subseteq\cSccT$. We have asserted that there exists $C \in \ccT$ such that $C\in \cSccT$, which is absurd. Therefore, if $\Tan(A,x)\subseteq\cSccT$, then $\Tan(A,x)\subseteq \cTp$.

It remains to verify (\ref{e:t2t}). Suppose for contradiction that $\Tan(A,x)\subseteq\cS$ and $\Tan(A,x)\cap \ccT\neq \emptyset$, but  $\Tan(A,x)\cap\cSccT\neq \emptyset$. Choose tangent sets $T\in \Tan(A,x)\cap \ccT$ and  $R\in \Tan(A,x) \cap \cSccT$. Then there exist sequences $t_i\downarrow 0$ and $r_i\downarrow 0$ such that  $(A-x)/t_i\to T$ and $(A-x)/r_i\to R$. Passing to subsequences as needed, we may assume without loss of generality that $4r_i<t_i$ for all $i\geq 1$.

Now, since $R\in\cSccT$, there exists some scale $c>0$ such that $\Theta^\cT_R(0,c)\geq \phi/2$ by (\ref{bcthm.1}). Let us abbreviate \begin{equation*}\Theta^{\cT}_A(x,tc)=\Theta^{\cT}_{(A-x)/t}(0,c)=:\Theta(t)\quad\text{for all }t>0.\end{equation*} On one hand, since $(A-x)/t_i\rightarrow T$ and $T\in\ccT$,  $$\limsup_{i\rightarrow\infty}\Theta(t_i)\leq 2\Theta^\cT_T(0,2c)=0$$ by Lemma \ref{ApproxVsLimits}.  On the other hand, since $(A-x)/r_i\rightarrow R$ and $\Theta^\cT_R(0,c)\geq \phi/2$, $$\liminf_{i\rightarrow\infty} \Theta(2r_i)=\liminf_{i\rightarrow\infty}\Theta^\cT_{(A-x)/r_i}(0,2c)\geq \frac{1}{2} \Theta_R^\cT(0,c)\geq \frac{\phi}{4}$$ by Lemma \ref{l:theta_inv} and Lemma \ref{ApproxVsLimits}.

Note that $\Theta(t)\in[0,1]$ for all $t>0$ by Lemma \ref{l:tsize}. In fact, $\Theta(t)\in(0,1]$ for all $t>0$. (Otherwise, if $\Theta(t_0)=0$ for some $t_0>0$, then $\Theta_A(x,tc)\leq (t_0/t)\Theta_A(x,t_0c)=0$ for all $t\leq t_0$ by monotonicity, forcing $\Tan(A,x)\subseteq\ccT$ by Corollary \ref{TanVsTheta}.) For all $j\geq 1$, set
$$I_j = (2^{-j}, 2^{-j+1}],\quad I_j^+ = \bigcup_{k=j}^\infty I_k = (0,2^{-j+1}],\quad\text{and }I_j^-=\bigcup_{k=1}^j I_k=(2^{-j},1].$$ For all $t>0$, there exists a unique $j\geq 1$ with $\Theta(t)\in I_j$. By monotonicity, we obtain the \emph{weak jump bound}:
\begin{equation}\label{WJB}
\Theta(t)\in I_k\Longrightarrow \Theta(\alpha t)\in I^+_{k-1}\text{ for all }\alpha\in [1/2,1]\text{ and }t >0.\end{equation}
 Let $p\geq 1$ be the unique integer such that $\phi \in I_p$.

We now aim to construct a sequence $s_i\to 0$ such that $(A-x)/s_i$ converges to some tangent set of $A$ at $x$ that lies in neither $\ccT$ nor $\cSccT$.

\begin{claim}For all $i\gg 1$, there exists $s_i\in (2 r_i, t_i)$ such that $\Theta(s_i)\in I_{p+3}$, $\Theta(t)\in I^+_{p+3}$ for all $t\in [s_i, t_i]$  and $\lim_{i\rightarrow\infty} t_i/ s_i=\infty$.\end{claim}
\begin{proof}[Proof of Claim] Choose $i_0$ large enough so that for $i\geq i_0$, $\Theta(2r_i)\geq\phi/4$ and $\Theta(t_i)\leq \phi/16$. Let $k_i$ be the unique integer such that $\Theta(t_i)\in I_{k_i}$. Note that $\Theta(t_i)\in I_{p+4}^+$ (since $\phi\in I_p$). Fix $i\geq i_0$ and consider the sequence $2^{-m} t_i$ for $m\geq 0$. Because $4r_i<t_i$, there exists an integer $M\geq 1$ such that $2^{-M-1}t_i \leq 2 r_i < 2^{-M} t_i$. By monotonicity, $$\frac{\phi}{4}\leq \Theta(2r_i)\leq 2\,\Theta(2^{-M}t_i).$$ Hence $\phi/8 \leq \Theta(2^{-M}t_i)\in I_{p+3}^-$. Therefore, by the weak jump bound (\ref{WJB}), there exists at least one integer $m\in \{1,\ldots, M\}$ such that $\Theta(2^{-m} t_i) \in I_{p+3}$. Define $m_i$ to be the smallest integer with this property and put $s_i := 2^{-m_i} t_i$.

By construction, $s_i\in(2r_i,t_i)$ and $\Theta(s_i)\in I_{p+3}$. Suppose $t\in (s_i, t_i]$. Then there exists an integer $m$ with $0\leq m<m_i$ such that $t\in (2^{-(m+1)}t_i, 2^{-m}t_i]$. By the minimality of $m_i$, $\Theta(2^{-m}t_i)\in I_{p+4}^+$. Hence, by the weak jump bound, $\Theta(t)\in I_{p+3}^+$ for all $t\in(s_i,t_i]$.

Finally, we note that $t_i/s_i = 2^{m_i}$. By the weak jump bound, $m_i\geq k_i - (p+3)$. Because $\Theta(t_i)\to 0$, we  have $k_i\to \infty$ as $i\rightarrow\infty$. Therefore, $t_i/s_i\to \infty$ as $i\rightarrow\infty$.\end{proof}

We now return to the proof of (\ref{e:t2t}). Passing to a subsequence, we may assume that $(A-x)/s_i\to S$ for some tangent set $S\in \Tan(A,x)\subseteq\cS$.
By the claim, for all $i\gg 1$, $\Theta^\cT_{(A-x)/s_i}(0,c)=\Theta(s_i)\in (2^{-(p+3)}, 2^{-(p+2)}]$. Hence, by Lemma \ref{ApproxVsLimits}, \begin{equation*} \Theta^{\cT}_S(0,2c) \geq \frac{1}{2}\limsup_{i\rightarrow\infty} \Theta_{(A-x)/s_i}^\cT(0,c)\geq 2^{-(p+4)}.\end{equation*} It follows that $S\not\in\ccT$. On the other hand, by Lemma \ref{ApproxVsLimits},  \begin{equation*} \Theta^{\cT}_S(0,r)\leq 2\liminf_{i\rightarrow\infty} \Theta^\cT_{(A-x)/s_i}(0,2r) =2\liminf_{i\rightarrow\infty}\Theta((2r/c)s_i)\leq 2^{-(p+1)}\end{equation*} for all $r\geq c/2$, since $\Theta(t)\in (0,2^{-(p+2)}]$ for all $t\in[s_i,t_i]$ and $\liminf_{i\rightarrow\infty} t_i/s_i>2 r/c$. Thus, $S\in\cSccT$, but $\limsup_{r\rightarrow\infty}\Theta^{\cT}_S(0,r)\leq 2^{-(p+1)}<\phi$. This violates (\ref{bcthm.1}). Therefore, if $\Tan(A,x)\subseteq\cS$ and $\Tan(A,x)\cap\ccT\neq\emptyset$, then $\Tan(A,x)\subseteq\ccT$.
\end{proof}

\begin{corollary} \label{c:bc} Let $\cT$ and $\cS$ be local approximation classes. If $\cT$ is separated at  infinity in $\cS$ and $\ccT$ is translation invariant, then for all $A\in\CL{x}$, \begin{equation*}\bPsTan(A,x)\subseteq\cS\text{ and }\bPsTan(A,x)\cap \ccT\neq\emptyset\Longrightarrow \bPsTan(A,x)\subseteq\ccT.\end{equation*} \end{corollary}

\begin{proof} Assume $\cT$ is separated at infinity in $\cS$ and $\ccT$ is translation invariant. Suppose $\bPsTan(A,x)\subseteq\cS$ and $\bPsTan(A,x)\cap \ccT\neq\emptyset$. Choose $B\in\bPsTan(A,x)\cap\ccT$. By Lemma \ref{l:bdtan}, there exists $C\in \Tan(A,x)$ and $y\in C$ such that $B=C-y$. Equivalently, $C=B-(-y)$ for some $-y\in B$. Hence $C\in\ccT$, because  $\ccT$ is translation invariant. Thus, $\Tan(A,x)\subseteq\bPsTan(A,x)\subseteq\cS$ and $\Tan(A,x)\cap\ccT\neq\emptyset$. Because $\cT$ is separated at infinity in $\cS$, we conclude that $\Tan(A,x)\subseteq\ccT$ by Theorem \ref{bcthm}. Therefore, $\bPsTan(A,x)\subseteq\ccT$, because every bounded pseudotangent set is the translate of a tangent set (Lemma \ref{l:bdtan}) and $\ccT$ is translation invariant.\end{proof}

\begin{remark}\label{r:ubad} The analogue of Theorem \ref{bcthm} and Corollary \ref{c:bc} for arbitrary (unbounded) pseudotangent sets is false. For example, consider $$A=\{(x,y)\in\RR^2:x=0\text{ or }y=0\},$$ which is the union of the $x$-axis $X=\{(x,0):x\in\RR\}$ and the $y$-axis $Y=\{(0,y):y\in\RR\}$ in the plane. Let $\cS$ consist of $X$, $Y$, and all translates of $A$. One can readily check that $\cS$ is a closed, translation invariant local approximation class and $\Theta^\cS_A(x,r)=0$ for all $x\in A$ and $r>0$. Hence $\PsTan(A,x)\subseteq\cS$ by Corollary \ref{PsTanVsApprox}. It is also easy to check that $\cT:=\{X\}$ is separated at infinity in $\cS$ and $\cT$ is translation invariant. Let $x_i=(1/i,0)$ and $r_i=1/i^2$. Then $X=\lim_{i\rightarrow\infty}(A-x_i)/r_i$ is an unbounded psuedotangent of $A$ at the origin. Thus, $\PsTan(A,(0,0))\cap \cT\neq\emptyset$, but $\PsTan(A,(0,0))\not\subseteq \cT$.   \end{remark}

We end the section with several criteria for checking separation at infinity.

\begin{lemma} \label{l:infsep} Let $\cT$ and $\cS$ be local approximation classes and let $\phi>0$. If for all $S\in\cS$ there is a function $\Phi_S:(0,1)\rightarrow(0,\infty)$ with $\liminf_{s\rightarrow 0+} \Phi_S(s)=0$ such that \begin{equation*}\Theta^{\cT}_S(0,r)<\phi \Longrightarrow \Theta^{\cT}_S(0,sr)< \Phi_S(s)\text{ for all }s\in(0,1),\end{equation*} then $\mathcal{T}$ is $\phi$ separated at infinity in $\mathcal{S}$.
\end{lemma}

\begin{proof}
	Suppose for contradiction that there exists $S\in\cSccT$ and $r_0>0$ such that $\Theta^\cT_S(0,r) < \phi$ for all $r\geq r_0$. Because $S\not\in\ccT$, there exists $r_1>r_0$ such that $\Theta^\cT_S(0,r_1)>0$. Since $\liminf_{s\rightarrow 0+}\Phi_S(s)=0$, we can choose $s<1$ such that $\Phi(s) \leq \Theta^{\cT}_S(0,r_1)$. Hence $\Theta^\cT_S(0,r_1) = \Theta^\cT_S(0, s(r_1/s)) < \Phi(s) \leq \Theta^{\cT}_S(0,r_1)$, which is absurd.
\end{proof}

\begin{lemma} \label{l:infsep2} Let $\cT$ and $\cS$ be local approximation classes and let $\phi>0$. If for all $S\in\cS$ there is a function $\Phi_S:(0,1)\rightarrow(0,\infty)$ with $\lim_{s\rightarrow 0+} \Phi_S(s)=0$ such that \begin{equation*}\Theta^{\cT}_S(0,r)<\phi \Longrightarrow \Theta^{\cT}_S(0,sr)< \Phi_S(s)\text{ for all }s\in(0,1),\end{equation*} then $\inf_{S\in\cSccT}\liminf_{r\uparrow \infty}\Theta_S^{\cT}(0,r)\geq \phi$, and thus, $\mathcal{T}$ is $\phi$ separated at infinity in $\mathcal{S}$.
\end{lemma}

\begin{proof}Suppose for contradiction that there exists $S\in\cSccT$ and a sequence $r_i\rightarrow\infty$ such that $\Theta^{\cT}_S(0,r_i)<\phi$ for all $i\geq 1$. Because $S\not\in\ccT$, there exists $j\geq 1$ such that $\delta=\Theta^{\cT}_S(0,r_j)>0$. Since $\lim_{s\rightarrow 0+}\Phi_S(s)=0$, we can find $k\geq j$ such that $\Phi(r_j/r_k)\leq \delta$. Hence $\Theta^{\cT}_S(0,r_j)=\Theta^{\cT}_S(0,(r_j/r_k)r_k)<\Phi_S(r_j/r_k)\leq \Theta^{\cT}_S(0,r_j)$, which is absurd. \end{proof}

The following property, which we informally call ``detectability", is a uniform version of the criterion for separation at infinity in Lemma \ref{l:infsep}. In the next section, we shall see that where separation at infinity gives pointwise information about the tangents of sets, detectability yields locally uniform information about the tangents of sets.

\begin{definition}[$\cT$ point detection property] \label{d:Tpdp} Let $\cT$, $\cS$ be local approximation classes. We say that \emph{$\cT$ points are detectable in $\cS$} if there exist a constant $\phi>0$ and a function $\Phi:(0,1)\rightarrow(0,\infty)$ with $\liminf_{s\rightarrow 0+}\Phi(s)=0$ such that if $S\in\cS$ and $\Theta^\cT_S(0,r)<\phi$, then $\Theta^\cT_S(0,sr)<\Phi(s)$ for all $s\in(0,1)$. To emphasize a choice of $\phi$ and $\Phi$, we may say that \emph{$\cT$ points are $(\phi,\Phi)$ detectable in $\cS$}. \end{definition}

\begin{example} Given $n\geq 2$ and $d\geq 1$, let $\cH(n,d)$ denote the collection of zero sets of nonconstant harmonic polynomials $p:\RR^n\rightarrow\RR$ of degree at most $d$ such that $p(0)=0$. Then $\cG(n,n-1)$ points (``flat points") are $(\delta_{n,d}, C_{n,d}s)$ detectable in $\cH(n,d)$ by \cite{Badger3}*{Theorem 1.4}.\end{example}

Detectability implies separation at infinity in a stronger sense than in Lemma \ref{l:infsep}.

\begin{lemma} \label{l:infsep3} If $\cT$ points are $(\phi,\Phi)$ detectable in $\cS$, then $\cT$ is $\phi$ separated at infinity in $\ccS$.\end{lemma}

\begin{proof} Suppose for contradiction that $\cT$ points are $(\phi,\Phi)$ detectable in $\cS$, but there exists $S\in\ccS\setminus \ccT$  such that $\limsup_{r\rightarrow\infty} \Theta^\cT_S(0,r)<\phi$. Then there exist $\delta>0$ and $r_0>0$ such that $\Theta_S^\cT(0,r)< \phi/(1+\delta)$ for all $r\geq r_0$. Since $S\in\ccS\setminus\ccT$, there exists a sequence $S_i\in\cSccT$ with $S_i\rightarrow S$ in $\CL{0}$. Passing to a subsequence, we may assume that $\mD{0}{i}[S_i,S]\leq 1/i^2$ for all $i\geq 1$. Then, by monotonicity, $$\mD{0}{r}[S_i,S] \leq \frac{1}{ir}\quad\text{for all }i\geq r.$$ Because $S\not\in\cT$, there exists $r_1>r_0$ such that $\Theta_S(0,r_1)>0$. Pick any $s<1$ such that $\Phi(s)\leq \Theta_S(0,r_1)/4$. By the weak quasitriangle inequality, $$\Theta^\cT_S(0,r_1)\leq 2\Theta^\cT_{S_i}(0,2r_1)+\frac{1}{ir_1}
= 2\Theta_{S_i}\left(0,s\frac{2r_1}{s}\right)+\frac{1}{ir_1}\quad\text{for all }i\geq 2r_1.$$ Note that $$\limsup_{i\rightarrow\infty} \Theta^{\cT}_{S_i}\left(0,\frac{2r_1}{s}\right)
\leq (1+\delta)\Theta^{\cT}_S\left(0,(1+\delta)\frac{2r_1}{s}\right)<\phi$$ by Lemma \ref{ApproxVsLimits}. Thus, by the detectability hypothesis, $$\Theta^\cT_S(0,r_1)\leq 2\Theta_{S_i}\left(0,s\frac{2r_1}{s}\right)+\frac{1}{ir_1}< 2\Phi(s)+\frac{1}{ir_1}\leq \frac{1}{2}\Theta_S(0,r_1)+\frac{1}{ir_1}\quad\text{for all }i\gg 2r_1.$$ We have reached a contradiction. Therefore, if $\cT$ points are $(\phi,\Phi)$ detectable in $\cS$, then $\cT$ is $\phi$ separated at infinity in $\ccS$. \end{proof}

The terminology in Definition \ref{d:Tpdp} is justified by the following statement, whose proof demonstrates the utility of Theorem \ref{bcthm}.

\begin{lemma}\label{l:Tpoint} If $\cT$ points are $(\phi,\Phi)$ detectable in $\cS$, $S\in\ccS$, and $\Theta_S^{\cT}(0,r)<\phi$ for some $r>0$, then $0$ is a $\cT$ point of $S$.
\end{lemma}

\begin{proof} Assume $\cT$ points are $(\phi,\Phi)$ detectable in $\cS$. Then $\cT$ is separated at infinity in $\ccS$ by Lemma \ref{l:infsep3}. Suppose $S\in\ccS$ and $\Theta_S^{\cT}(0,r)<\phi$ for some $r>0$. Then we can find $\eta>0$ such that $\Theta_S^{\cT}(0,r)<\phi/(1+\eta)$. Since $S\in\ccS$, there exists a sequence $S_i\in\cS$ such that $S_i\rightarrow S$ in $\CL{\RR^n}$.

Let $\sigma>0$ be given. Choose $s<\min\{\sigma,1\}$ such that $\Phi(s)\leq \sigma$ and fix a parameter $\varepsilon\in(0,\sigma]$ to be specified later. Since $S_i\rightarrow S$, we can pick  $j\geq 1$ large enough so that \begin{equation*}\mD{0}{r/(1+s\varepsilon)}[S,S_j]<s\varepsilon\quad\text{and}\quad \mD{0}{2r/(1+s\varepsilon)}[S,S_j]<\frac{\varepsilon}{2}.\end{equation*} Since $\mud{0}{r/(1+s\varepsilon)}(S_j,S)<s\varepsilon$ and $0\in S$, the strong quasitriangle inequality implies \begin{align*} \Theta_{S_j}^\cT\left(0,\frac{r}{1+s\varepsilon}\right) &\leq (1+s\varepsilon)\Theta_S^{\cT}(0,r) + 2\mD{0}{2r/(1+s\varepsilon)}[S,S_j]\\
&< \left(1+s\varepsilon\right)\frac{\phi}{1+\eta} + \varepsilon.\end{align*} We now specify that $\varepsilon$ was chosen so that
$\phi\left(1+s\varepsilon\right)/(1+\eta)+ \varepsilon\leq \phi.$ Then, since $S_j\in \cS$ and  $\Theta_{S_j}^\cT(0,r/(1+s\varepsilon))<\varphi$, detectability implies that $$ \Theta_{S_j}^\cT\left(0,\frac{sr}{1+s\varepsilon}\right) <\Phi(s) \leq \sigma.$$ Thus, by the weak quasitriangle inequality and monotonicity, \begin{align*} \Theta_S^\cT\left(0,\frac{sr}{2(1+s\varepsilon)}\right) &\leq 2 \Theta_{S_j}^\cT\left(0,\frac{sr}{1+s\varepsilon}\right)+2\mD{0}{sr/(1+s\varepsilon)}[S,S_j]\\
&\leq 2\Theta_{S_j}^{\cT}\left(0,\frac{sr}{1+s\varepsilon}\right)+ \frac{2}{s}\mD{0}{r/(1+s\varepsilon)}[S,S_j]\leq 2\sigma + 2\varepsilon\leq 4\sigma. \end{align*}
We have shown that for all $\sigma>0$ there exists $t=sr/2(1+s\varepsilon)$ such that $\Theta_S^\cT(0,t)\leq 4\sigma$. Note that $t\rightarrow 0$ as $\sigma\rightarrow 0$. Hence $\liminf_{t\rightarrow 0} \Theta_S^{\cT}(0,t) =0$, and thus, $\Tan(S,0)\cap \ccT\neq\emptyset$ by Corollary \ref{TanVsTheta0}. On the other hand, $\Tan(S,0)\subseteq \ccS$, since $S\in\ccS$. Invoking Theorem \ref{bcthm}, we conclude that $\Tan(S,0)\subseteq \ccT$. Therefore, $0$ is a $\cT$ point of $S$ by Corollary \ref{TanVsTheta}.
\end{proof}

\begin{corollary} \label{c:Tpoint} If $\cT$ points are $(\phi,\Phi)$ detectable in $\cS$, then $$\ccS\cap \cTp = \{S\in \ccS: \Theta_S^\cT(0,r)\geq \varphi\text{ for all } r>0\}.$$\end{corollary}

Before giving our last criterion for separation at infinity, we record an auxiliary lemma, which encapsulates a crucial feature of detectability: good approximability estimates at one scale yield better approximability estimates at a smaller scale.

\begin{lemma}[improving approximability on smaller scales] \label{l:ass}
Suppose that $\cT$ points are $(\phi,\Phi)$ detectable in $\cS$. For all $\beta'<\phi/4$ and $\gamma'>0$ there exist $\alpha'>0$ and $s<1/8$ with the following property: if $A\subseteq\RR^n$ is nonempty, $\Theta_A^{\cS}(x,r)<\alpha'$, and $\Theta_A^{\cT}(x,r)<\beta'$ for some $x\in A$ and $r>0$, then $\Theta_A^{\cT}(x,sr)<\gamma'$. \end{lemma}

\begin{proof} Let $\beta'<\phi/4$ and $\gamma'>0$ be given. Choose $s<1/8$ such that $\Phi(8s)\leq \gamma'/4$ and choose $\alpha'>0$ to be the lesser of $\phi/4-\beta'$ and $s\gamma'/2$. Without loss of generality, suppose $0\in A\subseteq\RR^n$, $\Theta_A^{\cS}(0,r)<\alpha'$ and $\Theta_A^{\cT}(0,r)<\beta'$. Then there exist $S\in\cS$ and $T\in\cT$ such that $\mD{0}{r}[A,S]<\alpha'$ and $\mD{0}{r}[A,T]<\beta'$. By the weak quasitriangle inequality and closure property of the relative Walkup-Wets distance, \begin{align*} \mD{0}{r/4}[S,T]&\leq 2\mD{0}{r/2}[S,\cl{A}]+2\mD{0}{r/2}[\cl{A},T]\\
&\leq 4\mD{0}{r}[S,A]+4\mD{0}{r}[A,T]<4\alpha'+4\beta'.\end{align*} Hence $\Theta^{\cT}_S(0,r/4)\leq \mD{0}{r/4}[S,T]<\phi$, because $\alpha'\leq \phi/4-\beta'$. But $\cT$ points are $(\phi,\Phi)$ detectable in $\cS$, so $\Theta_S^{\cT}(0,2sr)<\Phi(8s)\leq \gamma'/4$. Hence there exists $T'\in\cT$ such that $\mD{0}{2sr}[S,T']< \gamma'/4$. On the other hand, since $\mD{0}{r}[A,S]<\alpha'$, we have $\mD{0}{2sr}[A,S]< \alpha'/2s$ by monotonicity. Thus, by the weak quasitriangle inequality, \begin{equation*} \mD{0}{sr}[A,T'] \leq 2\mD{0}{2sr}[A,S]+2\mD{0}{2sr}[S,T'] <\alpha'/s + \gamma'/2.\end{equation*} Therefore, $\Theta_A^{\cT}(0,sr) \leq \mD{0}{sr}[A,T']
< \alpha'/s + \gamma'/2\leq \gamma'$, because  $\alpha'\leq s\gamma'/2$.\end{proof}

We now give a criterion for separation at infinity for $\varepsilon$-enlargements.

\begin{lemma} \label{l:d2i} Suppose $\cT$ points are $(\phi,\Phi)$ detectable in $\cS$. For all $\psi>0$ and  $\delta>0$ such that $\psi+\delta<\phi/8$ there exists $\varepsilon^*>0$ such that $(\cT;\delta)^\Theta_{0,\infty}$ is $\psi$ separated at infinity in $(\cS;\varepsilon)^\Theta_{0,\infty}$ for all $\varepsilon<\varepsilon^*$.
\end{lemma}

\begin{proof}Assume that $\cT$ points are $(\phi,\Phi)$ detectable in $\cS$. Let $\psi>0$ and $\delta>0$ such that $\psi+\delta<\phi/8$ be given, fix $\varepsilon^*>0$ to be specified later, and pick any $\varepsilon<\varepsilon^*$. We abbreviate \begin{equation*} \cP:=(\cS;\varepsilon)^\Theta_{0,\infty} \quad \text{and} \quad \cQ:=(\cT;\delta)^\Theta_{0,\infty}.\end{equation*} Choose any $\eta>0$ small enough so that $2\psi+2\delta+2\eta<\phi/4$.

Suppose $A\in\cP$ and $\Theta^{\cQ}_A(0,r)<\psi$ for some $r>0$.  On one hand, since $\Theta^{\cQ}_A(0,r)<\psi$, there exists $\hat T\in\cQ$ such that $\mD{0}{r}[A,\hat T]<\psi$. Also, since $\Theta^{\cT}_{\hat T}(0,r)\leq \delta$, there exists $T\in\cT$ such that $\mD{0}{r}[\hat T,T]\leq \delta+\eta$. Hence, by the weak quasitriangle inequality, \begin{equation*} \Theta_A^{\mathcal{T}}(0,r/2)\leq \mD{0}{r/2}[A,T]<2\psi + 2\delta+2\eta< \phi/4.\end{equation*} On the other hand,  $\Theta^{\cS}_A(0,r/2)\leq \varepsilon<\varepsilon^*$, because $A\in\cP$. Let $\alpha'>0$ and $s<1/8$ be constants from Lemma \ref{l:ass} corresponding to $\beta'=2\psi+2\delta+2\eta$ and $\gamma'=\min\{\psi,\delta\}$ and set $\varepsilon^*=\alpha'$. Then $\Theta_A^{\cQ}(0,sr/2)\leq \Theta_A^{\cT}(0,sr/2)<\min\{\psi,\delta\}$ by Lemma \ref{l:ass}. To summarize, we have shown if $A\in\cP$ and $\Theta^{\cQ}_A(0,r)<\psi$ for some $r>0$, then $\Theta_A^{\cQ}(0,(s/2)r)\leq \Theta_A^{\cT}(0,(s/2)r)<\min\{\psi,\delta\}$.

Thus, by induction, \begin{equation}\label{e:d2i-1} A\in\cP\text{ and }\Theta_A^\cQ(0,r)<\psi\text{ for some }r>0 \Longrightarrow \Theta_A^\cQ(0,(s/2)^kr)<\delta\text{ for all }k\geq 1.\end{equation}

Now suppose that $A\in\cP$ and $\limsup_{r\uparrow \infty} \Theta_A^\cQ(0,r)<\psi$. Then there is an  $r_0>0$ such that $\Theta^\cQ_A(0,r)<\psi$ for all $r\geq r_0>0$. Hence $\Theta_A^\cQ(0,r)<\delta$ for all $r>0$ by (\ref{e:d2i-1}). That is, $A\in\cQ$. We have shown for all $A\in\cP$ either $\limsup_{r\uparrow\infty} \Theta_A^{\cQ}(0,r)\geq \psi$ or $A\in\cQ$. Therefore, $\cQ$ is $\psi$ separated at infinity in $\cP$.
\end{proof}

\begin{corollary}\label{c:d2i} If $\cT$ points are $(\phi,\Phi)$ detectable in $\cS$, $\delta\leq \phi/32$ and $\Phi(t)\leq \delta/4$, then $(\cT;\delta)^\Theta_{0,\infty}$ is $\phi/16$ separated at infinity in $(\cS;\varepsilon)^\Theta_{0,\infty}$ for all $\varepsilon< t\phi/16$. \end{corollary}

\begin{proof} Given $\delta\leq \phi/32$, set $\psi=\phi/16$ and $\eta=\phi/64$. Then $\beta'=2\delta+2\psi+2\eta\leq 14\phi/64$ and $\gamma'=\min\{\psi,\delta\}=\delta$. Thus, by the proof of Lemma \ref{l:ass}, $\alpha'=\min\{\phi/4-\beta',s\gamma'/2\}=s\delta/2$ for any $s<1/8$ such that $\Phi(8s)\leq \gamma'/4=\delta/4$. Therefore, writing $t=8s$, we have $\varepsilon^*=\alpha' = t\delta/16$ for any $t<1$ such that $\Phi(t)\leq \delta/4$.\end{proof}

\section{Decompositions of Reifenberg type sets} \label{s:decomp}

In this section, we discuss variations on the following theorem, which is an application of the connectedness of the cone of tangent sets at a point (Theorem \ref{bcthm}).

\begin{theorem}[pointwise decomposition]\label{t:decomp} Let $\cT$ and $\cS$ be local approximation classes. Suppose $\cT$ is separated at infinity in $\ccS$. If $A\subseteq\RR^n$ is locally well approximated by $\cS$, then $A$ can be written as a disjoint union \begin{equation} A=A_\ccT\cup A_\cTp \quad (A_{\cT}\cap A_\cTp=\emptyset),\end{equation} where  $\Tan(\cl{A},x)\subseteq \ccS\cap \ccT$ for all $x\in A_\ccT$ and $\Tan(\cl{A},x)\subseteq\ccS\cap \cTp$ for all $x\in A_\cTp$.\end{theorem}

\begin{proof} Suppose $\cT$ is separated at infinity in $\ccS$ and let $A\subseteq\RR^n$ be locally well approximated by $\cS$. Then $\Tan(\cl{A},x)\subseteq\PsTan(\cl{A},x)\subseteq \ccS$ for all $x\in A$ by Corollary \ref{PsTanVsApprox}. Define \begin{equation*} A_\ccT = \{ x\in A: \Tan(\cl{A},x)\cap\ccT\neq\emptyset \}\end{equation*} and \begin{equation*} A_\cTp = \{ x\in A: \Tan(\cl{A},x)\cap \ccT=\emptyset\}.\end{equation*} Then $A=A_\ccT\cup A_\cTp$ and $A_\ccT\cap A_\cTp=\emptyset$. It follows immediately from Theorem \ref{bcthm} that $\Tan(\cl{A},x)\subseteq\ccT$ for all $x\in A_\ccT$ and $\Tan(\cl{A},x)\subseteq\cTp$ for all $x\in A_\cTp$.\end{proof}

We now show that if $\cT$ points are detectable in $\cS$, then the set $A_\ccT$ in Theorem \ref{t:decomp} is locally well approximated by $\cT$.

\begin{theorem}[open/closed decomposition]\label{t:open} Let $\cT$ and $\cS$ be local approximation classes. Suppose $\cT$ points are detectable in $\cS$.
If $A\subseteq\RR^n$ is locally well approximated by $\cS$, then $A$ can be written as a disjoint union \begin{equation} A=A_\ccT\cup A_\cTp \quad (A_\ccT\cap A_\cTp=\emptyset),\end{equation} where $\PsTan(\cl{A},x)\subseteq \ccS\cap\ccT$ for all $x\in A_\ccT$ and $\Tan(\cl{A},x)\subseteq\ccS\cap \cTp$ for all $x\in A_\cTp$. Moreover,  $A_\ccT$ is relatively open in $A$ and $A_\ccT$ is locally well approximated by $\cT$.
\end{theorem}

To aid in the proof of Theorem \ref{t:open}, we first establish an auxiliary lemma, which is a generalization of \cite{Badger3}*{Lemma 5.9}.

\begin{lemma}\label{l:Tpdp} Suppose that $\cT$ points are $(\phi,\Phi)$ detectable in $\cS$. For all $\gamma>0$ there exist $\alpha>0$ and $\beta>0$ such that if $x\in A\subseteq\RR^n$, \begin{equation}\label{l:Tpdp.1}\Theta_A^{\cS}(x,r')<\alpha\quad\text{for all }0<r'\leq r,\end{equation} and  $\Theta^\cT_A(x,r)<\beta$, then $\Theta^\cT_A(x,r')<\gamma$ for all $0<r'\leq r$.
\end{lemma}

\begin{proof} Assume that $\cT$ points are $(\phi,\Phi)$ detectable in $\cS$. Let $\gamma>0$ be given. Replacing $\gamma$ with a smaller value if necessary, we may assume without loss of generality that $\gamma<\phi/4$. Let $x\in A\subseteq\RR^n$. By Lemma \ref{l:ass}, applied with $\beta'=\gamma<\phi/4$ and $\gamma'=\gamma$, there exist $\alpha'>0$ and $s<1/8$ so that \begin{equation}\label{l:Tpdp.2} \Theta_A^{\cS}(x,t)<\alpha' \text{ and } \Theta_A^{\cT}(x,t)<\gamma\Longrightarrow \Theta_A^{\cT}(x,st)<\gamma.\end{equation} Set $\alpha=\alpha'$ and $\beta=s\gamma$.

To finish the lemma, assume that (\ref{l:Tpdp.1}) holds for some for some $r>0$ and
$\Theta^\cT_A(x, r) < \beta$. Then, by monotonicity, \begin{equation}\label{l:Tpdp.10} \Theta^\cT_A(x, tr) < \beta/t \leq \beta/s = \gamma\quad\text{for all }s\leq t\leq 1.\end{equation}
 Induction on (\ref{l:Tpdp.10}) using (\ref{l:Tpdp.1}) and (\ref{l:Tpdp.2}) gives $\Theta^\cT_A(x, r')<\gamma$ for all $0<r'\leq r$.
\end{proof}

\begin{proof}[Proof of Theorem \ref{t:open}] Assume that $\cT$ points are $(\phi,\Phi)$ detectable in $\cS$ for some $\phi<1$. Then $\cT$ is separated at infinity in $\ccS$ by Lemma \ref{l:infsep3}.  Let $\alpha$ and $\beta$ be the constants from Lemma \ref{l:Tpdp} corresponding to $\gamma=\phi/8$. Set $\tilde\beta=\min\{\beta/6,1/2\}$. Suppose that $A\subseteq\RR^n$ is locally well approximated by $\cS$. Then $\PsTan(\cl{A},x)\subseteq\ccS$ by Corollary \ref{PsTanVsApprox}.  To continue, partition $A$ into two sets: $$ A_\ccT=\left\{x\in A: \liminf_{r\downarrow 0} \Theta^{\cT}_A(x,r)< \tilde\beta\right\}\quad\text{and}\quad A_\cTp=\left\{x\in A: \liminf_{r\downarrow 0} \Theta^{\cT}_A(x,r)\geq \tilde\beta\right\}.$$ Then $A=A_\ccT\cup A_\cTp$ and $A_\ccT\cap A_\cTp=\emptyset$. Since $\liminf_{r\rightarrow 0}\Theta^{\cT}_A(x,r)=0$ whenever $\Tan(\cl{A},x)\cap\ccT\neq\emptyset$ by Corollary \ref{TanVsTheta0}, it is clear that $\Tan(\cl{A},x)\subseteq\ccS\setminus \ccT$ for all $x\in A_\cTp$. Thus, $\Tan(\cl{A},x)\subseteq \cTp$ for all $x\in A_\cTp$ by Theorem \ref{bcthm}.  It remains to show $A_\ccT$ is relatively open in $A$; $A_\ccT$ is locally well approximated by $\cT$; and $\PsTan(\cl{A},x)\subseteq \ccT$ for all $x\in A_\ccT$.

Fix $x_0\in A_\ccT$. Because $A$ is locally well approximated by $\cS$, there is a scale $r_0\in(0,1]$ such that $$\Theta_A^{\cS}(x,r')\leq \alpha/2<\alpha\quad\text{for all }x\in A\cap \ball(x_0,1)\text{ and }0<r'\leq r_0.$$ Since $x_0\in A_\ccT$, $\liminf_{r\rightarrow 0}\Theta^{\cT}_A(x_0,r)<\tilde\beta$. Hence there exists $r_1\in (0,r_0/4]$ such that $\Theta^{\cT}_A(x_0,4r_1)<\tilde\beta$.
By monotonicity (Lemma \ref{l:tmono}), for all $x\in A$ such that $|x-x_0|\leq \tilde\beta (2r_1)$,  \begin{align*}
 \Theta_A^{\cT}(x,r_1) &\leq 2\left(\tilde\beta + (1+\tilde\beta)\Theta_A^\cT(x_0,(1+\tilde\beta)2r_1)\right)\\
 &\leq 2\left(\tilde\beta+2\Theta_A^{\cT}(x_0,4r_1)\right)<6\tilde\beta\leq \beta.\end{align*} Therefore, by Lemma \ref{l:Tpdp}, \begin{equation}\label{e:open1} \Theta^{\cT}_A(x,r') < \phi/8\quad\text{ for all }x\in A\cap \ball(x_0,\tilde\beta (2r_1))\text{ and }0<r'\leq r_1.\end{equation}
Fix $x\in A\cap \ball(x_0,\tilde\beta(2r_1))$. We will now show that $\Tan(\cl{A},x)\cap \ccT\neq\emptyset$. Recall $\Tan(\cl{A},x)\subseteq \PsTan(\cl{A},x)\subseteq\ccS$. Pick any $S\in\Tan(\cl{A},x)$, say $(\cl{A}-x)/s_i\rightarrow S\in\ccS$ for some sequence $s_i\rightarrow 0$. By \eqref{e:open1} and Lemmas \ref{ApproxVsLimits} and \ref{l:close}, \begin{align*}\Theta_S^\cT(0,1/2) &\leq 2\liminf_{i\rightarrow\infty}\Theta^\cT_{(\cl{A}-x)/s_i}(0,1)\\ &=2\liminf_{i\rightarrow\infty}\Theta^\cT_{\cl{A}}(x,s_i)
\leq 4\sup_{0<r'\leq r_1}\Theta^\cT_A(x,r')\leq \phi/2.\end{align*} Hence $\Tan(S,0)\subseteq \ccT$ by Lemma \ref{l:Tpoint}, and thus, $\Tan(\cl{A},x)\cap \ccT \neq\emptyset$ by Lemma \ref{l:iterate}.
Therefore, $A\cap \ball(x_0,\tilde\beta(2r_1))\subseteq A_\ccT$ by Corollary \ref{TanVsTheta0}. Since $x_0\in A_\ccT$ was fixed arbitrarily, we conclude $A_\ccT$ is relatively open in $A$.

It now remains to show that $A_\ccT$ is locally well approximated by $\cT$ and  $\PsTan(\cl{A},x)\subseteq\ccT$ for all $x\in A_{\ccT}$. Fix any compact set $K\subseteq A_\ccT$ and $0<\tau\leq 3$. Redefine $\alpha$ and $\beta$ to be the constants from Lemma \ref{l:Tpdp} corresponding to $\gamma=\tau/6$. Because $A$ is locally well approximated by $\cS$, there exists $r_2>0$ such that $$\Theta_A^{\cS}(x,r')\leq\alpha/2<\alpha\quad\text{for all }x\in K\text{ and }0<r'\leq r_2.$$ For each $x\in K$, choose $r_x\in(0,r_2/4]$ so that $A\cap \ball(x,8r_x)\subseteq A_\ccT$ and $\Theta^\cT_A(x,4r_x)<\beta$. Then $\Theta^{\cT}_{A_\ccT}(x,r')=\Theta^\cT_A(x,r')<\gamma$ for all $0<r'\leq 4r_x$ by Lemma \ref{l:Tpdp}. By monotonicity, we conclude that \begin{align*} \Theta^\cT_{A_\ccT}(y,r') &\leq 2\left(\gamma +(1+\gamma)\Theta_{A_\ccT}^\cT\left(x,(1+\gamma)2r'\right)\right)\\
&\leq 2\left(\gamma+2\Theta_A^\cT(x,4r')\right)<6\gamma=\tau\end{align*} for all $x\in K$, $|y-x|\leq \gamma(2r_x)$, and $0<r'\leq r_x$. Finally, extracting a finite subcover of $K$ from  $\{\ball(x,\gamma(2r_x)):x\in K\}$, it follows that there exists $r_K>0$ such that $\Theta_{A_\ccT}^\cT(y,r')<\tau$ for all $y\in K$ and for all $0<r'\leq r_K$. Therefore, $A_\ccT$ is locally $\tau$-approximable by $\cT$ for all $\tau>0$. That is, $A_\ccT$ is locally well approximated by $\cT$.
Because $A$ and $A_\ccT$ coincide locally near all $x\in A_\ccT$, it follows from Corollary \ref{PsTanVsApprox} that $\PsTan(\cl{A},x)\subseteq\ccT$ for all $x\in A_{\cT}$.
\end{proof}

With additional terminology, we can say something more about the singular set $A_\cTp$ appearing in Theorem \ref{t:open}.

\begin{definition}\label{lwa:along}
Let $\cS$ a local approximation class and let $A'\subseteq A\subseteq\RR^n$ be nonempty. We say that $A$ is \emph{locally well approximated along $A'$ by $\cS$} if for all compact sets $K\subseteq A'$,
$$\limsup_{r\downarrow 0} \sup_{x\in K}\Theta^\cS_A(x,r)= 0.$$
\end{definition}

\begin{corollary}\label{S-T} If $\cT$ points are $(\phi, \Phi)$ detectable in $\cS$ and $A$ is locally well approximated by $\cS$, then $A$ is locally well approximated along $A_\cTp$ by $\ccS\cap \cTp$.
\end{corollary}

\begin{proof}Assume that $\cT$ points are $(\phi,\Phi)$ detectable in $\cS$. Then, by Corollary \ref{c:Tpoint}, $$\cS':=\ccS\cap \cTp=\{S\in \ccS: \Theta^\cT_S(0,r)\geq \phi\text{ for all }r>0\}$$
We note for later use that $\ccS\cap \cTp$ is closed.

Let $A\subseteq\RR^n$ be locally well approximated by $\cS$ and decompose $A=A_\ccT\cup A_\cTp$ as in Theorem \ref{t:open}. If $A_\cTp$ is empty, then there is nothing to prove. Thus, we may suppose that $A_\cTp$ is nonempty. In particular, since $\Tan(\cl{A},x)\subseteq \cS'$ for all $x\in A_\cTp$, we know that $\cS'$ is nonempty.  Let $\tilde\beta$ be the constant appearing in the proof of Theorem \ref{t:open}, so that \begin{equation}\label{e:C0} \liminf_{r\downarrow 0} \Theta^\cT_A(x,r)\geq \tilde\beta\quad\text{for all }x\in A_\cTp.\end{equation} Suppose in order to obtain a contradiction that there exist a compact set $K\subseteq A_\cTp$ and $c_0>0$ such that $$\limsup_{r\downarrow 0}\sup_{x\in K}\Theta^{\cS'}_A(x,r)>c_0>0.$$ Choose sequences $r_i>0$ and $x_i\in K$ such that $\Theta^{\cS'}_A(x_i,r_i)>c_0$ for all $i\geq 1$ and $r_i\rightarrow 0$. Because $K$ and $\CL{0}$ are sequentially compact, we may pass to a subsequence of $(r_i,x_i)_{i=1}^\infty$ to assume that $x_i\rightarrow x$ for some $x\in K$ and $(\cl{A}-x_i)/r_i\rightarrow S$ for some $S\in\PsTan(\cl{A},x)\subseteq\ccS$. Then, by Lemmas  \ref{ApproxVsLimits}, \ref{l:theta_inv} and \ref{l:close}, $$\Theta^{\cS'}_S(0,2) \geq \frac{1}{2}\limsup_{i\rightarrow\infty} \Theta^{\cS'}_{(\cl{A}-x_i)/r_i}(0,1) =\frac{1}{2}\limsup_{i\rightarrow \infty} \Theta_{\cl{A}}^{\cS'}(x_i,r_i)\geq  \frac{1}{2}\limsup_{i\rightarrow\infty} \Theta^{\cS'}_A(x_i,r_i)\geq \frac{c_0}{2}.$$ In particular, $S \not\in \cl{\cS'}= \ccS\cap\cTp$. Hence $\Theta_S^\cT(0,r)<\phi$ for some $r>0$, and thus, $0$ is a $\cT$ point of $S$ by Lemma \ref{l:Tpoint}.

We are now ready to derive a contradiction. Let $\alpha$ and $\beta$ be the constants from Lemma 6.3 corresponding to $\gamma=\tilde\beta/2$. On one hand, since $A$ is locally well approximated by $\cS$, there exists $\rho>0$ such that $$ \Theta^\cS_A(x,r) \leq \alpha/2<\alpha\quad\text{for all }x\in K\text{ and for all }0<r\leq \rho.$$ On the other hand, because $0$ is a $\cT$ point of $S$, we can pick $\lambda>0$ such that $\Theta_S^\cT(0,2\lambda)\leq \beta/4$. Then, by Lemmas \ref{l:theta_inv}, \ref{l:close} and \ref{ApproxVsLimits} (as above), $$ \limsup_{i\rightarrow\infty} \Theta_A^\cT(x_i,\lambda r_i) \leq 2 \Theta_S^\cT(0,2\lambda ) \leq \beta/2 <\beta.$$ Pick any $i\geq 1$ large enough so that $\Theta_A^\cT(x_i,\lambda r_i)<\beta$ and $\lambda r_i\leq \rho$. Then, by Lemma \ref{l:Tpdp}, $\Theta^\cT_A(x_i,r)<\tilde\beta /2$ for all $0<r\leq \lambda r_i$. Thus, $\limsup_{r\downarrow 0} \Theta^\cT_A(x_i,r)\leq \tilde\beta/2<\tilde\beta$, violating (\ref{e:C0}). We have reached a contradiction. Therefore, $A$ is locally well approximated along $A_\cTp$ by $\cS'=\ccS\cap\cTp$.
\end{proof}

Our last goal of the section is to establish a perturbation of Theorem \ref{t:open}, where sets locally well approximated by $\mathcal{S}$ are replaced by sets locally $\varepsilon$-approximable by $\mathcal{S}$.

\begin{theorem}[open/closed decomposition with $\varepsilon$-enlargements] Suppose $\cT$ points are $(\phi,\Phi)$ detectable in $\cS$. \label{t:perturb} Given $\delta\leq \phi/32$, $t<1$ such that $\Phi(t)\leq \delta/4$, and $\varepsilon<t\delta/16$, let $$\cP=(\cS;\varepsilon)^\Theta_{0,\infty}\quad\text{and}\quad \cQ=(\cT;\delta)^\Theta_{0,\infty}.$$ If $A\subseteq\RR^n$ is  locally $\varepsilon$-approximable by $\cS$, then $A$ can be written as a disjoint union $$ A=A_\cQ\cup A_{\cQp} \quad (A_{\cQ}\cap A_{\cQp}=\emptyset),$$ where $\PsTan(\cl{A},x) \subseteq \cP \cap \cQ$ for all $x\in A_\cQ$ and $\Tan(\cl{A},x)\subseteq \cP\cap \cQ^\perp$ for all $x\in A_\cQp$. Moreover, $A_\cQ$ is relatively open in $A$ and $ A_\cQ$ is locally $\delta$-approximable by $\cT$.
\end{theorem}

To enable the proof of Theorem \ref{t:perturb}, we first establish a variant of Lemma \ref{l:Tpdp}.

\begin{lemma}\label{l:Tpdp2} Suppose that $\cT$ points are $(\phi,\Phi)$ detectable in $\cS$. Let $\delta\leq \phi/32$ be given and write $\cQ=(\cT;\delta)^\Theta_{0,\infty}$. For all $\gamma\leq \phi/24$ and $s<1/8$ such that $\Phi(8s)\leq \gamma/4$, if $x\in A\subseteq \RR^n$, \begin{equation}\label{l:Tpdp2.1}\Theta_A^{\cS}(x,r')<\frac{s\gamma}{2}\quad\text{for all }0<r'\leq r,\end{equation} and  $\Theta^{\cQ}_A(x,r)<s\gamma/2$, then $ \Theta^{\cT}_A(x,r')<\gamma$ for all $0<r'\leq sr/2$.
\end{lemma}

\begin{proof} Assume that $\cT$ points are $(\phi,\Phi)$ detectable in $\cS$. Let $\delta\leq \phi/32$ and $\cQ=(\cT;\delta)^\Theta_{0,\infty}$. Let $\gamma\leq \phi/24$ be given. Suppose  $x\in A\subseteq\RR^n$. By Lemma \ref{l:ass}, applied with $\beta'=\gamma+4\delta\leq \phi/6$ and $\gamma'=\gamma$, there exist $\alpha'>0$ and $s<1/8$ so that \begin{equation}\label{l:Tpdp2.2} \Theta_A^{\cS}(x,t)<\alpha' \text{ and } \Theta_A^{\cT}(x,t)<\gamma+4\delta\Longrightarrow \Theta_A^{\cT}(x,st)<\gamma.\end{equation} Reviewing the proof of Lemma \ref{l:ass}, $s<1/8$ can be chosen to be any number such that $\Phi(8s)\leq \gamma'/4=\gamma/4$ and $\alpha'=\min\{\phi/4-\beta',s\gamma'/2\}=s\gamma/2$. Set $\alpha=\alpha'$ and $\beta=s\gamma/2$.

To finish the lemma, assume that (\ref{l:Tpdp2.1}) holds for some for some $r>0$ and
$\Theta^\cQ_A(x, r) < \beta$.  Then, by monotonicity, $$\Theta^\cQ_A(x, tr) < \beta/t \leq \beta/s = \gamma/2\quad\text{for all }s\leq t\leq 1.$$ Thus, by the weak quasitriangle inequality, \begin{equation}\label{l:Tpdp2.11} \Theta^\cT_A(x,tr/2) < \gamma + 4\delta\quad\text{for all }s\leq t\leq 1.\end{equation} (Indeed, pick $\hat T\in \cQ$ with $\mD{x}{tr}[A,x+\hat T] < \gamma/2$ and $T\in\cT$ with $\mD{0}{tr}[\hat T,T]<2\delta$. Then $\Theta^\cT_A(x,tr/2)\leq \mD{x}{tr/2}[A,x+T]\leq 2\mD{x}{tr}[A,x+\hat T] + 2\mD{x}{tr}[x+\hat T,x+T]<\gamma+4\delta$.)
Induction on (\ref{l:Tpdp2.11}) using (\ref{l:Tpdp2.1}) and (\ref{l:Tpdp2.2}) gives $\Theta^\cT_A(x, r')<\gamma$ for all $0<r'\leq sr/2$.
\end{proof}

\begin{proof}[Proof of Theorem \ref{t:perturb}] Assume that $\cT$ points are $(\phi,\Phi)$ detectable in $\cS$ for some $\phi<1$. Let $\delta\leq \phi/32$, $t<1$ such that $\Phi(t)\leq \delta/4$, and $\varepsilon<t\delta/16$ be given, and define $\cP=(\cS;\varepsilon)^\Theta_{0,\infty}$ and $\cQ=(\cT;\delta)^\Theta_{0,\infty}$. Note that $\cQ$ is separated at infinity in $\cP$ by Corollary \ref{c:d2i}. Let $s=t/8$ so that $\Phi(8s)=\Phi(t)\leq \delta/4$, and put $\tilde\gamma=\min\{s\delta/12,1/2\}$. Now suppose that $A\subseteq\RR^n$ is locally $\varepsilon$-approximable by $\cS$. Note that $\PsTan(\cl{A},x)\subseteq \cP$ for all $x\in A$ by Theorem \ref{PsTanVsPerturb}. To proceed, partition $A$ into two sets: $$ A_\cQ=\left\{x\in A: \liminf_{r\downarrow 0} \Theta^{\cQ}_A(x,r)< \tilde\gamma\right\}\quad\text{and}\quad A_{\cQp}=\left\{x\in A: \liminf_{r\downarrow 0} \Theta^{\cQ}_A(x,r)\geq \tilde\gamma\right\}.$$ Then $A=A_\cQ\cup A_{\cQp}$ and $A_\cQ\cap A_{\cQp}=\emptyset$. Since $\liminf_{r\rightarrow 0}\Theta^{\cQ}_A(x,r)=0$ whenever $\Tan(\cl{A},x)\cap\cQ\neq\emptyset$ by Corollary \ref{TanVsTheta0}, it is clear that $\Tan(\cl{A},x)\subseteq\cPcQ$ for all $x\in A_{\cQp}$. It follows that $\Tan(\cl{A},x)\subseteq\cQ^\perp$ for all $x\in A_{\cQp}$ by Theorem \ref{bcthm}.  To complete the proof of the theorem, we must show that $A_\cQ$ is relatively open in $A$; $A_\cQ$ is locally $\delta$-approximable by $\cT$; and $\PsTan(\cl{A},x)\subseteq \cQ$ for all $x\in A_\cQ$.

Fix $x_0\in A_\cQ$. Since $A$ is locally $\varepsilon$-approximable by $\cS$, there exists $r_0\in(0,1]$ such that $$\Theta_A^\cS(x,r')\leq \varepsilon<s\delta/2\quad\text{for all }x\in A\cap \ball(x_0,1)\quad\text{and}\quad 0<r'\leq r_0.$$ Since $x_0\in A_\cQ$, $\liminf_{r\rightarrow 0}\Theta^\cQ_A(x_0,r)< \tilde\gamma$. Hence there exists $r_1\in (0,r_0/4]$ such that $\Theta^\cQ_A(x_0,4r_1)<\tilde\gamma$. By monotonicity, for all $x\in A$ such that $|x-x_0|\leq \tilde\gamma(2r_1)$, \begin{align*} \Theta_A^\cQ(x,r_1) &\leq 2\left(\tilde\gamma + (1+\tilde\gamma)\Theta^\cQ_A(x_0,(1+\tilde\gamma)2r_1)\right)\\
&\leq 2\left(\tilde\gamma+2\Theta_A^{\cQ}(x_0,4r_1)\right)<6 \tilde\gamma \leq s\delta/2.\end{align*} Therefore, by Lemma \ref{l:Tpdp2} with $\gamma=\delta$, \begin{equation} \label{e:Bopen1} \Theta^{\cT}_A(x,r')<\delta\quad\text{ for all }x\in A\cap \ball(x_0,\tilde\gamma(2r_1))\text{ and }0<r'\leq sr_1/2.\end{equation} Fix $x\in A\cap \ball(x_0,\tilde\gamma(2r_1))$. Then $\limsup_{r\rightarrow 0} \Theta_A^\cT(x,r)\leq \delta$ by (\ref{e:Bopen1}). Thus, $\Tan(\cl{A},x)\subseteq\cQ$ by Theorem \ref{l:tan-perturb},  and $\liminf_{r\rightarrow 0} \Theta_A^\cQ(x,r)=0$ by Corollary \ref{TanVsTheta0}. Hence $A\cap\ball(x_0,\tilde\gamma(2r_1))\subseteq A_\cQ$. Since $x_0\in A_\cQ$ was fixed arbitrarily, we conclude $A_\cQ$ is relatively open in $A$.

It now remains to show that $A_\cQ$ is locally $\delta$-approximable by $\cT$ and $\PsTan(\cl{A},x)\subseteq\cQ$ for all $x\in A_{\cQ}$. Let $K$ be a compact subset of $A_\cQ$. By (\ref{e:Bopen1}), for every $x\in A_\cQ$ there exists $r_x>0$ such that $\Theta^\cT_{A_Q}(y,r')=\Theta^\cT_A(y,r')<\delta$ for all $y\in A_Q\cap \ball(x,\tilde\gamma r_x)$ and for all $0<r'\leq sr_x/8$.   Thus, by extracting a finite subcover of $K$ from  $\{\ball(x,\tilde\gamma r_x):x\in K\}$, it follows that there exists $r_K>0$ such that $\Theta_{A_\cQ}^\cT(y,r')<\delta$ for all $y\in K$ and for all $0<r'\leq r_K$. Therefore, $A_\cQ$ is locally $\delta$-approximable by $\cT$. Because $A$ and $A_\cQ$ coincide locally near all $x\in A_\cQ$, $\PsTan(\cl{A},x)\subseteq\cQ$ for all $x\in A_{\cQ}$ by Theorem \ref{PsTanVsPerturb}.
\end{proof}

\section{Unilateral approximation and Mattila-Vuorinen type sets}

The following definition generalizes Jones' beta numbers and Mattila and Vuorinen's linear approximation property to unilateral approximation of a set $A$ by closed sets in an arbitrary local approximation class $\cS$.

\begin{definition}[Mattila-Vuorinen type sets]\label{MVsets}
Let $\cS$ be a local approximation class and let $A\subseteq\RR^n$.
\begin{enumerate}
\item For all $x\in\RR^n$ and $r>0$, define the \emph{unilateral approximability $\beta^\cS_A(x,r)$} of $A$ by $\cS$ at location $x$ and scale $r$ by
$$\beta^\cS_A(x,r) = \inf_{S\in \cS} \mud x r(A, x+S).$$
\item We say that $A$ is \emph{unilaterally $(\varepsilon,r_0)$-approximable by $\cS$} if $\beta^\cS_A(x,r)\leq \varepsilon$ for all $x\in A$ and $0<r\leq r_0$.
\item We say that $A$ is \emph{locally unilaterally $\varepsilon$-approximable by $\cS$} if for all compact sets $K\subseteq A$, there exists $r_K>0$ such that $\beta^\cS_A(x,r)\leq \varepsilon$ for all $x\in K$ and $0<r\leq r_K$.
\item We say that $A$ is \emph{locally unilaterally well approximated by $\cS$} if $A$ is locally unilaterally $\varepsilon$-approximable by $\cS$ for all $\varepsilon>0$.
\end{enumerate}
\end{definition}

\begin{lemma}[properties of $\beta^\cS_A(x,r)$] Let \label{l:bp} $\cS$ be a local approximation class, let $A\subseteq\RR^n$ be nonempty, let $x,y\in \RR^n$ and let $r,s>0$. \begin{itemize}
\item comparison with bilateral approximability: $\beta^\cS_A(x,r) \leq \Theta^\cS_A(x,r)$.
\item size: $0\leq \beta^\cS_A(x,r)\leq 1$; $\beta^\cS_A(x,r)=0$ if $A\cap \ball(x,r)=\emptyset$.
\item scale invariance: $\beta^\cS_A(x,r)=\beta^\cS_{\lambda A}(\lambda x,\lambda r)$ for all $\lambda>0$.
\item translation invariance: $\beta^\cS_A(x,r)=\beta^\cS_{A+z}(x+z,r)$ for all $z\in\RR^n$.
\item closure: $\beta^\cS_A(x,r)\leq \beta^\cS_{\cl{A}}(x,r) \leq (1+\delta)\beta^\cS_{A}(x,(1+\delta)r)$ for all $\delta>0$.
\item monotonicity: If $A\subseteq A'$, $\ball(x,r)\subseteq \ball(y,s)$ and $|x-y|\leq ts$, then \begin{equation}\label{beta-mono}\beta_A^\cS(x,r) \leq \frac{s}{r}\left(t+ \beta^\cS_{A'}(y,s)\right).\end{equation}
\item limits: If $A,A_1,A_2,\dots\in \CL{\RR^n}$ and $A_i\rightarrow A$ in $\CL{\RR^n}$, then \begin{equation}\begin{split}\label{betaLimits.1}\frac{1}{1+\varepsilon} \limsup_{i\to\infty}\,&\beta^\cS_{A_i}\left(x,\frac{r}{1+\varepsilon}\right)\\ \leq &\,\beta^\cS_{A}(x,r) \leq
(1+\varepsilon) \liminf_{i\to\infty}\beta^\cS_{A_i}(x,r(1+\varepsilon))\quad\text{for all }\varepsilon>0.\end{split}.\end{equation}
\end{itemize}
\end{lemma}

\begin{proof} The properties of unilateral approximability can be verified by modifying the proofs of the properties of bilateral approximability given in \S\ref{lsa}. Details are left to the reader.\end{proof}

\begin{remark} It immediately follows from the definitions that any subset of a Reifenberg type set is a Matilla-Vuorinen type set: \emph{If $A$ is locally bilaterally $\varepsilon$-approximable by $\cS$ and $B\subseteq A$, then $B$ is locally unilaterally $\varepsilon$-approximable by $\cS$.} The converse of this fact is generally false, as shown by the following example due to David and Toro.\end{remark}

\begin{example}[\cite{DT}*{Example 12.4}] Let $\cG(3,2)$ denote the set of 2-dimensional linear subspaces of $\RR^3$, which is a closed linear approximation class. For all $0< \delta\ll 1$ and for all $0<\varepsilon\ll \delta$, there exists a M\"obius strip $M\subset\RR^3$ with the following property: $M$ is unilaterally $(\varepsilon,r_0)$-approximable by $\cG(3,2)$ for all $r_0>0$, but $M$ cannot be extended to a set $N$, which is bilaterally $(\delta,r_1)$-approximable by $\cG(3,2)$, for any $r_1\gtrsim 1$.
\end{example}

Our next goal is to characterize unilaterally approximable sets in terms of the tangent and pseudotangent sets of their closure.

\begin{definition}[unilateral $\varepsilon$-enlargements] Let $\cS$ be a local approximation class. For all $\varepsilon\geq 0$, define  $$(\cS;\varepsilon)^\beta_{0,\infty}=\{A\in\CL{0}: \beta_{A}^{\cS}(0,r)\leq\varepsilon\text{ for all }r>0\}$$ and $$(\cS;\varepsilon)^\beta_{\RR^n,\infty}=\{A\in\CL{0}: \beta_{A}^{\cS}(x,r)\leq\varepsilon\text{ for all }x\in A\text{ and all }r>0\}.$$
\end{definition}

\begin{lemma} \label{l:enlarge2} Let $\cS$ be a local approximation class and let $\varepsilon\geq 0$. Then $(\cS;\varepsilon)^\beta_{\RR^n,\infty}$ and $(\cS;\varepsilon)^\beta_{0,\infty}$ are closed local approximation classes, and $(\cS;\varepsilon)^\beta_{\RR^n,\infty}$ is the maximal translation invariant local approximation class that is contained in $(\cS;\varepsilon)^\beta_{0,\infty}$.\end{lemma}

\begin{proof} The enlargements $(\cS;\varepsilon)^\beta_{\RR^n,\infty}$ and $(\cS;\varepsilon)^\beta_{0,\infty}$ are local approximation classes, because unilateral approximability $\beta^\cS_A(x,r)$ is scale invariant by Lemma \ref{l:bp}.

To show that $(\cS;\varepsilon)^\beta_{0,\infty}$ is closed in $\CL{0}$, suppose that $A_i\rightarrow A$ in $\CL{0}$ for some sequence $(A_i)_{i=1}^\infty$ in $(\cS;\varepsilon)^\beta_{0,\infty}$. By Lemma \ref{l:bp}, for all $r>0$, $$\beta^\cS_A(0,r) \leq \liminf_{\delta\downarrow 0}\liminf_{i\rightarrow\infty} \beta^\cS_{A_i}(0,(1+\delta)r)\leq \liminf_{\delta\downarrow 0}\liminf_{i\rightarrow\infty}\varepsilon=\varepsilon,$$ since each set $A_i\in(\cS;\varepsilon)^\beta_{0,\infty}$. Hence $A\in (\cS;\varepsilon)^\beta_{0,\infty}$, and thus,  $(\cS;\varepsilon)^\beta_{0,\infty}$ is closed in $\CL{0}$.

It is clear that $(\cS;\varepsilon)^\beta_{\RR^n,\infty}\subseteq (\cS;\varepsilon)^\beta_{0,\infty}$ and $(\cS;\varepsilon)^\beta_{\RR^n,\infty}$ is translation invariant. Suppose that $\cT\subseteq (\cS;\varepsilon)^\beta_{0,\infty}$ and $\cT$ is translation invariant. Let $T\in\cT$ and let $x\in \cT$. Then $$\beta^\cS_T(x,r)=\beta^{\cS}_{T-x}(0,r)\leq \varepsilon\quad\text{for all }r>0,$$ because $T-x\in \cT\subseteq (S;\varepsilon)^\beta_{0,\infty}$. Hence $T\in(\cS;\varepsilon)^\beta_{\RR^n,\infty}$ for all $T\in\cT$. Thus $\cT\subseteq(\cS;\varepsilon)^\beta_{\RR^n,\infty}$ for all translation invariant $\cT\subseteq(\cS;\varepsilon)^\beta_{0,\infty}$. Therefore, $(\cS;\varepsilon)^\beta_{\RR^n,\infty}$ is the maximal translation invariant local approximation class contained in $(\cS;\varepsilon)^\beta_{0,\infty}$. Finally, since $(\cS;\varepsilon)^\beta_{\RR^n,\infty}$ is translation invariant, its closure $(\cS;\varepsilon)^\beta_{\RR^n,\infty}$ is also translation invariant (see the proof of Lemma \ref{l:enlarge}). By maximality of $(\cS;\varepsilon)^\beta_{\RR^n,\infty}$, it follows that $(\cS;\varepsilon)^\beta_{\RR^n,\infty}$ is closed.\end{proof}

\begin{lemma} \label{l:0b} $(\cS;0)^\beta_{0,\infty}=\{A\in\CL{0}:A\subseteq S$ for some $S\in\ccS\}$.\end{lemma}

\begin{proof} If $A\subseteq S$ for some $S\in\ccS$, then $\beta^\cS_A(0,r) \leq \beta^\cS_S(0,r)\leq \Theta^\cS_S(0,r)=0$ for all $r>0$ by Lemma \ref{l:bp} and Lemma \ref{l:enlarge}. Thus, $\{A\in\CL{0}:A\subseteq S$ for some $S\in\ccS\}\subseteq (\cS;0)^\beta_{0,\infty}$.

For the reverse inclusion, suppose that $A\in (\cS;0)^\beta_{0,\infty}$. Then, for all $i\geq 1$, we can find $S_i\in\cS$ such that $\mud{0}{i}(A,S_i)\leq 1/i^2$. Passing to a subsequence of $(S_i)_{i=1}^\infty$, we may assume that $S_i\rightarrow S$ for some $S\in \ccS$. Passing to a further subsequence, we may also assume that $\mD{0}{2i}[S_i,S]\leq 1/i^2$ for all $i\geq 1$. It follows that for all $r>0$ and integers $i\geq r$, \begin{align*} \mud{0}{r}(A,S) &\leq \mud{0}{r}(A,S_i) + 2\mud{0}{2r}(S_i,S)\\
&\leq \frac{i}{r}\mud{0}{i}(A,S_i) + \frac{2i}{r}\mud{0}{2i}(S_i,S) \leq \frac{3}{ri},\end{align*} where the first inequality holds by the weak quasitriangle inequality for relative excess and the second inequality holds by monotonicity for relative excess. Letting $i\rightarrow\infty$, we conclude that $\mud{0}{r}(A,S)=0$ for all $r>0$. Hence, by the containment property of the relative excess, $A\subseteq S$. Therefore, $(\cS;0)^\beta_{0,\infty}\subseteq \{A\in\CL{0}:A\subseteq S$ for some $S\in\ccS\}.$\end{proof}

The following theorem and its corollary are unilateral variants of Theorem \ref{l:tan-perturb} and Corollary \ref{TanVsTheta0}.

\begin{theorem} \label{t:betatan} Let $\cS$ be a local approximation class, let $A\subseteq\RR^n$ and let $x\in A$. Then $\limsup_{r\downarrow 0}\beta^\cS_A(x,r)\leq \varepsilon$ if and only if $\Tan(\cl{A},x)\subseteq (\cS;\varepsilon)^\beta_{0,\infty}$.\end{theorem}

\begin{proof} Assume $x\in A\subseteq\RR^n$ and let $\varepsilon\geq 0$. Suppose that $\Tan(\cl{A},x)\subseteq (\cS;\varepsilon)^\beta_{0,\infty}$. Choose a sequence $r_i>0$ such that $r_i\rightarrow 0$ and $$\lim_{i\rightarrow\infty} \beta_A^\cS(x,r_i) = \limsup_{r\downarrow 0} \beta_A^\cS(x,r).$$ Since $\CL{0}$ is sequentially compact, there exists a subsequence $(r_{i_j})_{j=1}^\infty$ of $(r_i)_{i=1}^\infty$ such that $(\cl{A}-x)/r_{i_j}\rightarrow T$ for some $T\in\Tan(\cl{A},0)$. Then, by Lemma \ref{l:bp}, \begin{align*} \lim_{j\rightarrow\infty} \beta_A^\cS(x,r_{i_j}) &\leq \limsup_{j\rightarrow\infty} \beta_{\cl{A}}^\cS(x,r_{i_j}) =\limsup_{j\rightarrow\infty} \beta_{(\cl{A}-x)/r_{i_j}}^\cS(0,1) \\&\leq (1+\delta) \beta^\cS_{T}(0,1+\delta)\leq (1+\delta)\varepsilon\end{align*} for all $\delta>0$, where the last inequality holds because $T\in(\cS;\varepsilon)^\beta_{0,\infty}$. Therefore, letting $\delta\rightarrow 0$, we see that $\limsup_{r\downarrow 0} \beta_A^\cS(x,r) = \lim_{j\rightarrow\infty} \beta_A^\cS(x,r_{i_j})\leq \varepsilon$.

Conversely, suppose that $\limsup_{r\downarrow 0} \beta_A^\cS(x,r)\leq \varepsilon$. Let $T\in \Tan(\cl{A},x)$, say $(\cl{A}-x)/r_i\rightarrow T$ for some sequence $r_i\rightarrow 0$. Let $s>0$. Then, by Lemma \ref{l:bp}, \begin{align*}\beta_T^\cS(0,s) &\leq (1+\delta)\liminf_{i\rightarrow\infty} \beta_{(\cl{A}-x)/r_i}^\cS(0,(1+\delta)s) = (1+\delta)\liminf_{i\rightarrow\infty} \beta_{\cl{A}}^\cS(x,(1+\delta)sr_i) \\ &\leq (1+\delta)^2 \liminf_{i\rightarrow\infty} \beta_{A}^\cS(x,(1+\delta)^2sr_i) \leq (1+\delta)^2\limsup_{r\downarrow 0} \beta_A^\cS(x,r)\leq (1+\delta)^2\varepsilon\end{align*} for all $\delta>0$. Thus, $\beta_T^\cS(0,r)\leq \varepsilon$ for all $r>0$, for all $T\in \Tan(\cl{A},x)$. Therefore, $\Tan(\cl{A},x)\subseteq (\cS;\varepsilon)^\beta_{0,\infty}$.
\end{proof}

\begin{corollary} Let $\cS$ be a local approximation class, let $x\in A\subseteq\RR^n$, and let $\varepsilon\geq 0$. \begin{itemize}
\item If $\Tan(\cl{A},x)\cap (\cS;\varepsilon)^\beta_{0,\infty}\neq\emptyset$, then $\liminf_{r\downarrow 0} \beta^\cS_A(x,r)\leq \varepsilon$.
 \item If $\liminf_{r\downarrow 0} \beta^\cS_A(x,r)=0$, then $\Tan(\cl{A},x)\cap (\cS;0)^\beta_{0,\infty}\neq\emptyset$.\end{itemize} \end{corollary}

\begin{proof} Assume $x\in A\subseteq\RR^n$ and let $\varepsilon\geq 0$. Suppose that $\Tan(\cl{A},x)\cap (\cS;\varepsilon)^\beta_{0,\infty}\neq\emptyset$. Choose $T\in\Tan(A,x)\cap (\cS;\varepsilon)^\beta_{0,\infty}$, say $T=\lim_{i\rightarrow\infty} (\cl{A}-x)/r_i$ for some sequence $r_i\rightarrow 0$. Then, by Lemma \ref{l:bp}, \begin{align*} \liminf_{r\downarrow 0}\beta^\cS_A(x,r)&\leq \limsup_{i\rightarrow\infty} \beta_A^\cS(x,r_i) \leq \limsup_{i\rightarrow\infty} \beta_{\cl{A}}^\cS(x,r_i)
= \limsup_{i\rightarrow\infty} \beta_{(\cl{A}-x)/r_i}^\cS(0,1) \\&\leq (1+\delta) \beta^\cS_{T}(0,1+\delta)\leq (1+\delta)\varepsilon\end{align*} for all $\delta>0$, where the last inequality holds because $T\in(\cS;\varepsilon)^\beta_{0,\infty}$. Therefore, letting $\delta\rightarrow 0$, we see that $\liminf_{r\downarrow 0} \beta_A^\cS(x,r) \leq  \varepsilon$.

Conversely, suppose that $\liminf_{r\rightarrow 0} \beta^\cS_A(x,r)=0$. Then we can pick $r_i\rightarrow 0$ such that $\lim_{i\rightarrow \infty}\beta^\cS_A(x,r_i)=0$. Passing to a subsequence, we may assume that $(\cl{A}-x)/r_i\rightarrow B$ for some $B\in \Tan(\cl{A},x)$. Then, for all $0<r\leq 1/4$,  \begin{align*}\beta_B^{\cS}(0,r) &\leq 2\liminf_{i\rightarrow\infty} \beta_{(\cl{A}-x)/r_i}^\cS(0,2r)\\
&= 2\liminf_{i\rightarrow\infty} \beta_{\cl{A}}^\cS(x,2rr_i) \leq 4\liminf_{i\rightarrow\infty} \beta_{A}^\cS(x,4rr_i)\leq \frac{1}{r}\liminf_{i\rightarrow\infty} \beta^\cS_A(x,r_i)=0\end{align*} by Lemma \ref{l:bp}. Hence, by Theorem \ref{t:betatan}, $\Tan(B,0)\subseteq (\cS;0)^\beta_{0,\infty}$. Pick any $C\in\Tan(B,0)$. Then $C\in \Tan(\cl{A},x)$ by Lemma \ref{l:iterate}. Therefore, $C\in \Tan(\cl{A},x)\cap (\cS;0)^\beta_{0,\infty}$. In particular, $\Tan(\cl{A},x)\cap (\cS;0)^\beta_{0,\infty}\neq\emptyset$.\end{proof}

The next result is a unilateral variant of Theorem \ref{PsTanVsPerturb}.

\begin{theorem} \label{t:lua} Let $\cS$ be a local approximation class and let $A\subseteq\RR^n$ be a nonempty set. Then the following are equivalent:
\begin{enumerate}
\item $A$ is locally unilaterally $\varepsilon'$-approximable by $\cS$ for all $\varepsilon'>\varepsilon$;
\item $\PsTan(\cl{A},x)\subseteq(\cS;\varepsilon)^\beta_{0,\infty}$ for all $x\in A$;
\item $\PsTan(\cl{A},x)\subseteq(\cS;\varepsilon)^\beta_{\RR^n,\infty}$ for all $x\in A$.\end{enumerate} \end{theorem}

\begin{proof} Let $A\subseteq\RR^n$ be nonempty and fix $\varepsilon\geq 0$. Statements $(ii)$ and $(iii)$ are equivalent by Lemma \ref{l:enlarge2}, because $\PsTan(\cl{A},x)$ is a translation invariant local approximation class. To complete the proof, we will show that $(i)$ is equivalent to $(ii)$.

Suppose that $A$ is locally unilaterally $\varepsilon'$-approximable by $\cS$ for all $\varepsilon'>\varepsilon$. Then \begin{equation}\label{e:bquiv1} \limsup_{r\downarrow 0} \sup_{x\in K} \beta^\cS_A(x,r)\leq \varepsilon\end{equation} for every compact subset $K\subseteq A$. Let $T\in \PsTan(\cl{A},x)$ for some $x\in A$, say that $T=\lim_{i\rightarrow\infty} (\cl{A}-y_i)/r_i$ for some sequences $y_i\in \cl{A}$ and $r_i>0$ such that $y_i\to x$ and $r_i\rightarrow 0$. For each $i\geq 1$, choose $x_i\in A$ such that $|x_i-y_i|\leq r_i/i$. Then, since for all $s>0$, $$\mD{0}{s}\left[\frac{\cl{A}-x_i}{r_i},T\right] \leq \frac{1}{s}\left|\frac{x_i-y_i}{r_i}\right|+2\mD{0}{2s}\left[\frac{\cl{A}-y_i}{r_i},T\right],$$ we have that $(\cl{A}-x_i)/r_i\rightarrow T$ as well. Fix a scale $r>0$ and an error $\delta>0$. By Lemma \ref{l:bp} and (\ref{e:bquiv1}) applied with the compact set $K=\{x\}\cup \{x_i:i\geq 1\}\subseteq A$, \begin{align*}\beta_T^{\cS}(0,r) &\leq (1+\delta)\liminf_{i\rightarrow\infty} \beta_{(\cl{A}-x_i)/r_i}^\cS(0,(1+\delta)r)= (1+\delta)\liminf_{i\rightarrow\infty} \beta_\cl{A}^\cS(x_i,(1+\delta)rr_i) \\ &\leq (1+\delta)^2\liminf_{i\rightarrow\infty} \beta_A^\cS(x_i,(1+\delta)^2rr_i) \leq (1+\delta)^2\varepsilon.\end{align*} Letting $\delta\rightarrow 0$ yields $\beta^{\cS}_T(0,r)\leq \varepsilon$ for all $r>0$. This shows that $T\in (S;\varepsilon)^\beta_{0,\infty}$ and $\PsTan(\cl{A},x)\subseteq (\cS;\varepsilon)^\beta_{0,\infty}$ for all $x\in A$. Therefore,    $(i)\Rightarrow(ii)$.

Conversely, suppose that $\PsTan(\cl{A},x)\subseteq(S;\varepsilon)^\beta_{0,\infty}$ for all $x\in A$.  Fix a compact set $K \subseteq A$. Let $x_i\in K$ and $r_i\rightarrow 0$ be sequences such that $$\lim_{i\rightarrow\infty} \beta^\cS_A(x_i,r_i)= \limsup_{r\downarrow 0} \sup_{x\in K}\beta^\cS_A(x, r).$$ Passing to a subsequence of $(x_i,r_i)_{i=1}^\infty$, we may assume (since $K$ and $\CL{0}$ are sequentially compact) that $x_i\rightarrow x$ for some $x\in K$ and $(\cl{A}-x_i)/r_i\rightarrow T$ for some $T\in\PsTan(\cl{A},x)$.  By Lemma \ref{l:bp}, \begin{align*}\lim_{i\rightarrow\infty} \beta_A^\cS(x_i,r_i) &\leq \limsup_{i\rightarrow\infty} \beta_{\cl{A}}^\cS(x_i,r_i) \\
&=\limsup_{i\rightarrow\infty} \beta_{(\cl{A}-x_i)/r_i}^\cS (0,1) \leq (1+\delta) \beta_{T}^\cS(0,1+\delta)\leq (1+\delta)\varepsilon\end{align*} for all $\delta>0$, where the last inequality holds since $T\in(\cS;\varepsilon)^\beta_{0,\infty}$. Letting $\delta\rightarrow 0$ implies that $\limsup_{r\rightarrow 0}\sup_{x\in K} \beta_A^\cS(x,r) = \lim_{i\rightarrow\infty} \beta_A^\cS(x_i,r_i)\leq \varepsilon$. That is, $A$ is locally unilaterally $\varepsilon'$-approximable by $\cS$ for all $\varepsilon'>\varepsilon$. Therefore, $(ii)\Rightarrow (i)$.
\end{proof}

\begin{remark}By substituting beta numbers for theta numbers in relevant definitions and proofs, one can obtain unilateral variants of the results in \S\S \ref{s:connect}--\ref{s:decomp}. For example, unilateral connectedness of the cone of tangents sets: \emph{Suppose there exists $\phi>0$ such that $\limsup_{r\rightarrow\infty} \beta^\cT_S(0,r)\geq \phi$ for all $S\in \cS\setminus (\cT;0)^\beta_{0,\infty}$. If $A\in\CL{x}$ and $\Tan(A,x)\subseteq \cS$, then} $$\Tan(A,x)\subseteq (\cT;0)^\beta_{0,\infty}\quad\text{or}\quad\Tan(A,x)\subseteq \{S\in\cS: \liminf_{r\downarrow 0} \beta_S^\cT(0,r)>0\}.$$ The underlying reason that results about bilateral approximation can be transferred to results about unilateral approximation is that theta and beta numbers satisfy the same essential properties such as scale invariance, monotonicity, etc. We leave the details to the interested reader. \end{remark}

\subsection{Unilateral approximation of the singular parts of Reifenberg type sets}
For the remainder of this section, we shall assume that $\cS$ and $\cT$ are local approximation classes such that \begin{equation} \label{e:rs} \hbox{$\ccS$ is translation invariant, and $\cT$ points are $(\phi,\Phi)$ detectable in $\cS$.}\end{equation}
Since $\ccS$ is translation invariant, every $X\in\ccS$ is (globally) bilaterally well approximated by $\cS$ by Lemma \ref{l:enlarge}. Thus, since $\cT$ points are $(\phi,\Phi)$ detectable in $\cS$, Theorem \ref{t:open} implies that every $X\in\ccS$ can be decomposed $X=X_\ccT\cup X_\cTp$, where $$X_\ccT=\{x\in X:\PsTan(X,x)\subseteq\ccS\cap \ccT\}$$ is a relatively open set in $X$ and $$X_{\cTp}=\{x\in X:\Tan(X,x)\subseteq\ccS\cap \cTp\}$$ is a closed set in $\RR^n$ (because $X$ is closed in $\RR^n$).

\begin{definition}\label{sing} Assume (\ref{e:rs}). We define the class of \emph{$\cT$ singular parts of sets in $\ccS$} by
$\sing\cT\ccS = \{X_\cTp: X\in \ccS \text{ and } 0\in X_\cTp\}.$
\end{definition}

\begin{example} \label{ex:m32}
Let $\cM$ denote the collection of all translates of 2-dimensional Almgren minimal cones in $\RR^3$. This class has three types of sets---planes, $Y$-type sets, and $T$-type sets---as described in the introduction and redisplayed below in Figure \ref{figure:re:minCones}.

\begin{figure}[htb]
	\begin{center}
		\includegraphics[width=.7\textwidth]{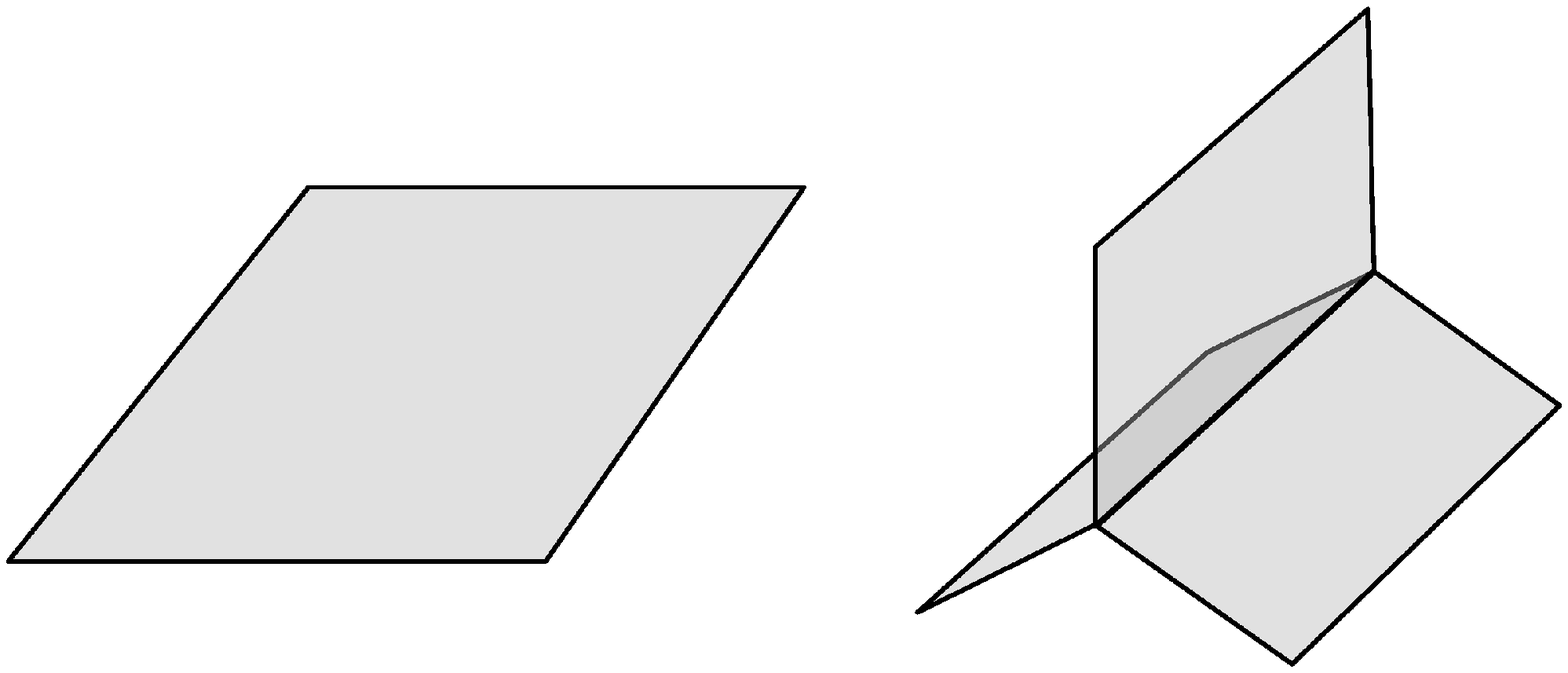}
	\end{center}
	\caption{The class $\cM$ consists of planes, $Y$-type sets, and $T$-type sets.}
	\label{figure:re:minCones}
\end{figure}

\noindent Let $\cG=\cG(3,2)$ denote the collection of planes in $\cM$, and let $\cY$ denote the collection of (translates of) $Y$-type sets in $\cM$. The reader may verify that $\cG$ points are detectable in $\cG\cup\cY$ and in $\cM$; and $\cG\cup\cY$ points are detectable in $\cM$. The class $\sing\cG (\cG\cup\cY)= \cG(3,1)$ consists of lines through the origin. The class $\sing\cG\cM$ consists of lines through the origin and spines of $T$-type sets (see Figure \ref{figure:singGM}). Meanwhile, the class $\sing{\cG\cup\cY}\cM=\{\{0\}\}$ solely consists of a singleton at the origin.
\begin{figure}[htb]
	\begin{center}
		\includegraphics[width=.38\textwidth]{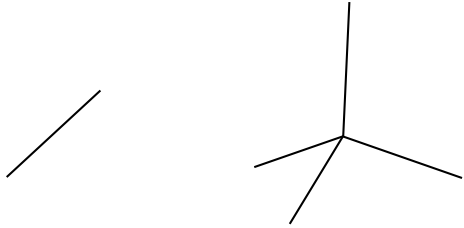}
	\end{center}
	\caption{The class $\sing\cG\cM$ consists of lines and spines of $T$-type sets.}
	\label{figure:singGM}
\end{figure}
\end{example}

\begin{theorem}\label{singapprox} Assume (\ref{e:rs}). If $A\subseteq\RR^n$ is locally bilaterally well approximated by $\cS$, then $A_\cTp$ is locally unilaterally well approximated by $\sing\cT\ccS$.
\end{theorem}

\begin{proof} We will show that $\PsTan(\cl{A_{\cTp}},x)\subseteq (\sing\cT\ccS;0)^\beta_{0,\infty}$ for all $x\in A_\cTp$, so that $A_\cTp$ is locally unilaterally well approximated by $\sing\cT\ccS$ by Theorem \ref{t:lua}. Let $x\in A_\cTp$ and let $\Sigma\in \PsTan(\cl{A_\cTp},x)$, say $\Sigma=\lim_{i\rightarrow\infty}(\cl{A_\cTp}-x_i)/r_i$ for some sequences $x_i\in \cl{A_\cTp}$ and $r_i>0$ such that $x_i\rightarrow x$ and $r_i\rightarrow 0$. Passing to a subsequence of $(x_i,r_i)_{i=1}^\infty$, we assume that $(\cl{A}-x_i)/r_i\rightarrow X$ for some $X\in \PsTan(\cl{A},x)\subseteq\ccS$. To proceed, we will show that $\Sigma\subseteq X_\cTp$.

Let $y\in \Sigma$. Because $(\cl{A_\cTp}-x_i)/r_i\rightarrow \Sigma$, there exists a sequence $y_i=(z_i-x_i)/r_i$ with $z_i\in \cl{A_\cTp}$ such that $y_i\rightarrow y$. Replacing each $z_i$ by some $z_i'\in A_\cTp$ such that $|z_i'-z_i|\leq r_i/i$, we may assume that $z_i\in A_\cTp$ for all $i\geq 1$. Note that $z_i\rightarrow x$ and  $$\frac{\cl{A}-z_i}{r_i} = \frac{\cl{A}-x_i}{r_i} + \frac{x_i-z_i}{r_i}\to X - y \quad\text{in }\CL{\RR^n},$$ because $(\cl{A}-x_i)/r_i\rightarrow X$ in $\CL{\RR^n}$ and $(z_i-x_i)/{r_i}\rightarrow y$ (see e.g.~ the proof of Lemma \ref{inv_of_tan}). Let $\cS' := \ccS\cap \cTp=\{S\in\ccS: \Theta^\cT_S(0,r)\geq \phi$ for all $r>0\}$. By Corollary \ref{S-T}, the set $A$ is locally bilaterally well approximated along $A_\cTp$ by $\cS'$. Thus, since $K=\{x\}\cup\{z_i:i\geq 1\}$ is a compact subset of $A_\cTp$ (because $z_i\rightarrow x$), we obtain $$\limsup_{i\rightarrow\infty}\Theta_{\cl A}^{\cS'}(z_i,sr_i)\leq 2\limsup_{i\rightarrow\infty}\Theta^{\cS'}_A(z_i,2sr_i)\leq 2 \limsup_{r\downarrow 0}\sup_{z\in K} \Theta^{\cS'}_A(z,r)=0\quad\text{for all }s>0.$$ By Lemma \ref{l:bp}, it follows that $\Theta^{\cS'}_{X-y}(0,s)=0$ for all $s>0$, whence $X-y \in \overline{\cS'}=\cS'$ and $\Theta^\cT_X(y,r)=\Theta^\cT_{X-y}(0,r)\geq \phi$ for all $r>0$. Hence $y\in X_\cTp$ for all $y\in\Sigma$. Thus, $\Sigma\subseteq X_\cTp$. In particular, note that $0\in X_\cTp$, since $0\in \Sigma$. Therefore, $\Sigma\subseteq X_\cTp\in \sing\cT\ccS$.

We have shown that for all $x\in A_\cTp$ and for all pseudotangent sets $\Sigma\in \PsTan(\cl{A_\cTp},x)$, there exists $X_{\cTp}\in \sing\cT\ccS$ such that $\Sigma\subseteq X_\cTp$. Thus, $\PsTan(\cl{A_\cTp},x)\subseteq (\sing\cT\ccS;0)^{\beta}_{0,\infty}$ for all $x\in A_{\cTp}$ by Lemma \ref{l:0b}. Therefore, $A_\cTp$ is locally unilaterally well approximated by $\sing\cT\ccS$ by Theorem \ref{t:lua}.\end{proof}

\section{Dimension bounds for Mattila-Vuorinen type sets}

The defining property of Mattila-Vuorinen type sets interacts naturally with the concept of Minkowski dimension. There are several equivalent definitions of Minkowski dimension; for general background, we refer the reader to \cite{Mattila95}.

\begin{definition}[covering numbers] Let $A\subseteq\RR^n$, let $x\in\RR^n$, and let $r,s>0$. Define the \emph{(intrinsic) $s$-covering number of $A$} by \begin{equation*}\covplain(A,s)=\min\left\{k\geq 0:A\subseteq \bigcup_{i=1}^k \ball(a_i,s)\text{ for some }a_i\in A\right\},\end{equation*} and define
the \emph{$s$-covering number of $A$ in $\ball(x,r)$} by $\cov{x}{r}(A,s)= \covplain(A\cap \ball(x,r),s)$.
\end{definition}

\begin{definition}[Minkowski dimension via covering numbers] For bounded sets $A\subseteq\RR^n$, the \emph{upper Minkowski dimension of $A$} is given by \begin{equation*}
 \udim_M(A)=\limsup_{s\downarrow 0} \frac{\log\left(\covplain(A,s)\right)}{\log(1/s)}.
\end{equation*} For unbounded sets $A\subseteq\RR^n$, the \emph{upper Minkowski dimension of $A$} is given  by $$\udim_M(A)=\lim_{t\uparrow\infty} \left(\udim_M A\cap \ball(0,t)\right).$$\end{definition}

\begin{remark} To compute the upper Minkowski dimension of a bounded set, it suffices to examine covering numbers along a geometric sequence of scales: $$\udim_M(A)=\limsup_{k\rightarrow\infty} \frac{\log(\covplain(A,\lambda^k r_0))}{\log (1/(\lambda^kr_0))}\quad\text{for all }0<\lambda<1\text{ and }r_0>0.$$\end{remark}

In order to control Minkowski dimension using unilateral approximability, we will need to assume uniform control on local covering numbers.

\begin{definition}[covering profiles]\label{coverings}
Let $\cS$ be a local approximation class. For all $\alpha>0$, $C>0$, and $s_0\in(0,1]$, we say that $\cS$ has an \emph{$(\alpha, C, s_0)$ covering profile} if $\cov{0}{r}(S,sr) \leq Cs^{-\alpha}$ for all $S\in\cS$, $r>0$, and $s\in (0,s_0]$.
\end{definition}

The following lemma underpins the dimension estimates that we prove below.

\begin{lemma}\label{CoveringLm}
Let $\cS$ be a local approximation class with an $(\alpha, C, s_0)$ covering profile. Given $\delta>0$ such that $C^{1/\alpha}\delta\leq s_0$, let \begin{equation}\label{e:ldef} \lambda=\delta\left(2+2C^{1/\alpha}(1+\delta)\right).\end{equation} If $x\in A\subseteq\RR^n$ and $\beta^\cS_A(x,r)<\delta$ for some $r>0$, then $\cov{x}{r}(A, \lambda r) \leq 1/\delta^\alpha$.
\end{lemma}

\begin{proof} Suppose that $\cS$ is a local approximation class with an $(\alpha, C,s_0)$ covering profile. Without loss of generality, suppose that $x=0$ and $r=1$, $0\in A\subseteq\RR^n$, and $\beta^\cS_A(0,1)<\delta$ for some $\delta\in (0, s_0/C^{1/\alpha}]$. Because $\beta^\cS_A(0,1)<\delta$, there exists $S\in\cS$ such that $\mud 0 1 (A, S) < \delta$.
To continue, let $\mathcal{B}=\{B_1,\dots, B_k\}$ be a minimal cover of $S\cap \ball(0,1+\delta)$ by balls of radius $C^{1/\alpha}\delta(1+\delta)$ with centers in $S\cap \ball(0,1+\delta)$. Because $\cS$ has an $(\alpha, C, s_0)$ covering profile and $C^{1/\alpha}\delta \leq s_0$, it follows that
$$k=\cov 0 {1+\delta}(S, C^{1/\alpha}\delta(1+\delta))\leq C (C^{1/\alpha}\delta)^{-\alpha}  = 1/\delta^{\alpha}.$$ For each $1\leq i\leq k$, let $z_i$ denote the center of $B_i$.
We claim that $\{\ball(z_i, \lambda/2):1\leq i\leq k\}$ covers $A\cap \ball(0,1)$. To check this, pick any $y\in A\cap \ball(0,1)$. Then, since $\mud{0}{1}(A,S)<\delta$, there exists $x\in S\cap \ball(0,1+\delta)$ such that $|x-y|\leq \delta$. But $S\cap \ball(0,1+\delta)\subseteq\bigcup \mathcal{B}$, so there exists $1\leq i\leq k$ such that $|x-z_i|\leq  C^{1/\alpha}\delta(1+\delta).$ By the triangle inequality, $|y-x_i|\leq\delta + C^{1/\alpha}\delta(1+\delta)=\lambda/2$. Thus, for all $y\in A\cap\ball(0,1)$ there exists $1\leq i\leq k$ such that $y\in \ball(z_i,\lambda/2)$. Finally, for each $1\leq i\leq k$ such that $\ball(z_i,\lambda/2)$ intersects $A\cap \ball(0,1)$, pick $a_i\in \ball(z_i,\lambda/2)\cap \left(A\cap \ball(0,1)\right)$. Then $A\cap \ball(0,1) \subseteq \bigcup_i \ball(a_i,\lambda)$ by the triangle inequality. Therefore, $N^{0,1}(A,\lambda) \leq k\leq 1/\delta^{\alpha}$.\end{proof}

We now give a dimension bound, which is the main result of this section.

\begin{theorem}\label{t:dim} For all covering profiles $(\alpha, C, s_0)$, there exist $\varepsilon_0>0$ and $C'>0$  such that if $\cS$ is a local approximation class with an $(\alpha,C, s_0)$ covering profile and $A\subseteq\RR^n$ is unilaterally $(\varepsilon,r_0)$-approximable by $\cS$ for some $0<\varepsilon\leq \varepsilon_0$, then \begin{equation}\label{e:dim} \udim_M(A) \leq \alpha + \frac{C'}{\log (1/\varepsilon)}.\end{equation}\end{theorem}

\begin{corollary}\label{c:dim} If $\cS$ is a local approximation class with an $(\alpha,C, s_0)$ covering profile and $A\subseteq\RR^n$ is unilaterally $(\varepsilon,r_\varepsilon)$-approximable by $\cS$ for all $\varepsilon>0$, then $\udim_M A\leq \alpha$.
\end{corollary}

\begin{proof}[Proof of Theorem \ref{t:dim}] Let $\alpha>0$, $C>0$, and $s_0\in(0,1]$ be given, and fix a parameter $\varepsilon_0>0$ to be specified later. Let $\cS$ be any local approximation class with an $(\alpha, C, s_0)$ covering profile, and let $A\subseteq\RR^n$ be unilaterally $(\varepsilon,r_0)$-approximable by $\cS$ for some $0<\varepsilon\leq \varepsilon_0$ and $r_0>0$.

Choose $t>0$ large enough so that $A^t=A\cap \ball(0,t)\neq\emptyset$. Since $\ball(0,t)$ is compact, its subset $A^t$ is totally bounded. Thus, we can cover $A^t$ by finitely many balls $\mathcal{B}=\{B_1,\dots,B_l\}$ of radius $r_0$ with centers $\{a_1,\dots,a_l\}$ in $A^t$. Let $\delta\in (\varepsilon,2\varepsilon_0]$ be arbitrary, but eventually close to $\varepsilon$, and define $\lambda$ by (\ref{e:ldef}). We now specify that $\varepsilon_0>0$ be chosen sufficiently small so that \begin{equation}\label{e:r1} \lambda \leq \delta\left(2+2C^{1/\alpha}(1+2\varepsilon_0)\right) \leq 2\varepsilon_0\left(2+2C^{1/\alpha}(1+2\varepsilon_0)\right)\leq 1/2\end{equation} and \begin{equation}\label{e:r2} C^{1/\alpha}\delta \leq C^{1/\alpha}(2\varepsilon_0)\leq s_0.\end{equation} Fix $1\leq i \leq l$. Because $\cS$ has an $(\alpha, C,s_0)$ covering profile, $A$ is unilaterally $(\varepsilon,r_0)$-approximable by $\cS$, $\varepsilon<\delta$, and (\ref{e:r2}) holds, induction on Lemma \ref{CoveringLm} yields
$$\cov{a_i}{r_0}(A^t, \lambda^k r_0) \leq 1/\delta^{k\alpha}\quad\text{for all }k=1,2,\dots.$$ Hence, for all $1\leq i\leq l$, we have $$\udim_M(A^t\cap B_i)=\limsup_{k\rightarrow\infty} \frac{\log\left(\covplain(A^t\cap B_i, \lambda^kr_0)\right)}{\log( 1/(\lambda^kr_0))}\leq \limsup_{k\rightarrow\infty} \frac{\log(1/\delta^{k\alpha})}{\log(1/(\lambda^kr_0))}=\alpha\left(\frac{\log(1/\delta)}{\log(1/\lambda)}\right).$$ Thus, $\udim_M (A^t)\leq \alpha \log(1/\delta)/\log(1/\lambda)$, because Minkowski dimension is stable under finite unions. Observe that $\log \lambda \leq \log \delta+ \log \left(2+2C^{1/\alpha}(1+2\varepsilon_0)\right)$ by (\ref{e:r1}).  Writing $$\log (\mu_0)=\log\left(2+2C^{1/\alpha}(1+2\varepsilon_0)\right),$$ it follows that $$\udim_M(A^t) \leq \alpha\left(\frac{\log(1/\delta)}{\log(1/\delta)-\log(\mu_0)}\right) = \alpha\left/\left(1-\frac{\log(\mu_0)}{\log(1/\delta)}\right)\right..$$ We now declare that, in addition to (\ref{e:r1}) and (\ref{e:r2}), the parameter $\varepsilon_0$ be chosen sufficiently small so that \begin{equation}\label{e:r3}\frac{\log(\mu_0)}{\log(1/\delta)} \leq \frac{\log(\mu_0)}{\log(1/(2\varepsilon_0))} \leq 1/2.\end{equation} Therefore, since $1/(1-x)\leq 1+2x$ for all $0\leq x\leq 1/2$, $$\udim_M(A^t) \leq \alpha+\frac{2\alpha\log\mu_0}{\log(1/\delta)}\quad\text{for all }\delta\in (\varepsilon,2\varepsilon_0]\text{ and }t\gg 1.$$ Letting $\delta\downarrow \varepsilon$ and $t\uparrow\infty$, we conclude that $\udim_M(A) \leq \alpha + C'/\log(1/\varepsilon)$, where $C'=2\alpha \log(\mu_0)=2\alpha\log\left(2+2C^{1/\alpha}(1+2\varepsilon_0)\right)$ ultimately depends only on $\alpha$, $C$, and $s_0$.\end{proof}

\begin{corollary}\label{c:dim2} If $\cS$ is a local approximation class with an $(\alpha, C, s_0)$ covering profile and $A\subseteq\RR^n$ is closed and locally unilaterally $\varepsilon$-approximable by $\cS$ for some $0<\varepsilon\leq \varepsilon_0$, then  (\ref{e:dim}) holds.\end{corollary}

\begin{corollary}\label{c:dim3} If $\cS$ is a local approximation class with an $(\alpha, C, s_0)$ covering profile and $A\subseteq\RR^n$ is closed and locally unilaterally well approximated by $\cS$, then $\udim_M(A)\leq \alpha$.\end{corollary}

\begin{proof}[Proof of Corollary \ref{c:dim2}] Assume that $\cS$ has an $(\alpha, C,s_0)$ covering profile, $A\subseteq\RR^n$ is closed, and $A$ is locally $\varepsilon$-approximable by $\cS$ for some $0<\varepsilon\leq \varepsilon_0$. Fix $t>0$. The truncation $A^t=A\cap \ball(0,t)$ is a compact subset of $A$, because $A$ is closed. Hence $A^t$ is unilaterally $(\varepsilon,r_t)$-approximable by $\cS$ for some $r_t>0$, since $A$ is locally unilaterally $\varepsilon$-approximable by $\cS$. Thus, we have $\udim(A^t)\leq \alpha+C'/\log(1/\varepsilon)$ by Theorem \ref{t:dim}. Letting $t\rightarrow\infty$ yields (\ref{e:dim}).\end{proof}

\begin{remark} Recall that when $\cS=\cG(n,m)$, which admits an $(m,C,s_0)$ covering profile, Mattila and Vuorinen \cite{MV} proved that if $A$ is unilaterally $(\varepsilon,r_0)$-approximable by $\cS$ for some $\varepsilon>0$ sufficiently small, then $\udim_M(A)\leq m+ C''\varepsilon^2$. By contrast, Theorem \ref{t:dim} only guarantees the weaker bound $\udim_M(A)\leq m+ C'/\log(1/\varepsilon)$. The reason for this disparity is that the proof the Mattila-Vuorinen bound relies on the Hilbert space geometry of $\RR^n$, while the proof of Theorem \ref{t:dim} holds more generally in proper metric spaces that are translation and dilation invariant in the sense described in Remark \ref{r:geo1}. \end{remark}

For all $A\subseteq\RR^n$, let $\dim_H(A)$ denote the Hausdorff dimension of $A$ (see e.g.~ \cite{Mattila95}). It is well known that Hausdorff dimension is dominated by Minkowski dimension, $$\dim_H (A)\leq \udim_M (A),$$ and Hausdorff dimension is stable under countable unions, $$\dim_H \left(\bigcup_{i=1}^\infty A_i\right) = \sup \{\dim_H (A_i):i=1,2,\dots\}.$$ These properties allow one to transfer the Minkowski dimension bound in Theorem \ref{t:dim} to a Hausdorff dimension bound on a large class of locally unilaterally approximable sets that satisfy a certain topological size condition. In particular, the following dimension bounds may be applied to the set $A_\cQ$ in Theorem \ref{t:perturb} and the set $A_\ccT$ in Theorem \ref{t:open}, whenever $A$ is closed.

\begin{corollary}\label{c:dim4} Let $\cS$ be a local approximation class with an $(\alpha,C,s_0)$ covering profile, and let $A\subseteq\RR^n$. If $A$ is locally unilaterally $\varepsilon$-approximable by $\cS$ for some $0<\varepsilon\leq \varepsilon_0$ and the subspace topology on $A$ is $\sigma$-compact, then $$\dim_H (A) \leq \alpha + \frac{C'}{\log 1/\varepsilon}.$$\end{corollary}

\begin{corollary}\label{c:dim5} Let $\cS$ be a local approximation class with an $(\alpha,C,s_0)$ covering profile, and let $A\subseteq\RR^n$. If $A$ is locally unilaterally well approximated by $\cS$ and the subspace topology on $A$ is $\sigma$-compact, then $\dim_H (A)\leq \alpha$.\end{corollary}

\section{Applications and future directions} \label{s:app}

To conclude, we illustrate the use of results developed above in a few specific instances. We then discuss a few open problems and directions for future research.

\subsection{Asymptotically optimally doubling measures}

Let $\cC_{KP}$ denote the collection of all rotations of the light cone $\{x\in\RR^n:x_1^2+x_2^2+x_3^2=x_4^2\}$.

\begin{theorem} Suppose $n\geq 4$. If $\mu$ is an $(n-1)$-asymptotically optimally doubling measure in $\RR^n$, then $\spt \mu = G\cup S$, where $\PsTan(\spt\mu, x)\subseteq \cG(n,n-1)$ for all $x\in G$ and $\Tan(\spt\mu,x)\subseteq \cC_{KP}$ for all $x\in S$. Moreover, $S$ is closed and locally unilaterally well approximated by $\cG(n,n-4)$, and hence, $S$ has upper Minkowski dimension at most $n-4$. \end{theorem}

\begin{proof} Suppose $n\geq 4$. Let $\cU(n,n-1)$ denote the collection of supports of $(n-1)$-uniform measures in $\RR^n$ that contain the origin. Then $\cU(n,n-1)$ is a closed translation invariant local approximation class. Recall that $\cU(n,n-1)$ consists of hyperplanes and translates of cones in $\cC_{KP}$ by the classification of Kowalski and Preiss \cite{KP}. Since the light cone $\{x\in\RR^n:x_1^2+x_2^2+x_3^2=x_4^2\}$ is smooth away from $\{0\}^{4}\times\RR^{n-4}$, it follows that $$\sing{\cG(n,n-1)}{\cU(n,n-1)}= \cG(n,n-4).$$ Now assume that $\mu$ is an $(n-1)$-asymptotically optimally doubling measure in $\RR^n$. Then $\spt\mu$ is locally bilaterally well approximated by $\cU(n,n-1)$ by \cite{Lewis}*{Theorem 3.8}. But $\cG(n,n-1)$ points are detectable in $\cU(n,n-1)$ by \cite{Lewis}*{Lemma 2.5}. By Theorem \ref{t:open} and Theorem \ref{singapprox}, it follows that $\spt\mu=(\spt\mu)_\ccG \cup (\spt\mu)_\cGp$, where \begin{itemize} \item $\PsTan(\spt\mu,x)\subseteq\cG(n,n-1)$ for all $x\in (\spt \mu)_{\ccG}$, \item $\Tan(\spt\mu,x)\subseteq \cU(n,n-1)\cap \cG(n,n-1)^\perp=\cC_{KP}$ for all $x\in(\spt\mu)_\cGp$, and \item $(\spt\mu)_\cGp$ is closed and locally unilaterally well approximated by $\cG(n,n-4)$.\end{itemize} Therefore, $(\spt\mu)_\cGp$ has upper Minkowski dimension at most $n-4$ by Corollary \ref{c:dim3}. \end{proof}

\subsection{Free boundary regularity for harmonic measure from two sides} For all $n\geq 2$ and $d\geq 1$, let $\cH^*(n,d)$ denote the collection of zero sets of nonconstant harmonic polynomials $p:\RR^n\rightarrow\RR$ of degree at most $d$ such that $p(0)=0$ and such that the positive set $\{p>0\}$ and the negative set $\{p<0\}$ are connected. For all $n\geq 2$ and $1\leq k\leq d$, let $\cF^*(n,k)\subset\cH^*(n,d)$ denote the subcollection of zero sets of \emph{homogeneous} harmonic polynomials of degree $k$. For general background and an explanation of the terminology in the following theorem, see \cite{Badger3}*{\S6}.

\begin{theorem}[\cite{Badger3}*{Theorem 6.8}] Assume \label{t:bad3} that $\Omega\subset\RR^n$ is a 2-sided NTA domain with interior harmonic measure $\omega^+$ on $\Omega^+=\Omega$ and exterior harmonic measure $\omega^-$ on $\Omega^-=\RR^n\setminus\overline{\Omega}$. If $\omega^+\ll\omega^-\ll\omega^+$ and $\log(d\omega^-/d\omega^+)\in\mathrm{VMO}(d\omega^+)$, then there is $d_0>0$ such that $\partial\Omega$ is locally bilaterally well approximated by $\cH^*(n,d_0)$. Moreover, $\partial\Omega$ can be partitioned into sets $\Gamma_d$ ($1\leq d\leq d_0$), $$\partial\Omega=\Gamma_1\cup\dots\cup\Gamma_{d_0},\quad i\neq j\Longrightarrow \Gamma_i\cap \Gamma_j=\emptyset,$$ with the following properties. \begin{enumerate}
\item $\Tan(\partial\Omega,x)\subseteq \cF^*(n,d)$ for all $x\in\Gamma_d$.
\item $\Gamma_1$ is relatively open and dense in $\partial\Omega$, $\Gamma_1$ is locally bilaterally well approximated by $\cF(n,1)=\cG(n,n-1)$, and hence, $\Gamma_1$ has Hausdorff dimension at most $n-1$.
\item The set of ``singularities" $\partial\Omega\setminus\Gamma_1 = \Gamma_2\cup\dots\cup\Gamma_{d_0}$ is closed and has harmonic measure zero.
\end{enumerate}
\end{theorem}

Combining the results of \S\S 7.1--8 from above with estimates on the Minkowski content of zero sets of harmonic functions from Naber and Valtorta \cite{NV14} yields  new bounds on the dimension of the free boundary $\partial\Omega$ and the singular set $\partial\Omega\setminus \Gamma_1$ in Theorem \ref{t:bad3}.

\begin{theorem} \label{t:ssdim} Under the setup of Theorem \ref{t:bad3}, $\partial\Omega$ has upper Minkowski dimension at most $n-1$ and the singular set $\partial\Omega\setminus \Gamma_1$ has upper Minkowski dimension at most $n-2$.
\end{theorem}

\begin{proof} Let $\cH(n,d)$ denote the zero sets $\Sigma_p$ of nonconstant harmonic polynomials $p$ in $\RR^n$ of degree at most $d$, and let $\cS\cH(n,d)=\{S_p=\Sigma_p\cap |Dp|^{-1}(0):\Sigma_p\in\cH(n,d),\ 0\in S_p\}$ denote the singular sets of nonconstant harmonic polynomials in $\RR^n$ of degree at most $d$. Applied to harmonic polynomials, \cite{NV14}*{Theorem 3.37} says that \begin{equation}\mathrm{Vol}\left(\{x\in \ball(0,1/2): \dist(x,S_p)\leq r\}\right) \leq C(n)^{d^2}r^2\quad\text{for all }S_p\in\cS\cH(n,d).\end{equation} In addition, \cite{NV14}*{Theorem A.3} gives  \begin{equation}\label{e:Mc2} \mathrm{Vol}\left(\{x\in \ball(0,1/2):\dist(x,\Sigma_p)\leq r\}\right)\leq (C(n)d)^dr\quad\text{for all } \Sigma_p\in\cH(n,d).\end{equation} It easily follows that $\cS\cH(n,d)$ admits an $(n-2,C_{n,d},1)$ covering profile and $\cH(n,d)$ admits an $(n-1,\widetilde C_{n,d},1)$ covering profile for some constants $C_{n,d}, \widetilde{C}_{n,d}>1$ depending only on $n$ and $d$ (see e.g. \cite{Mattila95}*{(5.4) and (5.6)}).

Suppose that $\Omega\subseteq\RR^n$ satisfies the hypothesis of Theorem \ref{t:bad3}. Then $\partial\Omega$ is closed and locally bilaterally well approximated by $\cH(n,d)$ by Theorem \ref{t:bad3}. Hence $\udim_M(\partial\Omega)\leq n-1$ by Corollary \ref{c:dim3}, since $\cH(n,d)$ has an $(n-1,\widetilde{C}_{n,d},1)$ covering profile. Next note that the singular set $\partial\Omega\setminus\Gamma_1 = \partial\Omega_\cGp$ is closed and locally unilaterally well approximated by $$\sing{\cG(n,n-1)}{\cH(n,d)}=\cS\cH(n,d)$$ by Theorem \ref{t:open} and Theorem \ref{singapprox}, because $\cG(n,n-1)=\cF(n,1)$ points (``flat points") are detectable in $\cH(n,d)$ by \cite{Badger3}*{Theorem 1.4}. Therefore, $\udim_M(\partial\Omega\setminus\Gamma_1)\leq n-2$ by Corollary \ref{c:dim3}, since $\cS\cH(n,d)$ has an $(n-2,C_{n,d},1)$ covering profile. \end{proof}

\begin{remark} We do not currently know if the dimension bound for the singular set $\partial\Omega\setminus\Gamma_1$ in Theorem \ref{t:ssdim} is sharp, and in fact, the first author conjectured in \cite{Badger3}*{Remark 6.17} that $\partial\Omega\setminus\Gamma_1$ has dimension at most $n-3$. The reasoning behind the conjecture is that $\partial\Omega$ is actually locally bilaterally well approximated by $\cH^*(n,d)$ rather than just by $\cH(n,d)$. Examples (see \cite{Badger3}) suggest that the extra topological condition imposed on zero sets in $\cH^*(n,d)$ (i.e. that $\RR^n\setminus \Sigma$ has exactly two connected components when $\Sigma\in\cH^*(n,d)$) should force the smaller dimension bound. Thus, it would be interesting to know whether  \begin{equation} \mathrm{Vol}\left(\{x\in \ball(0,1/2): \dist(x,S_p)\leq r\}\right) \lesssim_{n,d}r^3\quad\text{for all }S_p\in\cS\cH^*(n,d),\end{equation} where $\cS\cH^*(n,d)=\{S_p:\Sigma_p\in\cH^*(n,d), 0\in S_p\}$ denotes the singular sets of harmonic polynomials in $\RR^n$ of degree at most $d$ whose zero set separates $\RR^n$ into two components.
\end{remark}

\subsection{Approximations by $\cH(2,2)$} Up to similarity, the sets in $\cH(2,2)$ are the zero sets $\Sigma_j$ of one of four harmonic polynomials $h_j:\RR^2\rightarrow\RR$ of degree 1 or 2: $h_1(x,y) = x$, $h_2(x,y) = xy$, $h_3(x,y) = (x+1)y$, and $h_4(x,y) = xy - (x+y)$. See Figure \ref{d2geom}. Although $\Sigma_2$ and $\Sigma_3$ are are translates of each other, the sets are distinguished by the fact that the origin is a \emph{flat point} ($\cG(2,1)$ point) of $\Sigma_3$, but the origin is not a flat point of $\Sigma_2$. Let $\cH'(2,2)\subset\cH(2,2)$ denote the subcollection of zero sets similar to $\Sigma_1$, $\Sigma_2$ or $\Sigma_3$; that is, $\cH'(2,2)$ excludes zero sets of type 4.
\begin{figure}[htb]
\begin{center}\includegraphics[width=.85\textwidth]{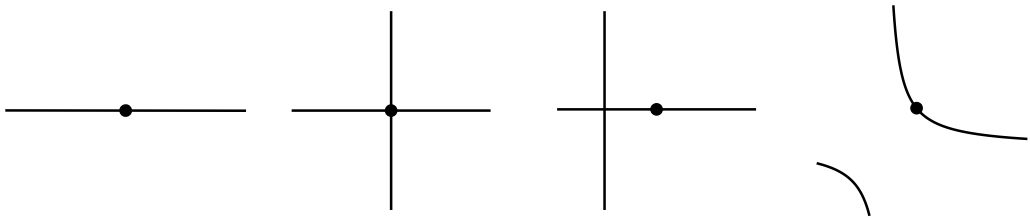}\end{center}
\caption{Zero sets of planar harmonic polynomials of degree 1 or 2. From left to right, they are types 1, 2, 3, and 4.}\label{d2geom}\end{figure}

By Theorem \ref{t:open}, if $A\subseteq\RR^2$ is locally bilaterally well approximated by $\cH(2,2)$, then $A=A_1\cup A_2$, where all pseudotangent sets of $\cl{A}$ at $x\in A_1$ are zero sets of type $1$ and all tangent sets of $\cl{A}$ at $x\in A_2$ are zero sets of type 2. Thus, all bounded psuedotangent sets of $\cl{A}$ at $x\in A_2$ are zero sets of type 2 or 3 by Lemma \ref{l:bdtan}. However, one may still ask what do unbounded pseudotangent sets of $\cl{A}$ look like at $x\in A_2$? In particular, can zero sets of type 4 appear as unbounded pseudotangent sets of $\cl{A}$ at $x\in A_2$? The answer is no as shown by the following theorem.

\begin{theorem} \label{t:fan} If $A\subseteq\RR^2$ is locally bilaterally well approximated by $\cH(2,2)$, then $A$ is locally bilaterally well approximated by $\cH'(2,2)$.\end{theorem}

\begin{proof} It remains to demonstrate that zero sets of type 4 do not appear as unbounded pseudotangent sets of $\cl{A}$ at $x\in A_2$. In fact, we will prove that $\uPsTan(\cl{A},x)\subseteq\cG(2,1)$ for all $x\in A_2$.  Thus, $\PsTan(\cl{A},x)\subseteq \cH'(2,2)$ for all $x\in A$, whence $A$ is locally bilaterally well approximated by $\cH'(2,2)$ by Corollary \ref{PsTanVsApprox}.

Without loss of generality, assume $0\in A_2$ and let $T$ be an unbounded pseudotangent set of $\cl{A}$ at $0$, say $T=\lim_{i\rightarrow\infty} (\cl{A}-x_i)/r_i$ for some sequences $x_i\in \cl{A}$ and $r_i>0$ such that $x_i\rightarrow 0$, $r_i\rightarrow 0$, and $|x_i|/r_i\rightarrow\infty$. Replacing each $x_i$ with a nearby $x_i'\in A$, we may assume without loss of generality that $x_i\in A$ for all $i\geq 1$ (see e.g.~ the proof of Theorem \ref{PsTanVsPerturb}). We will show that $\lim_{i\rightarrow\infty} \Theta^{\cG(2,1)}_A(x_i,sr_i)=0$ for all $s>0$, which implies that $T\in\cG(2,1)$ by Lemmas \ref{l:close}, \ref{ApproxVsLimits}, and \ref{l:enlarge}.

Fix $s>0$ (large) and $\gamma>0$ (small). Recall that $\cG(2,1)$ points are detectable in $\cH(2,2)$.  By Lemma \ref{l:Tpdp}, there exist numbers $\alpha,\beta>0$ such that if  $z\in A$, $\Theta^{\cH(2,2)}_A(z,r')<\alpha$ for all $0<r'\leq r$ and $$\Theta^{\cG(2,1)}_A(z,r)< \beta,$$ then $\Theta^{\cG(2,1)}_A(z,r')<\gamma$ for all $0<r'\leq r$.
We will apply this fact with $z=x_i$ and $r=|x_i|/16$ for large $i$. On one hand, because $A$ is locally well approximated by $\cH(2,2)$ and $K=\{x\}\cup \{x_i:i\geq 1\}$ is a compact subset of $A$, there exists $t_0>0$ such that $\Theta^{\cH(2,2)}_A(x_i,r')\leq \alpha/2<\alpha$ for all $i\geq 1$ and $0<r'\leq t_0$.  On the other hand, because $0\in A_2$, we can find $t_1\in(0,t_0]$ such that $\Theta_A^{\cT_2}(0,t)<\beta/80$ for all $0<t\leq t_1$, where $\cT_2$ denotes the class of sets of type 2.
\begin{figure}[htb]
\begin{center}\includegraphics[width=.6\textwidth]{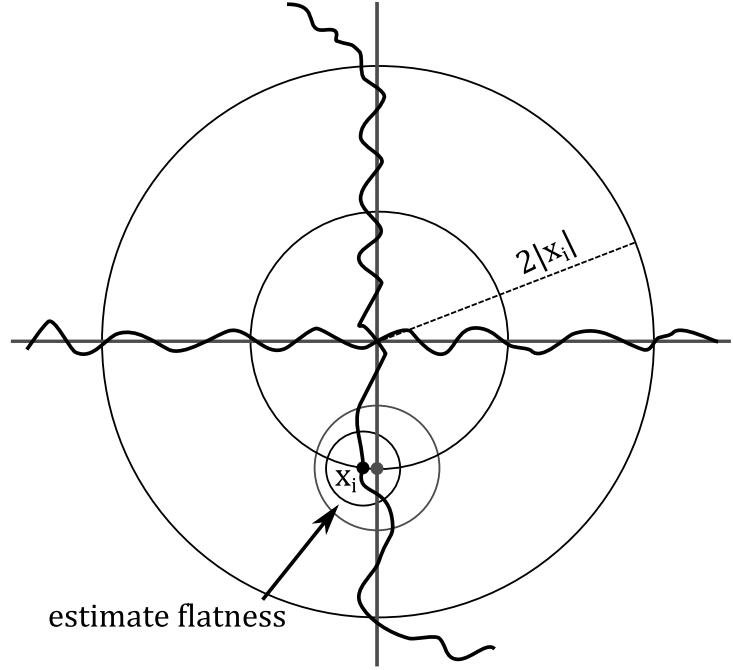}\end{center}
\caption{Estimating $\Theta^{\cG(2,1)}_A(x_i,|x_i|/16)$}\label{f:fan}\end{figure}
Because $x_i\rightarrow 0$ and $|x_i|/r_i\rightarrow\infty$, we can choose $i_0\geq 1$ so that $2|x_i|\leq t_1$ and $sr_i\leq |x_i|/16$ for all $i\geq i_0$. We claim that $\Theta^{\cG(2,1)}_A(x_i,|x_i|/16)<\beta$ for all $i\geq i_0$. To see this (refer to Figure 9.2), fix $i\geq i_0$ and pick a type 2  set $W\in\cT_2$ such that $\mD{0}{2|x_i|}[A,W]<\beta/80$. Then we can find $w\in W$ such that $|w-x_i|<(\beta/40)|x_i|$. Note that $W\cap B(w,|x_i|/2)$ is a line segment. Now, by monotonicity (Lemma \ref{l:tmono}), \begin{align*}\Theta^{\cG(2,1)}_A(x_i,|x_i|/16) &\leq 2\left(\beta/10+(1+\beta/10) \Theta^{\cG(2,1)}_A(w,(1+\beta/10)|x_i|/8)\right)\\ &\leq \beta/5 + 4\Theta_A^{\cG(2,1)}(w,|x_i|/4),\end{align*} since $B(x_i,|x_i|/16)\subseteq B(w,|x_i|/8)$ and $|x_i-w|<(\beta/10)|x_i|/8$. By the weak quasitriangle inequality and monotonicity for the relative Walkup-Wets distance, $$\Theta_A^{\cG(2,1)}(w,|x_i|/4) \leq 2\mD{w}{|x_i|/2}[A,W]+2\Theta^{\cG(2,1)}_W(w,|x_i|/2)
\leq 8\mD{0}{2|x_i|}[A,W]<\beta/10,$$ where $\Theta^{\cG(2,1)}_W(w,|x_i|/2)=0$ because $W\cap B(w,|x_i|/2)$ is a line segment. All together, $$\Theta_A^{\cG(2,1)}(x_i,|x_i|/16)\leq \beta/5 + 4(\beta/10) <\beta\quad\text{for all }i\geq i_0,$$ as claimed. Thus, $\Theta_A^{\cG(2,1)}(x_i,sr_i)< \gamma$ for all $i\geq i_0$ by Lemma \ref{l:Tpdp}, because $sr_i\leq |x_i|/16$ when $i\geq i_0$. Letting $\gamma\rightarrow 0$, we obtain $\lim_{i\rightarrow\infty} \Theta^{\cG(2,1)}_A(x_i,sr_i)=0$ for all $s>0$. Therefore, $T\in \cG(2,1)$ for all $T\in \uPsTan(\cl{A},0)$. As noted above, this completes the proof. \end{proof}

For any nonempty set $A\subseteq\RR^n$, let $$\PsTan(A)=\bigcup_{x\in A}\PsTan(\cl{A},x)$$ denote the \emph{pseudotangent sets of $A$} and let $$\Tan(A)=\bigcup_{x\in A}\Tan(\cl{A},x)$$ denote the \emph{tangent sets of $A$} (see Remark \ref{r:tan-closed}).
Theorem \ref{t:fan} says that sets $A\subseteq\RR^2$ that are locally well approximated by $\cH(2,2)$ exhibit a special phenomenon---\emph{every pseudotangent set of $A$ is the translate of a tangent set of $A$}. It would be interesting to know when this phenomenon holds in general.

\begin{problem} Classify nonempty sets $A\subseteq\RR^n$ that are locally bilaterally well approximated by translates of tangent sets of $A$.\end{problem}

\subsection{Additional applications} We expect the results of \S\S6--8 above will be applicable in a variety of additional situations where complete or at least partial information about minimizers can be established. For example, even without obtaining a full description, David \cite{David2dmin1} proves that if $A\subseteq\RR^n$ is a 2-dimensional Almgren minimal cone, then $A\cap \partial \ball(0,1)$ is composed of a finite number of great circles and arcs of great circles $C_j$ satisfying the following constraints:
\begin{enumerate}[(i)]
	\item if $e$ is the endpoint of some arc $C_j$, then there are exactly three arcs which meet at mutual 120$^\circ$ angles at $e$;
	\item the length of each $C_j$ is bounded below by some minimal length $\ell_0(n)>0$; and,
	\item if $x\in C_j$ and $y\in C_k$ satisfy $|x-y|\leq \eta_0(n)$, then $C_j$ and $C_k$ are arcs of great circles sharing a common endpoint in $\ball(x, |x-y|)$.
\end{enumerate} From the existence of $\ell_0$ and $\eta_0$, one can show (just as in the case $n=3$) that $\cG=\cG(n,2)$ points and $\cG\cup \cY$ points are detectable in the class $\cM(n,2)$ of translates of 2-dimensional Almgren minimal cones in $\RR^n$ (compare to Example \ref{ex:m32}). Research on the classification of $m$-dimensional Almgren minimal cones in $\RR^n$ is still ongoing, but see Liang \cites{Liang,Liang2} and Luu \cite{Luu13} for recent progress on 2- and 3-dimensional minimal cones in $\RR^4$.

\subsection{Directions for future research} There are many avenues for continued research. Here we restrict our discussion to just a few problems. One important, open line of inquiry is to decide when parameterization theorems hold for sets that are locally approximable by a class $\cS$. Ideally one would like to construct parameterizations at various grades of regularity.

\begin{problem}\label{p:Reif} Classify local approximation classes $\cS$ with the property that closed sets $A$ that are locally bilaterally $\varepsilon$-approximable by $\cS$ ($0<\varepsilon\leq \varepsilon_\cS$) admit local Reifenberg type (tame, $C^{0,\alpha}$) parameterizations by open subsets of sets in $\cS$; cf.~ \cites{Reifenberg,DDT}. \end{problem}

\begin{problem} Develop higher regularity (e.g.~ $C^{0,1}$, $C^{1,\alpha}$, or $C^\infty$) local parameterization theorems for sets that are locally well approximated by classes $\cS$ from Problem \ref{p:Reif}; cf.~ \cites{Toro,DT} ($C^{0,1}$ regularity) and \cites{DKT,PTT,Lewis} ($C^{1,\alpha}$ regularity).\end{problem}

A second frontier for future research is to study local set approximation in additional geometric settings. As noted in Remark \ref{r:geo1}, we expect that the main results of this paper remain valid in the Heisenberg group and other self-similar geometries. However, in principle one may also examine Reifenberg type sets in a sphere $S^n$ or in a Riemannian (or sub-Riemannian) manifold whose metric tangents are Euclidean (Carnot).

\begin{problem} \label{p:geo} Develop a theory of local set approximation (tangents, approximability\dots) in Riemannian manifolds, sub-Riemannian manifolds, or other geometric settings.\end{problem}

\bibliography{lsa-refs}{}
\bibliographystyle{amsalpha}
\end{document}